%% file: tighter_hal.tex
\documentclass[11pt]{article}
\usepackage[dvips]{graphicx}
\usepackage{amssymb,amsmath,amsthm,color}
\usepackage{url}
\usepackage{natbib}
\bibpunct{(}{)}{;}{a}{,}{,}
 \usepackage[colorlinks,
            linkcolor=blue,
            citecolor=blue,
            urlcolor=magenta,
            linktocpage,
            plainpages=false]{hyperref}

 \usepackage[mathcal]{eucal}

\renewcommand\epsilon\varepsilon % Graphie du symbole "epsilon".

\newcommand\eps{\epsilon} % Ensemble des entiers naturels.

\newcommand\E{\mathbb E} % Ensemble des entiers naturels.
\newcommand\EE{\mathbb E} % Ensemble des entiers naturels.

\newcommand\NN{\mathbb{N}} % Ensemble des entiers naturels.
\newcommand\RR{\mathbb{R}} % Ensemble des nombres rÃ©els.
\renewcommand\H{\mathcal{H}}

\newcommand{\BEAS}{\begin{eqnarray*}}
\newcommand{\EEAS}{\end{eqnarray*}}
\newcommand{\BEA}{\begin{eqnarray}}
\newcommand{\EEA}{\end{eqnarray}}
\newcommand{\BEQ}{\begin{equation}}
\newcommand{\EEQ}{\end{equation}}
\newcommand{\BIT}{\begin{itemize}}
\newcommand{\EIT}{\end{itemize}}
\newcommand{\BNUM}{\begin{enumerate}}
\newcommand{\ENUM}{\end{enumerate}}
\newcommand{\BA}{\begin{array}}
\newcommand{\EA}{\end{array}}

\newcommand{\tr}{\mathop{ \rm tr}}

\newcommand{\idm}{I}
\newcommand{\rb}{\mathbb{R}}

\newcommand{\eq}[1]{Eq.~(\ref{eq:#1})}
\newcommand{\myfig}[1]{Figure~\ref{fig:#1}}

\def \E{{\mathbb E}}

\def \E{{\mathbb E}}

\def \F{{\mathcal F}}

\newtheorem{lemma}{Lemma}					        %Lemma
\newtheorem{theorem}{Theorem}					        %Lemma
\newtheorem{corollary}{Corollary}					        %Lemma
\newtheorem{proposition}{Proposition}					        %Lemma

\makeatletter
\def\blfootnote{\xdef\@thefnmark{}\@footnotetext}
\makeatother

\input{preambule}

\title{Harder, Better, Faster, Stronger Convergence Rates \\ for Least-Squares Regression}

\author{
Aymeric Dieuleveut$^*$, Nicolas Flammarion$^*$ and Francis Bach\\
INRIA - Sierra project-team\\
D\'epartement d'Informatique de l'Ecole Normale Sup\'erieure \\
Paris, France \\
\texttt{aymeric.dieuleveut@ens.fr}, \texttt{nicolas.flammarion@ens.fr}, \texttt{francis.bach@ens.fr} 
}
\begin{document}

\maketitle

\begin{abstract}
We consider the optimization of a quadratic objective function whose gradients are only accessible through a stochastic oracle that returns the gradient at any given point plus a zero-mean finite variance random error. We present the first algorithm that achieves jointly the optimal prediction error rates for least-squares regression, both in terms of forgetting of initial conditions in $O(1/n^2)$, and in terms of dependence on the noise and dimension $d$ of the problem, as $O(d/n)$. Our new algorithm is based on averaged accelerated regularized gradient descent, and may also be analyzed through finer assumptions on initial conditions and the Hessian matrix, leading to dimension-free quantities that may still be small while the ``optimal'' terms above are large. In order to characterize the tightness of these new bounds, we consider an application to non-parametric regression and use the known lower bounds on the statistical performance (without computational limits), which happen to match our bounds obtained from a single pass on the data and thus show optimality of our algorithm in a wide variety of particular trade-offs between bias and variance.
\end{abstract}

\blfootnote{$^*$Both authors contributed equally.}

\input{intro.tex}

\section{Least-Squares Regression}
\input{assump}

\input{models}

\input{related}
\input{RLS}
\input{resultagd}
\input{resultaagd}

\input{hilbertspacehyp}

\input{rkhsrate}

\input{experiment}

\input{conclusion}

\subsection*{Acknowledgements}
 The authors would like to thank Damien Garreau for helpful discussions.
 
\bibliographystyle{plainnat}
\bibliography{bib_tighter_cr.bib}
\newpage 
\appendix
\input{lemmasemisto}
\input{appagd}
\input{appaagd}

\input{proofHShyp}
\input{technical}

\end{document}

%% file: preambule.tex
\usepackage{enumitem}
\newcommand{\lecc}{\preccurlyeq}

\newcommand{\vertiii}[1]{{\left\vert\kern-0.25ex\left\vert\kern-0.25ex\left\vert #1 
    \right\vert\kern-0.25ex\right\vert\kern-0.25ex\right\vert}}

%% file: intro.tex
\section{Introduction}

Many supervised machine learning problems are naturally cast as the minimization of a smooth function defined on a Euclidean space. This includes least-squares regression, logistic regression~\citep[see, e.g.,][]{hastie} or generalized linear models~\citep{glim}. While small problems with few or low-dimensional input features may be solved precisely by many potential optimization algorithms (e.g., Newton method), large-scale problems with many high-dimensional features are typically solved with simple gradient-based iterative techniques whose per-iteration cost is small. 

In this paper, we consider a quadratic objective function $f$ whose gradients are only accessible through a stochastic oracle that returns the gradient at any given point plus a zero-mean finite variance random error. In this stochastic approximation framework~\citep{rob1951stochastic}, it is known that two quantities dictate the behavior of various algorithms, namely the covariance matrix~$V$ of the noise in the gradients, and the deviation $\theta_0 - \theta_\ast$ between the initial point of the algorithm $\theta_0$ and any of the global minimizer $\theta_\ast$ of $f$. This leads to a ``bias/variance'' decomposition~\citep{bm1,hsu2014random} of the performance of most algorithms as the sum of two terms: (a) the bias term   characterizes how fast initial conditions are forgotten and thus is increasing in a well-chosen norm of $\theta_0 - \theta_\ast$; while (b) the variance term   characterizes the effect of the noise in the gradients, independently of the starting point, and with a term that is increasing in the covariance of the noise.

For quadratic functions with (a) a noise covariance matrix $V$ which is proportional (with constant $\sigma^2$) to the Hessian of $f$ (a situation which corresponds to least-squares regression) and (b) an initial point characterized by the norm $\| \theta_0 - \theta_\ast\|^2$, the optimal bias and variance terms are known \emph{separately}. On the one hand, the optimal bias term after $n$ iterations is proportional to $\frac{L \| \theta_0 - \theta_\ast\|^2}{n^2}$, where $L$ is the largest eigenvalue of the Hessian of $f$. This rate is achieved by accelerated gradient descent~\citep{n1,nest2004}, and is known to be optimal if the number of iterations $n$ is less than the dimension $d$ of the underlying predictors, but the algorithm is not robust to random or deterministic noise in the gradients~\citep{as,MR3232608}.  On the other hand, the optimal variance term is proportional to $\frac{\sigma^2 d}{n}$~\citep{tsybakov2003optimal}; it is known to be achieved by averaged gradient descent~\citep{bm1}, which for the bias term only achieves $\frac{L \| \theta_0 - \theta_\ast\|^2}{n}$ instead of $\frac{L \| \theta_0 - \theta_\ast\|^2}{n^2}$.

Our first contribution in this paper is to present a novel algorithm which attains optimal rates for \emph{both the variance and the bias terms}. This algorithm analyzed in Section~\ref{sec_result_acc} is averaged accelerated gradient descent; beyond obtaining jointly optimal rates, our result shows that averaging is beneficial for accelerated techniques and provides a provable robustness to noise.

While optimal when measuring performance in terms of the  dimension $d$ and the initial distance to optimum $\| \theta_0 - \theta_\ast\|^2$, these rates are not adapted in many situations where either $d$ is larger than the number of iterations $n$ (i.e., the number of observations for regular stochastic gradient descent) or $L \| \theta_0 - \theta_\ast\|^2$ is much larger than $n^2$. Our second contribution is to provide in Section~\ref{sec:tighter} an analysis of a new algorithm (based on some additional regularization) that can adapt our bounds to finer assumptions on $\theta_0 - \theta_\ast$ and the Hessian of the problem, leading in particular to dimension-free quantities that can thus be extended to the Hilbert space setting (in particular for non-parametric estimation).

 In order to characterize the optimality of these new bounds, our third contribution is to  consider an application to non-parametric regression in Section~\ref{sec:kernelreg} and use the known lower bounds on the statistical performance (without computational limits), which happen to match our bounds obtained from a single pass on the data and thus show optimality of our algorithm in a wide variety of particular trade-offs between bias and variance.

Our paper is organized as follows: in Section~\ref{sec_assumptions}, we present the main problem we tackle, namely least-squares regression; then, in Section~\ref{sec_result_av}, we present new results for averaged stochastic gradient descent that set the stage for Section~\ref{sec_result_acc}, where we present our main novel result leading to an accelerated algorithm which is robust to noise. Our tighter analysis of convergence rates based on finer dimension-free quantities is presented in Section~\ref{sec:tighter}, and their optimality for kernel-based non-parametric regression is studied in Section~\ref{sec:kernelreg}.

%% file: assump.tex
\label{sec_assumptions}

In this section, we present our least-squares regression framework, which is risk minimization with the square loss, together with the main assumptions regarding our model and our algorithms. These algorithms will rely on stochastic gradient oracles, which will come in two kinds, an additive noise which does not depend on the current iterate, which will correspond in practice to the full knowledge of the covariance matrix, and a ``multiplicative/additive'' noise, which corresponds to the regular stochastic gradient obtained from a single pair of observations. This second oracle is much harder to analyze.

\subsection{Statistical Assumptions} 

We make the following general assumptions:

\begin{itemize}
 \item $\mathcal{H}$ is a $d$-dimensional Euclidean space with $d\geq1$. The (temporary) restriction to finite dimension will be relaxed in Section~\ref{sec:kernelreg}.
 
 \item The observations $(x_n,y_n)\in\mathcal{H}\times\RR$, $n \geq 1$, are independent and identically distributed (i.i.d.), and such that $\EE \Vert x_n\Vert^2$ and $\EE  y_n^2$ are finite.
 \item We consider the \emph{least-squares regression} problem which is the minimization of the function $f(\theta)=\frac{1}{2}\E (\langle x_n,\theta\rangle-y_n)^2$.
\end{itemize}

\paragraph{Covariance matrix.}
We denote by  $\Sigma=\EE( x_n \otimes x_n) \in \rb^{d \times d}$ the population covariance matrix, which is the Hessian of $f$ at all points. Without loss of generality, we can assume $\Sigma$ invertible by reducing $\mathcal{H}$ to the minimal subspace where all $x_n$, $n \geq 1$, lie almost surely. This implies that all eigenvalues of $\Sigma$ are strictly positive (but they may be arbitrarily small). Following~\citet{bm1},
we assume there exists $R>0$ such that 
\begin{equation}\tag{$\mathcal A_1$}\label{eq:rr}
\EE \Vert x_n\Vert^2 x_n\otimes x_n \preccurlyeq R^2\Sigma,
\end{equation}
where $A \preccurlyeq B$ means that $B-A$ is positive semi-definite.
This assumption implies in particular that (a) $\EE \| x_n \|^4$ is finite and (b) $\tr \Sigma =\EE \Vert x_n\Vert^2 \leq R^2$ since taking the trace of the previous inequality we get $\EE\Vert x_n\Vert^4\leq R^2\EE\Vert x_n\Vert^2$ and using  Cauchy-Schwarz inequality we get $\EE\Vert x_n\Vert^2\leq \sqrt{\EE\Vert x_n\Vert^4}\leq R\sqrt{\EE\Vert x_n\Vert^2}$. 

Assumption (\ref{eq:rr})  is satisfied, for example,  for least-square regression with almost surely bounded data, since $\Vert x_n\Vert^2\leq R^2$ almost surely implies  $\EE \Vert x_n\Vert^2 x_n\otimes x_n \preccurlyeq \EE \big[ R^2 x_n\otimes x_n \big]= R^2\Sigma$. This assumption is also true for data with infinite support   and a bounded \emph{kurtosis} for the projection of the covariates $x_n$ on any direction $z\in\mathcal{H}$, e.g, for which there exists $\kappa>0$, such that:
%such as Gaussians or mixtures of Gaussians (where all covariance matrices of the mixture components are lower and upper bounded by a constant times the same matrix).
\begin{equation}\label{eq:kurtosis}\tag{$\mathcal A_2$}
\forall z\in \mathcal{H}, \ \ \EE\langle z,x_n\rangle^4 \leq \kappa \langle z,\Sigma z\rangle^2.
\end{equation}
Indeed, by Cauchy-Schwarz inequality, Assumption (\ref{eq:kurtosis}) implies for all $(z,t)\in\mathcal{H}^2$,  the following bound $\EE \langle z,x_n\rangle^2 \langle t,x_n\rangle^2\leq \kappa \langle z, \Sigma z\rangle \langle t, \Sigma t\rangle$, which in turn implies that for all positive semi-definite symmetric matrices $M,N$, we have $\EE \langle x_n, M x_n\rangle \langle x_n, N x_n\rangle \leq\kappa \tr (M\Sigma)\tr(N\Sigma).$ Assumption (\ref{eq:kurtosis}), which is true for Gaussian vectors with $\kappa=3$, thus implies (\ref{eq:rr}) for $R^2=\kappa \tr \Sigma=\kappa \E \Vert x_{n}\Vert^{2}$.

\paragraph{Eigenvalue decay.} Most convergence bounds depend on the dimension $d$ of $\mathcal{H}$. However it is possible to derive dimension-free and often tighter convergence rates by considering bounds depending on the value  $ \tr \Sigma^b$ for $b \in[0,1]$. Given $b$, if we consider the eigenvalues of $\Sigma$ ordered in decreasing order, which we denote by $s_i$, they decay at least as $\frac{(\tr \Sigma^b )^{1/b}}{i^{1/b}}$.  Moreover, it is known that $(\tr \Sigma^b)^{1/b}$ is decreasing in $b$ and thus, the smaller the $b$, the stronger the assumption.  
For $b$ going to $0$ then $\tr \Sigma^b$ tends to $d$ and we are back in the classical low-dimensional case. When $b=1$, we simply get $\tr \Sigma = \E \| x_n\|^2$, which will correspond to  the weakest assumption in our context.

\paragraph{Optimal predictor.}
In finite dimension the regression function $f(\theta)=\frac{1}{2}\E (\langle x_n,\theta\rangle-y_n)^2$ always admits a global minimum $\theta_*=\Sigma^{-1}\EE (y_n x_n)$. When initializing algorithms at $\theta_0=0$ or regularizing by the squared norm, rates of convergence generally depend on $\Vert \theta_*\Vert$, a quantity which could be   arbitrarily large. 

However  there exists a systematic upper-bound\footnote{Indeed for all $\theta \in \rb^d$ and in particular $\theta=0$, by Minkowski inequality,  $\Vert \Sigma^{\frac{1}{2}}\theta_*\Vert-\sqrt{\EE  y_n^2}=\sqrt{\E  \langle \theta_*,x_n\rangle ^2}-\sqrt{\EE  y_n^2} \leq  \sqrt{\EE( \langle\theta_*,x_n\rangle-y_n)^2} \leq  \sqrt{\EE( \langle\theta,x_n\rangle-y_n)^2}\leq  \sqrt{\EE(y_n)^2}$.}  $\Vert \Sigma^{\frac{1}{2}}\theta_*\Vert\leq 2\sqrt{\EE  y_n^2}$.
This leads naturally to the consideration of convergence bounds depending on $\Vert \Sigma^{r/2} \theta_*\Vert$    for $r\leq 1$. In infinite dimension this will correspond to assuming $\Vert \Sigma^{r/2} \theta_*\Vert<\infty$.
This new assumption relates the optimal predictor with sources of ill-conditioning (since $\Sigma$ is the Hessian of the objective function $f$), the smaller $r$, the stronger our assumption, with $r=1$ corresponding to no assumption at all, $r=0$ to $\theta_*$ in $\mathcal{H}$ and  $r=-1$ to a convergence of the bias of least-squares regression with averaged stochastic gradient descent in $O\big(\frac{\Vert \Sigma^{-1/2}\theta_{*}\Vert^2}{n^{2}}\big)$\citep{dieuleveut2015parametric,defossez2014constant}.  In this paper, we will use arbitrary initial points $\theta_0$ and thus our bounds will depend on $\Vert \Sigma^{r/2} (\theta_0-\theta_*)\Vert$.

\paragraph{Noise.}
We denote by $\eps_n=y_n-\langle\theta_*,x_n\rangle$ the residual for which we have $\EE[\eps_n x_n]=0$. Although we do not have  $\EE [\eps_n  \vert x_n]=0$ in general unless the model is well-specified, we  assume  the noise to be a structured  process such that there exists $\sigma>0$ with
\begin{equation}\label{eq:a2}\tag{$\mathcal A_3$}
 \EE [\eps_n^2 x_n \otimes x_n]\preccurlyeq \sigma^2 \Sigma.
 % \note{ $\sigma$ variance of regression noise, $\tau$ variance of the additive noise}
\end{equation}
Assumption  (\ref{eq:a2}) is satisfied for example for data almost surely bounded or when the model is well-specified, (e.g., $y_n=\langle \theta_*,x_n\rangle +\epsilon_n$, with $(\epsilon_n)_{n\in\NN}$ i.i.d. of variance $\sigma^2$ and independent of $x_n$).

%% file: models.tex
\subsection{Averaged Gradient Methods and Acceleration}\label{sec:gradient}
 
%\note{FB: mention that they come from regularizing by $\frac{\lambda}{2} \| \theta - \theta_0\|^2$}
%\note{FB: do two separate paragraphs to describe the two different algorithms. Too hard to read as is. Mention heavy ball for the second one}

We focus in this paper on  stochastic gradient methods with and without acceleration for a quadratic function regularized by $\frac{\lambda}{2} \| \theta - \theta_0\|^2$. Stochastic gradient descent (referred
to from now on as ``SGD'')  can be described for $n\geq 1$ as
\BEQ\label{eq:gd}
\theta_n=\theta_{n-1}-\gamma f'_n(\theta_{n-1})-\gamma \lambda (\theta_{n-1}-\theta_{0}),
\EEQ
starting from $\theta_0\in\mathcal H $, where $\gamma\in\RR$ is either called the step-size in optimization or the learning rate in machine learning, and $f'_n(\theta_{n-1})$ is an unbiased estimate of the gradient of $f$ at $\theta_{n-1}$, that is such that its conditional expectation  given all other sources of randomness is equal to $f'(\theta_{n-1})$.

Accelerated stochastic gradient descent is defined by  an iterative system with two parameters $(\theta_n,\nu_n)$ satisfying for $n\geq1$
\BEA
\theta_n&=&\nu_{n-1}-\gamma f'_n(\nu_{n-1})-\gamma \lambda (\nu_{n-1}-\theta_{0}) \nonumber \\
\nu_n&=&\theta_n+\delta\big( \theta_n-\theta_{n-1}\big) \label{eq:agd},
\EEA
starting from $\theta_0=\nu_{0}\in\mathcal H $, with $\gamma,\delta\in\RR^{2}$ and $f'_n(\theta_{n-1})$ described as before.  It may be reformulated as the following second-order recursion 
%\note{FB: make the $\delta$ term appear alone to make it clear it is an additional momentum term}
$$\theta_n=(1-\gamma\lambda)\big(\theta_{n-1} + \delta (\theta_{n-1}-\theta_{n-2})\big) -\gamma f'_n\big(\theta_{n-1}+\delta ( \theta_{n-1}-\theta_{n-2})\big)+\gamma \lambda\theta_{0}.
$$

The \emph{momentum} coefficient $\delta\in\RR$ is chosen to accelerate the convergence rate \citep{n1,bt} and has its roots in the heavy-ball algorithm from \citet{Polyak1964}.
We especially concentrate here, following \citet{pj}, on the average of the sequence 
\BEQ\label{eq:average}
\bar \theta_n= \frac{1}{n+1} \sum_{i=0}^n \theta_n,
\EEQ
and we note that it can be computed online as $\bar \theta_n=\frac{n}{n+1} \bar \theta_{n-1}+\frac{1}{n+1}\theta_n$.

The key ingredient in the algorithms presented above is the unbiased estimate on the gradient $f_n'(\theta)$, which we now describe.

\subsection{Stochastic Oracles on the Gradient}
We consider the standard stochastic approximation framework \citep{kuku}. That is, we let $(\mathcal{F}_n)_{n\geq0}$ be the increasing family of $\sigma$-fields that are generated by all variables $(x_i,y_i)$ for $i \leq n$, and such that for each $\theta\in\mathcal{H}$ the random variable $f'_n(\theta)$ is square-integrable and $\mathcal{F}_n$-measurable with 
$\E[f'_n(\theta)\vert \mathcal{F}_{n-1}]=f'(\theta),$ for all $n\geq0$. We will consider two different gradient oracles.

\paragraph{Additive noise.} The first oracle is the sum of the true gradient $f'(\theta)$ and an uncorrelated zero-mean noise that does not depend on $\theta$. Consequently it is of the form
\BEQ\tag{$\mathcal{A}_4$}\label{eq:a4}
f'_n(\theta)=f'(\theta)-\xi_n,
\EEQ
where the noise process $\xi_n$ is $\mathcal{F}_n$-measurable with $\E[\xi_n\vert\mathcal{F}_{n-1}]=0$ and $\E [\Vert \xi_n\Vert^2]$ is finite. Furthermore we also assume that there exists $\tau\in\RR$ such that
\BEQ
\tag{$\mathcal{A}_5$}\label{eq:a5new}
\E[\xi_n\otimes\xi_n ]\preccurlyeq \tau^2 \Sigma ,
\EEQ
that is, the noise has a particular structure adapted to least-squares regression. For optimal results for unstructured noise, with convergence rate for the noise part in $O(1/\sqrt{n})$, see~\citet{lan}.
Our oracle above with an additive noise which is independent of the current iterate corresponds to the first setting studied in stochastic approximation \citep{rob1951stochastic,dudu,pj}. While used by~\citet{bm1} as an artifact of proof, for least-squares regression, such an additive noise   corresponds to the situation where the distribution of $x$ is known so that the population covariance matrix is computable, but the distribution of the outputs $(y_n)_{n\in\mathbb{N}}$ remains unknown  and thus may be related to regression estimation with fixed design \citep{gyorfi2006distribution}. This oracle is equal to
\BEQ\label{eq:sigmanown}
f'_n(\theta)=\Sigma\theta-y_n x_n.
\EEQ
and has thus a noise vector $\xi_n=y_nx_n-\EE y_n x_n$ independent of $\theta$. 
Assumption (\ref{eq:a5new}) will be satisfied, for example if the outputs are almost surely bounded  because  $\EE [\xi_n \otimes \xi_n]\preccurlyeq \EE [y_n^2 x_n\otimes x_n]\preccurlyeq \tau^2 \Sigma$ if $y_n^2\leq \tau^2$ almost surely. But it will also be for data satisfying Assumption (\ref{eq:kurtosis}) since we will have
\BEAS
\EE [\xi_n \otimes \xi_n]&\preccurlyeq& \EE [y_n^2 x_n \otimes x_n] = \EE [(\langle \theta_*,x_n\rangle +\epsilon_n)^2 x_n \otimes x_n] \\
&\preccurlyeq&2\EE [\langle \theta_*,x_n\rangle^2x_n\otimes x_n]+2\sigma^2 \Sigma \preccurlyeq 2(\kappa \Vert \Sigma^{1/2} \theta_*\Vert^2 +\sigma^2)\Sigma \preccurlyeq 2(4\kappa \E[y_{n}^{2}] +\sigma^2)\Sigma,
\EEAS
and thus Assumption (\ref{eq:a4}) is satisfied with $\tau^2=2(4\kappa \E[y_{n}^{2}] +\sigma^2)$. 
%and the $\Sigma$-known model   satisfies assumption (\ref{eq:a4}) with $\tau^2=2(\kappa \Vert \Sigma^{1/2} \theta_*\Vert^2 +\sigma^2)$. Thus it appears that this model mixes the variance $\sigma^2$ and the bias   \note{FB: explain better!} but only under the small value $\Vert \Sigma^{1/2} \theta_*\Vert^2$.  Such a model can be implemented and performance will be the same as  with the ``semi-stochastic'' one. 

\paragraph{Stochastic noise (``multiplicative/additive'').} This corresponds to:
\BEQ\label{eq:sto}
f'_n(\theta)⁼=(\langle x_n,\theta\rangle -y_n)x_n=(\Sigma+\zeta_n)(\theta-\theta_*)-\xi_n,
\EEQ
with $\zeta_n=x_n\otimes x_n-\Sigma$ and $\xi_n=(y_n-\langle x_n,\theta_*\rangle )x_n=\eps_nx_n$. This oracle corresponds to regular SGD, which is often referred to as the least-mean-square (LMS) algorithm for least-squares regression, where the noise comes from sampling a single pair of observations. It combines the additive noise $\xi_n$ of Assumption (\ref{eq:a4}) and a multiplicative noise $\zeta_n$. 
 This multiplicative noise makes this stochastic oracle harder to analyze which explains it is often  approximated by an additive noise oracle. However it is the most widely used and most practical one. 
 Note that for the oracle in \eq{sto}, from Assumption  (\ref{eq:a2}), we have $\E[\xi_n\otimes\xi_n ]\preccurlyeq \sigma^2 \Sigma$; it  has a similar form to Assumption  (\ref{eq:a5new}), which is valid for the additive noise oracle from Assumption  (\ref{eq:a4}). We use different constants $\sigma^2$ and $\tau^2$ to highlight the difference between these two oracles.

%% file: resultagd.tex
\section{Averaged Stochastic Gradient Descent } \label{sec_result_av}

In this section, we provide convergence bounds for regularized averaged stochastic gradient descent. The main novelty compared to the work of \citet{bm1} is (a) the presence of regularization, which will be useful when deriving tighter convergence rates in Section~\ref{sec:tighter} and (b) a simpler more direct proof. We first consider the additive noise in Section~\ref{sec:add} before considering the multiplicative/additive noise in Section~\ref{sec:mult}.

\subsection{Additive Noise}
\label{sec:add}
We study here the convergence of the averaged SGD recursion defined by \eq{gd}  under the simple oracle from Assumption (\ref{eq:a4}). For least-squares regression, it takes the form: %\note{FB: show the actual additive oracle here, and not the formulation used for the proofs}
%\BEQ \label{eq:lms}
% \theta_n  - \theta_\ast  =  \big[ \idm - \gamma x_n\otimes x_n - \gamma \lambda \idm \big] ( \theta_{n-1}  - \theta_\ast) 
%+ \gamma \varepsilon_n x_n
%+ \lambda \gamma (\theta_{0} - \theta_\ast).
% \EEQ
 \BEQ \label{eq:lmssigma}
 \theta_n   =  \big[ \idm - \gamma \Sigma - \gamma \lambda \idm \big] \theta_{n-1}  
+ \gamma  y_n x_n
+ \lambda \gamma \theta_{0}. 
 \EEQ
This is an easy adaptation of the work of \citet[Lemma 2]{bm1} for the regularized case. 
%\note{FB: I have removed the reference to semi-sto. Make sure that in the proof this is the case as well}
\begin{lemma}\label{lemma:sgdsemisto}
Assume $(\mathcal{A}_{4,5})$. 
%with $\E ( \xi_n\otimes \xi_n )=V$ 
Consider the recursion  in Eq.~\eqref{eq:lmssigma} with any regularization parameter $\lambda\in\RR_{+}$ and any constant step-size $ \gamma (\Sigma +\lambda\idm ) \preccurlyeq \idm $. Then
\begin{equation}
\label{eq:res1}
 \E f(\bar \theta_n )- f ( \theta_* ) \leq \Big(\lambda+\frac{1}{{\gamma n}}\Big)^{2} \Vert \Sigma^{1/2}(\Sigma+\lambda\idm)^{-1}(\theta_{0} - \theta_\ast)\Vert^2 +\frac{\tau^{2}\tr\big[ \Sigma^{2} (\Sigma+\lambda\idm)^{-2}\big]}{n}.
 \end{equation}
% \begin{multline}
%  E\langle \bar\eta_n,\Sigma\bar \eta_n\rangle \leq 2\min\Big\{\frac{\Vert\Sigma^{1/2}(\Sigma+\lambda\idm)^{-1/2}\eta_0\Vert^2}{n\gamma}, \frac{\Vert \Sigma^{1/2}(\Sigma+\lambda\idm)^{-1}\eta_0\Vert^2}{\gamma^2n^2}\Big\}
%  \\+2\lambda \Vert \lambda^{1/2}\Sigma^{1/2}(\Sigma+\lambda\idm)^{-1}\eta_0\Vert^2+\frac{\tr V(\Sigma+\lambda\idm)^{-1}}{n}
% \end{multline}
\end{lemma}
We can make the following observations:
\begin{itemize}
\item The proof (see Appendix~\ref{app:proofsemisto})  relies on the fact that $\theta_{n}-\theta_{*}$ is obtainable in closed form since the cost function is quadratic and thus the recursions are linear, and follows from~\citet{pj}.
\item The constraint on the step-size $\gamma$ is equivalent to $\gamma ( L + \lambda) \leqslant 1$ where $L$ is the largest eigenvalue of $\Sigma$ and we thus recover the usual step-size from deterministic gradient descent~\citep{nest2004}.

\item When $n$ tends to infinity, the algorithm converges to the minimum of $f(\theta) + \frac{\lambda}{2} \| \theta - \theta_0\|^2$ and our performance guarantee becomes $\lambda^2 \Vert \Sigma^{1/2}(\Sigma+\lambda\idm)^{-1}(\theta_{0} - \theta_\ast)\Vert^2$. This is the standard ``bias term'' from regularized ridge regression~\citep{hsu2014random} which we naturally recover here. The term $\frac{\tau^{2}}{n}\tr\big[ \Sigma^{2} (\Sigma+\lambda\idm)^{-2}\big]$ is usually referred to as the ``variance term''~\citep{hsu2014random}, and is equal to $\frac{\tau^2}{n}$ times the quantity $\tr\big[ \Sigma^{2} (\Sigma+\lambda\idm)^{-2}\big]$, which is often called the degrees of freedom of the ridge regression problem~\citep{spline}.

\item For finite $n$, the first term is the usual bias term which depends on the distance from the initial point $\theta_0$ to the objective point $\theta_*$ with an appropriate norm. It includes a regularization-based component which is function of $\lambda^2 $  and optimization-based component which depends on $(\gamma n)^{-2}$. The  regularization-based bias appears because the algorithm tends to minimize the regularized function instead of the true function $f$.

\item Given Eq.~(\ref{eq:res1}), it is natural to set   $\lambda \gamma =\frac{1}{ n}$,  and the two components of the bias term are  exactly of the same order  $\frac{4}{\gamma^2n^2} {\Vert \Sigma^{1/2}(\Sigma+\lambda\idm)^{-1}(\theta_0-\theta_{*})\Vert^2}$.
 It corresponds up to a constant factor to the bias term of  regularized least-squares~\citep{hsu2014random}, but it is achieved by an algorithm accessing only $n$ stochastic gradients. Note that here as in the rest of the paper, we only prove results in the finite horizon setting, meaning that the number of samples is known in advance and the parameters $\gamma,\lambda$ may be chosen as  functions of $n$, but remain constant along the iterations (when $\lambda$ or $\gamma$ depend on $n$, our bounds only hold for the last iterate). 
\item Note that the bias term can also be bounded by $\frac{1}{\gamma n} \Vert \Sigma^{1/2}(\Sigma+\lambda\idm)^{-1/2}(\theta_0-\theta_{*})\Vert^2 $ when only $\|\theta_0-\theta_{*}\|$ is finite. See the proof in Appendix~\ref{app:semistonbiasbis} for details. 
%\note{FB: make sure the proof discusses this explicitly in a separate section}
\item The second term is the variance term. It depends on the  noise in the gradient.
When this one is not structured the variance turns to be also bounded by $\gamma \tr\big(\Sigma (\Sigma+\lambda\idm)^{-1}\E [ \xi_{n}\otimes \xi_{n}] \big)$ (see Appendix~\ref{app:semistonotstruc}) and we recover for $\gamma=O(1/\sqrt{n})$, the usual rate of $\frac{1}{\sqrt{n}}$ for SGD in the smooth case \citep{shalev}.
\item Overall we get the same performance as the empirical risk minimizer with fixed design, but with an algorithm that performs a single pass over the data.
\item When $\lambda=0$ we recover Lemma 2 of \citet{bm1}. In this case the variance term $\frac{\tau^{2}d}{n}$ is optimal over all estimators in $\mathcal{H}$ \citep{Tsyb} even without computational limits, in the sense that no   estimator that uses the same information can improve upon this rate. 
%\note{FB: Is the lower bound true for the additive noise??}

%\item The first term is  a  regularization based term. It appears because the algorithm tends to minimize the regularized function instead of the true function $f$ .
%\item The second one is a  bias term: it depends on the distance from the initial point $\theta_0$ to the objective point $\theta_*$. 
%\note{FB: bias may be not the best choice}
%\item With $\lambda \gamma =\frac{1}{ n}$, the first two terms are  exactly of the same order  $4\frac{\Vert \Sigma^{1/2}(\Sigma+\lambda\idm)^{-1}\eta_0\Vert^2}{\gamma^2n^2}$. It corresponds to the bias term of the regularized least square. 
%\item The third term is a variance term. It depends on the  noise.
%\item We thus get the same performance of the empirical risk minimizer with fixed design, but with an algorithm that performs a single pass over the data.
\end{itemize}

\subsection{Multiplicative/Additive Noise}
\label{sec:mult}

When the general stochastic oracle in \eq{sto} is considered, the regularized LMS algorithm  defined by \eq{gd}  takes the form: %\note{FB: show the actual additive oracle here, and not the formulation used for the proofs}
%\BEQ \label{eq:lms}
% \theta_n  - \theta_\ast  =  \big[ \idm - \gamma x_n\otimes x_n - \gamma \lambda \idm \big] ( \theta_{n-1}  - \theta_\ast) 
%+ \gamma \varepsilon_n x_n
%+ \lambda \gamma (\theta_{0} - \theta_\ast).
% \EEQ
 \BEQ \label{eq:lms}
 \theta_n   =  \big[ \idm - \gamma x_n\otimes x_n - \gamma \lambda \idm \big] \theta_{n-1}  
+ \gamma  y_n x_n
+ \lambda \gamma \theta_{0}. 
 \EEQ
 We have a very similar result with an additional corrective term (second line below) compared to Lemma~\ref{lemma:sgdsemisto}.
\begin{theorem} \label{th:av_sgd} Assume $(\mathcal{A}_{1,3}$). Consider the recursion  in Eq.~\eqref{eq:lms}.  For any regularization parameter $\lambda\leq R^2/2$ and for any constant step-size  $\gamma\leq \frac{1}{2R^2}$ we have
 \BEAS
\E f(\bar \theta_n )- f ( \theta_* ) 
& \leqslant & 3\Big(2\lambda+\frac{1}{{\gamma n}}\Big)^{2} \|  \Sigma^{1/2} ( \Sigma + \lambda \idm)^{-1} (\theta_0-\theta_\ast)  \|^{2} + 
 \frac{  6 \sigma^2}{n+1}   \tr  \big[ 
\Sigma^2 ( \Sigma + \lambda \idm)^{-2} \big] \\
  & &  +3\frac{{\big\|  (\Sigma+\lambda \idm)^{-1/2} (\theta_0 - \theta_\ast) \big \|^{2}}   { \tr ( \Sigma(\Sigma+\lambda \idm)^{-1} )}}{\gamma^{2}(n+1)^{2}}. 
\EEAS
\end{theorem}

We can make the following remarks:
\begin{itemize}
\item The proof (see Appendix~\ref{app:avsgd}) relies on a bias-variance decomposition, each term being treated separately.  We adapt a proof technique from \citet{bm1} which considers the difference between the recursions in \eq{lms} and in \eq{lmssigma}. 
\item As in   Lemma \ref{lemma:sgdsemisto}, the bias term can also be bounded by $\frac{1}{\gamma n} \Vert \Sigma^{1/2}(\Sigma+\lambda\idm)^{-1/2}(\theta_0-\theta_{*})\Vert^2 $  and the variance term by $\gamma \tr[\Sigma (\Sigma+\lambda\idm)^{-1}\xi_{n}\otimes \xi_{n}]$ (see proof in Appendices~\ref{app:stonbiasbis} and~\ref{app:stonotstruc}). This is useful in particular when considering unstructured noise.
 \item The variance term is the same than in the previous case. However there is a residual term   that now appears when we go to the fully stochastic oracle (second line). This term will go to zero when $\gamma$ tends to zero and can be compared to the corrective term which also appears when \citet{hsu2014random} go from fixed to random design. Nevertheless  our bounds are  more concise than theirs, make significantly fewer assumptions and rely on an efficient single-pass algorithm.
 \item   In this setting, the step-size may not exceed $1/(2R^{2})$, whereas with an additive noise in Lemma \ref{lemma:sgdsemisto} the condition is $\gamma \le 1/(L+\lambda)$, a quantity which can be much bigger than $1/(2R^{2})$, as $L$ is the spectral radius of ${\Sigma}$ whereas $R^{2}$ is of the order of $\tr(\Sigma)$. Note that in practice, computing $L$ is as hard as computing $\theta_\ast$ so that the step-size $\gamma \propto 1/R^2$ is a good practical choice.
 \item For $\lambda=0$ we recover results from \citet{defossez2014constant} with a non-asymptotic bound  but we lose the advantage of  having an asymptotic equivalent.
\end{itemize}
%
%\begin{corollary}
%For $\lambda=0$, \note{FB: line is too long}
%\BEAS 
%\big( \E \| \Sigma^{1/2} ( \bar\theta_n - \theta_\ast) \|^2 \big)^{1/2}
%& \leqslant & 
%    \min\bigg\{\frac{\|  (\theta_0 - \theta_\ast)\big\|}{\sqrt{\gamma(n+1)}}, \frac{\| \Sigma^{-1/2} (\theta_0 - \theta_\ast)\big\|}{\gamma(n+1)} \sqrt{ 
%  1 +     \gamma   R^2   d  }\bigg\}  + 
%   \bigg(
% \frac{  \sigma^2 d }{n+1}   \bigg)^{1/2} \sqrt{ \frac{1}{1 - \gamma   R^2/2  }}.
%\EEAS
%\end{corollary}
%\note{FB: comment on what is lost when going fully stochastic}

%Note also that we get a term which is essentially the performance of the empirical risk minimizer when $\gamma$ tends to zero.

%% file: resultaagd.tex
\section{Accelerated Stochastic Averaged Gradient Descent} \label{sec_result_acc}

We study the convergence under the stochastic oracle from Assumption ($\mathcal{A}_4$) of   averaged \emph{accelerated}  stochastic gradient descent defined by \eq{agd} which can be rewritten for the quadratic function~$f$ as a second-order iterative system with constant coefficients:
 \begin{equation}\label{eq:accgrad}
  \theta_n   =  \big[ \idm - \gamma \Sigma - \gamma \lambda \idm \big] \big[\theta_{n-1} +\delta (\theta_{n-1}-\theta_{n-2})\Big] 
+ \gamma  y_n x_n
+ \gamma \lambda \theta_{0}.
\end{equation}

\begin{theorem}\label{th:acc_sgd}
Assume ($\mathcal{A}_{4,5}$). For any regularization parameter $\lambda\in\RR_{+}$ and for any constant step-size $ \gamma (\Sigma +\lambda\idm ) \preccurlyeq \idm $, we have  for any $\delta\in\big[\frac{1-\sqrt{\gamma \lambda}}{1+\sqrt{\gamma \lambda}},1\big]$, for the recursion in Eq.~(\ref{eq:accgrad}):
$$
\E f(\bar \theta_n )- f ( \theta_* )  \le  2 \Big( \lambda+ \frac{36}{\gamma (n+1)^2}\Big)  \| \Sigma^{1/2}(\Sigma+\lambda I)^{-1/2}(\theta_0-\theta_*)\|^{2}  + 8\tau^{2}\frac{\tr\big[\Sigma^{2}(\Sigma+\lambda I)^{-2}\big]}{n+1 }.
$$
% and 
%\BEAS
%\E \Vert \bar \theta_n - \theta_* \Vert^{2}&\le& K_{1}\| \theta_0-\theta_*\|^2  + 2    \| \lambda  (\Sigma+\lambda I)^{-1} (\theta_{0}-\theta_{*})\|^{2} \\
%&& + \min\Big\{36\frac{\tr(\Sigma(\Sigma+\lambda I)^{-1}V(\Sigma+\lambda I)^{-1})}{n+1 }; \frac{\gamma (n+1)\tr(\Sigma(\Sigma+\lambda I)^{-1}V)}{2} \Big\}
%\EEAS
%With $K_1 = 2.75^{2}, K_2=(3+3\rho^{k})^{2} \le 36$. 
\end{theorem}
The numerical constants are partially artifacts of the proof (see Appendices~\ref{app:agd} and~\ref{app:techlem}). Thanks to a wise use of tight inequalities, the bound  is independent of $\delta$ and valid for all $\lambda \in \rb_+$. This results in the simple following corollary for $\lambda=0 $, which corresponds to the particularly simple recursion (with averaging):  
%\note{FB: give the recursion with the stochastic gradient here to show the simplicity}
 \begin{equation}\label{eq:accgradwithout}
  \theta_n  =  \big[ \idm - \gamma \Sigma \big] (2\theta_{n-1}-\theta_{n-2})
+ \gamma y_n x_n.
\end{equation}
 
\begin{corollary} \label{Cor_main}
Assume ($\mathcal{A}_{4,5}$). For any constant step-size $ \gamma \Sigma \preccurlyeq \idm $, we have for $\delta=1 $, 
\BEQ
\E f(\bar \theta_n )- f ( \theta_* ) \le 36 \frac{\|\theta_0-\theta_*\|^{2}}{\gamma (n+1)^2}  +  8\frac{\tau^{2}d }{n+1}.
\EEQ
\end{corollary}

We can make the following observations: 
\begin{itemize}

\item The proof technique relies on direct moment computations in each eigensubspace obtained by \citet{o2013adaptive} in the deterministic case. Indeed as $\Sigma$ is a symmetric matrix, the space can be decomposed on an orthonormal eigenbasis of $\Sigma$, and the iterations are decoupled in such an eigenbasis. Although we only provide an upper-bound, this is in fact an equality plus other exponentially small terms as shown in the proof which relies on linear algebra, with difficulties arising from the fact that this second-order system can be expressed as a  linear stochastic dynamical system with non-symmetric matrices.  We only provide a result for additive noise.

%\item The result above focuses on function values and does not lead to  bounds on the distance to  optimum however $\Vert \bar \theta_{n}-\theta_{0}\Vert$  seems empirically to remain bounded without converging to zero.

 \item The first bound $\frac{1}{\gamma n^2} \Vert \theta_0-\theta_*\Vert^2$ corresponds to the usual accelerated rate. It has been shown by \citet{nest2004} to be the optimal rate of convergence for optimizing a quadratic  function with a first-order method that can access only to sequences of gradients when $n\le d$. We recover by averaging an algorithm dedicated to strongly-convex function the traditional convergence rate for non-strongly convex functions. Even if it seems surprising, the algorithm works also for $\lambda=0$ and $\delta=1$ (see also simulations in Section~\ref{experiment}).

\item The second bound  also matches the optimal statistical performance $\frac{\tau^{2}d}{n}$ described in the observations following Lemma \ref{lemma:sgdsemisto}. Accordingly this algorithm achieves joint bias/variance optimality (when measured in terms of $\tau^2$ and $\|\theta_0 - \theta_\ast\|^2$).

 \item We have the same rate  of convergence for the bias than the regular Nesterov acceleration without averaging studied by \citet{fb}, which corresponds  to choosing  $\delta_{n}=1-2/n$ for all $n$. However if the problem is $\mu$-strongly convex, this latter was shown to also converge at the linear rate $O\big((1-\gamma \mu)^{n}\big)$ and thus is adaptive to hidden strong-convexity (since the algorithm does not need to know $\mu$ to run).  This explains that it ends up converging faster  for quadratic function since  for large $n$ the convergence at rate $1/n^{2}$ becomes slower  than the one at rate  $(1-\gamma \mu)^{n}$ even for very small $\mu$. This is confirmed in our experiments in  Section~\ref{experiment}.  Thanks to this adaptivity, we can also show using the same tools and considering its weighted average $\tilde \theta_{n}=\frac{2}{n(n+1)}\sum_{k=0}^{n}k\theta_{k}$ that the bias term of $\E f(\tilde \theta_n )- f ( \theta_* ) $ has a convergence rate of order  $\big( \lambda^{2}+ \frac{1}{\gamma^{2} (n+1)^4}\big)  \| \Sigma^{1/2}(\Sigma+\lambda I)^{-1}(\theta_0-\theta_*)\|^{2}$ without any change in the variance term. This has to be compared to the bias of averaged SGD  $\big( \lambda+ \frac{1}{\gamma (n+1)^2}\big)  \| \Sigma^{1/2}(\Sigma+\lambda I)^{-1}(\theta_0-\theta_*)\|^{2}$ in Section~\ref{sec_result_av} and may lead to faster convergence for the bias in presence of hidden strong-convexity.
  
  %\item \note{FB: mention explicit comparison with acceleration in the non-strongly convex case}
%\item \note{FB: discuss link with conjugate gradient}

 \item  Overall, the bias term is improved whereas the variance term is not degraded and acceleration is thus robust to noise in the gradients. Thereby, while second-order methods for optimizing quadratic functions in the singular case, such as conjugate gradient \citep[Section 6.1]{MR1099605}  are notoriously highly sensitive to noise, we are able to propose a version which is robust to stochastic noise.
 
 \item Note that when there is no assumption on the covariance of the noise we still  have the variance bounded by $ \frac{\gamma n}{2} \tr\big[\Sigma(\Sigma+\lambda I)^{-1}V\big]  $; setting  $\gamma=1/n^{3/2}$ and $\lambda=0$ leads to the   bound $ \frac{\|\theta_0-\theta_*\|^{2}}{\sqrt{n}}  + \frac{\tr V}{\sqrt{n}}$. We recover the usual rate for accelerated stochastic gradient in the non-strongly-convex case \citep{xiao}.  When the value of the bias and the variance are known, we can achieve the optimal trade-off of  \citet{lan}  $\frac{R^2\|\theta_0-\theta_*\|^2}{n^2}+\frac{\|\theta_0-\theta_*\|\sqrt{\tr V}}{\sqrt{n}}$ for $\gamma=\min\Big\{1/R^2, \frac{\|\theta_0-\theta_*\|}{\sqrt{\tr V}n^{3/2}}\Big\}$.

\end{itemize}

%% file: hilbertspacehyp.tex
\section{Tighter Convergence Rates}
\label{sec:tighter}

We have seen in Corollary~\ref{Cor_main} above that the averaged accelerated gradient algorithm matches the lower bounds $\tau^2 d / n$ and $\frac{L}{n^2} \| \theta_0 - \theta_\ast\|^2$ for the prediction error. However the algorithm performs better in almost all cases except the worst-case scenarios corresponding to the lower bounds. For example the algorithm may still predict well when the dimension $d$ is much bigger than $n$. Similarly the norm of the optimal predictor $\|\theta_{*}\|^{2}$ may be huge and the prediction still good, as gradients algorithms happen to be adaptive to the difficulty of the problem. In this section, we provide such a theoretical guarantee.

 The following bound stands for the averaged \emph{accelerated} algorithm. It extends previously known bounds in the kernel least-mean-squares setting \citep{dieuleveut2015parametric}.
 
\begin{theorem}\label{th:tighter_acc}
Assume ($\mathcal{A}_{4,5}$); for any regularization parameter $\lambda\in\RR_{+}$ and for any constant step-size such that  $ \gamma (\Sigma +\lambda\idm ) \preccurlyeq \idm $ we have  for $\delta\in\big[\frac{1-\sqrt{\gamma \lambda}}{1+\sqrt{\gamma \lambda}},1\big]$, for the recursion in Eq.~(\ref{eq:accgrad}):
\BEAS
\E f(\bar \theta_n )- f ( \theta_* ) \le \min_{ r \in [0,1],\ b \in [0,1]} \bigg[ 2 {\|\Sigma^{r/2} (\theta_0-\theta_*) \|^{2}}\ \lambda^{-r} \left(\frac{36}{\gamma (n+1)^2} +{\lambda} \right)  + 8 \frac{\tau^{2}\tr(\Sigma^{b}) \lambda^{-b}}{n+1 }\bigg] .
\EEAS
\end{theorem}

The proof is straightforward by upper bounding the  terms coming from regularization, depending on $\Sigma(\Sigma+\lambda \idm)^{-1}$, by a power of $\lambda$ times the considered quantities. More precisely, the quantity $\tr(\Sigma(\Sigma+\lambda \idm)^{-1})$ can be seen as an effective dimension of the problem \citep{spline}, and is upper bounded by $\lambda^{-b} \tr(\Sigma^{b})$ for any $b\in [0;1]$. Similarly, $\Vert \Sigma^{1/2} (\Sigma + \lambda I)^{-1/2} \theta_*\Vert^{2}$ can be upper bounded by $\lambda^{-r} \|\Sigma^{r/2} (\theta_0-\theta_*) \|^{2}$. A detailed proof of these results is given in Appendix~\ref{app:tighter}.

In order to   benefit from the acceleration, we    choose $\lambda=(\gamma n^{2})^{-1}$. With such a choice we have the following corollary:

\begin{corollary} \label{cor:tighter_acc}
Assume ($\mathcal{A}_{4,5}$), for any constant step-size  $\gamma (\Sigma +\lambda\idm ) \preccurlyeq \idm$, we have for  $\lambda=\frac{1}{\gamma(n+1)^2}$ and $\delta\in\big[1-\frac{2}{n+2},1\big]$, for the recursion in Eq.~(\ref{eq:accgrad}): 
\BEAS
\E f(\bar \theta_n )- f ( \theta_* ) \le \min_{  r \in [0,1],\ b \in [0,1]} \bigg[ 74\ \frac{\|\Sigma^{r/2} (\theta_0-\theta_*) \|^{2}}{\gamma^{1-r} (n+1)^{2(1-r)}}  + 8\frac{\tau^{2}\gamma^b\tr(\Sigma^{b})}{ (n+1)^{1-2b}}\bigg] .
\EEAS
\end{corollary}
We can make the following observations:
\begin{itemize}
\item The algorithm is independent of $r$ and $b$, thus all the bounds for different values of $(r,b)$ are valid. 
 This is a   strong property of the algorithm, which is indeed adaptative to the regularity  and the effective dimension of the problem (once $\gamma$ is chosen). 
 In situations in which either $d$ is larger than $n$ or $L \| \theta_0 - \theta_\ast \|^2$ is larger than $n^{2}$, the algorithm can still enjoy good convergence properties, by adapting to the best values of $b$ and $r$.
  
 \item For $b=0$ we recover the variance term of Corollary~\ref{Cor_main}, but for $b>0$ and fast decays of eigenvalues of $\Sigma$, the bound may be much smaller; note that we lose in the dependency in $n$, but typically, for large $d$, this can be advantageous.
 
 % and for $b=1$ the variance is bounded by $\tr \Sigma \leq R^2$. \note{that is why we have small learning rate}
\item For $r=0$ we recover the bias term of Corollary~\ref{Cor_main} and for $r=1$ (no assumption at all) the bias is bounded by $\Vert \Sigma^{1/2} \theta_*\Vert^2\leq 4R^2$, which is not going to zero. The smaller $r$ is, the stronger the decrease of the bias with respect to $n$ is (which is coherent with the fact that we have a stronger assumption). Moreover, $r$ is only considered between 0 and 1: indeed,  if $r<0$, the constant$\|(\gamma\Sigma)^{r/2} (\theta_0-\theta_*) \| $ is bigger than $\|\theta_0-\theta_* \|$, but the dependence on $n$ cannot improve beyond $(\gamma n^{2})^{-1}$. This is a classical phenomenon called ``saturation'' \citep{eng1996saturation}. It is linked with the uniform averaging scheme: here, the bias term cannot forget the initial condition faster than $n^{-2}$.
\item  A similar result happens to hold, for averaged gradient descent,  with  $\lambda=(\gamma n)^{-1}$   :
 \begin{eqnarray}
\!\!\!\!\!\E f(\bar \theta_n )- f ( \theta_* ) \le \min_{ r \in [-1,1],\ b \in [0,1]} \bigg[ (18 + \text{Res}(b,r,n,\gamma))\ \frac{ {\|\Sigma^{r/2} (\theta_0-\theta_*) \|^{2}} }{\gamma^{1-r} (n+1)^{(1-r)}}  + 6\frac{\sigma^{2}\gamma^b\tr(\Sigma^{b})}{ (n+1)^{1-b}}\bigg] , \label{eq:tighter_av}
\end{eqnarray} where $\text{Res}(b,r,n,\gamma))$ corresponds to a residual term, which is  smaller than $\tr(\Sigma^{b}) n^{b} \gamma^{1+b}$ if $r\geq0$ and does not exist otherwise. The bias term's dependence on $n$ is degraded, thus the ``saturation'' limit is logically pushed down to $r=-1$, which explains the $[-1;1]$ interval for $r$. The choice $\lambda=(\gamma n)^{-1}$ arises from Th.~\ref{th:av_sgd}, in order to balance both components of the bias term $\lambda + (\gamma n)^{-1}$. This result is proved in Appendix~\ref{app:tighter}.
\item Considering a non-uniform averaging, as proposed as after Theorem~\ref{th:av_sgd} the $\min_{0\le r \le 1}$ in Th.~\ref{th:tighter_acc} and Corollary~\ref{cor:tighter_acc} can be extended to $\min_{-1\le r \le 1}$. Indeed, considering a non-uniform averaging allows to have a faster decreasing bias, pushing the saturation limit observed below.
\end{itemize}

In finite dimension these bounds for the bias and the variance cannot be said to be optimal  independently in any sense we are aware of. Indeed, in finite dimension, the asymptotic rate of convergence for the bias (respectively the variance), when $n$ goes to $\infty$   is governed by $L\|\theta_0-\theta_\ast\|^{2}/n^{2}$ (resp. $\tau^2 d/n$). However, we show in the next section that in the setting of non parametric learning in kernel spaces, these bounds lead to the optimal statistical rate of convergence among all estimators (independently of their computational cost). Moving to the infinite-dimensional setting allows to characterize the optimality of the bounds by showing that they achieve the statistical rate when optimizing the bias/variance tradeoff in Corollary~\ref{cor:tighter_acc}.

%% file: rkhsrate.tex
 \section{Rates of Convergence for Kernel Regression} \label{sec:kernelreg}
 
Computational convergence rates give the speed at which an objective function can decrease depending on the amount of  computation which is allowed. Typically, they show how the error decreases with respect to the number of iterations, as in Theorem~\ref{th:av_sgd}.  Statistical rates, however,  show how close one can get to some objective given some amount of information which is provided. Statistical rates do not depend on some chosen algorithm: these bounds do not involve computation, on the contrary, they state the best performance that no algorithm can beat, given the information, and without computational limits. In particular, any lower bound on the statistical rate implies a lower bound on the computational rates, if each iteration corresponds to access to some new information, here pairs of observations. Interestingly, many algorithms these past few years  have proved to match, with minimal computations (in general one pass through the data), the statistical rate, emphasizing the importance of carrying together optimization and approximation in large scale learning, as    described by \citet{bot2008tradeoffs}. In a similar flavor, it also appears that regularization can be accomplished through early stopping \citep{yao2007early,rudi2015less}, highlighting this interplay between computation and statistics. 

 To characterize the optimality of our bounds, we will show that accelerated-SGD  matches the statistical lower bound in the context of non-parametric estimation.  Even if it may be computationally hard or impossible to implement accelerated-SGD with additive noise in the kernel-based framework below (see remarks following Theorem~\ref{th:asgd_rkhs}), it leads to  the optimal statistical rate for a broader class of problems than  averaged-SGD, showing that for a wider set of trade-offs, acceleration is optimal.
%Even if implementing AccSGD may be computationally hard or impossible in infinite dimension, it matches the statistical rate. Moreover this propertiy is satisfied for a larger set of problems. \note{ad: not improved that much... Do not find a way to express it clearly. Shall we describe properly what a ``problem'' is : a function, or a set function + kernel. Shall we clarify dependence between the function, its regularity (which depends on the kernel), the algorithm (which depends on the kernel) and discuss the fact that we can have the choice of the kernel or not ...?}. 

A natural extension of the finite-dimensional analysis is the non-parametric setting, especially with reproducing kernel Hilbert spaces. 
In the setting of non-parametric regression, we consider a probability space $\mathcal X\times \RR$  with probability distribution $ \rho$, and assume that we are given an i.i.d.~sample $(x_i,y_i)_{i=1,\dots, n} \sim \rho^{\otimes n}$, and denote by $\rho_X$ the marginal distribution of $x_n$ in $\mathcal X$;  the aim of non-parametric least-squares regression is to find a function $g: \mathcal{X} \rightarrow \mathbb R$,  which minimizes the expected risk: 
\begin{eqnarray}
f (g) = \E_{\rho} [(g(x_n)-y_n)^{2}]. \label{eq_minpbinrkhs}
\end{eqnarray}
The optimal function $g$ is the conditional expectation $g(x) = \E_\rho(y_n|x)$.
In the kernel regression setting, we consider as hypothesis space a reproducing kernel Hilbert space \citep{aron, steinwart, smolabook} associated with a kernel function $K$. The space $\mathcal{H}$ is a subspace of the space of squared integrable functions $L_{\rho_{X}}^{2}$. We look for a function $g_{\mathcal{H}}$ which satisfies: $f(g_\mathcal{H}) = \inf_{g\in \mathcal{H}} f(g)$, and $g_{\mathcal{H}} $ belongs to the closure $\bar{\H}$ of $\H$ (meaning that there exists a sequence of function $g_n \in \mathcal{H}$ such that $\Vert g_n-g_H\Vert_{L^{2}_{\rho_X}} \rightarrow 0$). When $\H$ is dense, the minimum is attained for the regression function defined above. This function however \emph{is not} in $\mathcal{H}$ in general. 
Moreover there exists an operator $\Sigma: \mathcal{H} \rightarrow \H$, which extends the finite-dimensional population covariance matrix, that will allow the characterization of the smoothness of $g_\mathcal{H}$. This operator is known to be trace class when $\E_{\rho_X} [K(x_n,x_n)] < \infty$. 

Data points $x_i$ are mapped into the RKHS, via the feature map: $x \mapsto K_{x}$, where $K_x: \H \rightarrow \RR $ is a function in the RKHS, such that $K_x: y \mapsto K(x,y)$. The reproducing property\footnote{It states that for any function $g \in \H$, $\langle g, K_x\rangle_{\H}= g(x)$, where $\langle \cdot, \cdot \rangle_{\H}$ denotes the scalar product in the Hilbert space.}  allows to express the minimization problem \eqref{eq_minpbinrkhs} as a least-squares linear regression problem: for any $g\in \H$, $ f (g) = \E_{\rho} [(\langle g, K_{x_n} \rangle_{\H} -y_n)^{2}],$ and can thus be seen as an extension to the infinite-dimensional setting of linear least-squares regression.

However, in such a setting, both quantities $\|\Sigma^{r/2} \theta_*\|_{\H}$ and $\tr(\Sigma^{b})$ \emph{may exist or not}.  It thus arises as a natural \emph{assumption} to consider the smaller $r \in [-1;1]$ and the smaller $b\in [0;1]$ such that  
\begin{itemize}
 \item $\|\Sigma^{r/2} \theta_*\|_{\H} < \infty$ (meaning that $\Sigma^{r/2} \theta_* \in \H$), \hfill{($\mathcal{A}_6$)}
 \item    $\tr(\Sigma^{b}) < \infty$.\hfill{($\mathcal{A}_7$)}
\end{itemize}
The quantities considered in Sections~\ref{sec_assumptions} and \ref{sec:tighter} are the  natural finite-dimensional twins of these assumptions. However in infinite dimension a quantity may exist or not and it is thus an assumption to consider its existence, whereas it can only be characterized by its value, big or small,  in finite dimension.

 In the last decade, \citet{vito2005model, sma2002best} studied non-parametric least-squares regression in the RKHS framework. These works were extended to derive rates of convergence depending on assumption $(\mathcal{A}_6)$:  \citet{yin2008online} studied un-regularized stochastic gradient descent  and derived asymptotic rate of convergence $O(n^{-\frac{1-r}{2-r}})$, for $-1\le r \le 1$;    \citet{zha2004solving} studies stochastic gradient descent with averaging, deriving  similar rates of convergence for $ 0 \le r\le 1$; whereas  \citet{tar2011online} give similar performance for $-1\le r \le 0$. This rate 
is optimal without assumption on the spectrum of the covariance matrix, but  comes from a worst-case analysis: we show in the next paragraphs that we can derive a tighter and optimal rate for both averaged-SGD (recovering results from \cite{dieuleveut2015parametric}) and accelerated-SGD, for a larger class of kernels for the latter.

We will first describe results for averaged-SGD, then  increase the validity region of these rates (which depends on $r,b$) using averaged accelerated SGD. We show that the derived rates match statistical rates for our setting and thus our algorithms reach the optimal prediction performance for certain $b$ and $r$.

\subsection{Averaged SGD}
We have the following result, proved in Appendix~\ref{app:tighter} and following from Theorem~\ref{th:av_sgd}: for some fixed $b,r$, we choose the best step-size $\gamma$, that optimizes the bias-variance  trade-off, while still satisfying the constraint $\gamma\le 1/(2R^{2})$. We get a result for the stochastic oracle (multiplicative/additive noise).

\begin{theorem}\label{th:av_rkhs}
With $\lambda =\frac{1}{\gamma n}$, we have, if $r \le b$, under Assumptions $(\mathcal{A}_{1,3,6,7}$) and the stochastic oracle \eq{sto}, for any constant step-size  $\gamma\leq \frac{1}{2R^2}$,  with $\gamma\varpropto n^{\frac{-b+r}{b+1-r}}$, for the recursion in Eq.~(\ref{eq:lms}): 
\begin{eqnarray*}
\E f(\bar \theta_n )- f ( \theta_* )  &\le & \bigg((27 + o(1))  \big\| \Sigma^{r/2}  (\theta_0 - \theta_\ast)\big\|^{2}+ 6 \sigma^{2} \tr(\Sigma^{b})\bigg) \ n ^{-\frac{1-r}{b+1-r}}.
\end{eqnarray*}
\end{theorem}

We can make the following remarks: 
\begin{itemize}
\item The term $o(1)$ stands for a quantity which is decreasing to 0 when $n\rightarrow\infty$. More specifically, this constant is smaller than $3 \tr(\Sigma^{b})$ divided by $n^{\chi}$, where $\chi$ is bigger than 0 (see Appendix~\ref{app:tighter}).  The result comes from Eq.~(\ref{eq:tighter_av}), with the choice of the optimal step-size.
\item We recover results from \citet{dieuleveut2015parametric}, but with a simpler analysis resulting from the consideration of the regularized version of the problem associated with a choice of~$\lambda$. However, we only recover rates in the finite horizon setting.
\item This result shows that we get the optimal rate of convergence under Assumptions~$(\mathcal{A}_{6,7})$, for $r\le b$. This point will be discussed in more details after Theorem~\ref{th:asgd_rkhs}.
\end{itemize}
We now turn to the averaged accelerated SGD algorithm. We prove that it enjoys the optimal rate of convergence for a larger class of problems, but only for the additive noise which corresponds to knowing the distribution of $x_n$.

\subsection{Accelerated SGD}
Similarly, choosing the best step-size $\gamma$, it comes from Theorem~\ref{th:tighter_acc}, that in the RKHS setting, under additional Assumptions $(\mathcal{A}_{6,7})$, we have for the the averaged accelerated  algorithm the following result:
\begin{theorem}\label{th:asgd_rkhs}
With $\lambda =\frac{1}{\gamma n^{2}}$, we have, if $ r \le b+1/2$, under Assumptions  $(\mathcal{A}_{4,5,6,7})$,   for any constant step-size  $\gamma\leq \frac{1}{L+\lambda}$, with $\gamma \varpropto n^{\frac{-2b+2r-1}{b+1-r}}$, for the recursion in Eq.~(\ref{eq:accgrad}): 
\begin{eqnarray*}
\E f(\bar \theta_n )- f ( \theta_* )  &\le & \bigg( 74   \big\| \Sigma^{r/2}  (\theta_0 - \theta_\ast)\big\|^{2}+ 8 \tau^{2} \tr(\Sigma^{b})\bigg) \ n ^{-\frac{1-r}{b+1-r}}.
\end{eqnarray*}
\end{theorem}

We can make the following remarks: 
\begin{itemize}
\item The rate $\frac{1-r}{b+1-r}$ is always between 0 and 1, and improves when our assumptions gets stronger ($r$ getting smaller, $b$ getting smaller). Ultimately, with $b\rightarrow 0$, and $r\rightarrow -1$, we recover the finite-dimensional $n^{-1}$ rate.
\item We can achieve this optimal rate when  $ r \le b +1/2$. Beyond, if $r> b+1/2$, the rate is only $n^{-2(1-r)}.$ Indeed, the bias term cannot decrease faster than $n^{-2(1-r)}$, as $\gamma$ is compelled to be upper bounded. 
\item The same phenomenon appears in the un-accelerated averaged situation, as shown by Theorem~\ref{th:av_rkhs}, but the critical value was then $r \le b$. There is thus a region (precisely $b< r \le b+1/2 $) in which only the accelerated algorithm gets the optimal rate of convergence. Note that we increase the optimality region towards optimization problems which are more ill-conditioned, naturally benefiting from acceleration.

\item This algorithm cannot be computed in practice (at least with computational limits). Indeed, without any further assumption on the kernel $K$, it is not possible to compute images of vectors by the covariance operator $\Sigma$ in the RKHS. 
However, as explained in the following remark, this is enough to show optimality of our algorithm.

Note that the easy computability is a great advantage of the   multiplicative/additive noise variant of the algorithms, for which the current point $\theta_n$ can always be expressed as a finite sum of features $\theta_n= \sum_{i=1}^{n} \alpha_i K_{x_i}$, with $\alpha_i\in \RR$, leading to a tractable algorithm.  An accelerated variant of SGD naturally arises from our algorithm, when considering this stochastic oracle from Eq.~\eqref{eq:sto}. Such a variant can be implemented but does not behave similarly for large step sizes, say, $\gamma \simeq 1/(2R^{2})$. It is an open problem to prove convergence results for averaged accelerated gradient under this multiplicative/additive noise.

\item These rates happen to be optimal from a statistical perspective, meaning that no algorithm which is given access to the sample points and the distribution of $x_n$ can perform better for all functions that satisfy assumption $(\mathcal{A}_7)$, for a kernel satisfying $(\mathcal{A}_6$). Indeed it is equivalent to assuming that the function lives in some ellipsoid in the space of squared integrable functions. Note that the statistical minimization problem (and thus the lower bound) does not depend on the kernel, and is valid without computational limits. The  case of learning with kernels is studied  by \citet{cap2007optimal} which shows these minimax convergence rates under $(\mathcal{A}_{6,7})$, under assumption that $-1 \le r\le 0$ (but state that it can be easily  extended to $ 0\le r \le 1$). They do not assume knowledge of the distribution of the inputs; however,  \citet{Massart} and \citet{Tsyb} discuss optimal rates on ellipsoids, and \citet{gyorfi2006distribution} proves similar results for certain class of functions under a known distribution for the input data, showing that the knowledge of the distribution does not make any difference. This minimax statistical  rate stands without computational limits and is thus valid for both algorithms (additive noise that corresponds to knowing $\Sigma$, and multiplicative/additive noise). The  optimal tradeoff  is derived for an extended region of $b, r$ (namely $r \le b +1/2$ instead of $r \le b $) in the accelerated case which shows the improvement upon non-accelerated averaged SGD. 
\item The choice of the optimal $\gamma $ is difficult in practice, as the parameters $b,r$ are unknown, and this remains an open problem \citep[see, e.g.,][for some methods for non-parametric regression]{lep}.
\EIT

%% file: experiment.tex
\section{Experiments}
\label{experiment}

We illustrate now our theoretical results on synthetic examples. For $d=25$ we  consider normally distributed inputs $x_n$ with random covariance matrix $\Sigma$ which has eigenvalues $1/i^{3}$ , for $i = 1,\dots,d$, and  random optimum $\theta_{*}$ and starting point $\theta_{0}$ such that $\Vert \theta_{0}-\theta_{*}\Vert=1$. The outputs $y_n$ are generated from a linear function with homoscedastic noise with unit signal to noise-ratio ($\sigma^2=1$),  we take $R^2=\tr \Sigma$  the average radius of the data and a step-size $\gamma=1/R^{2}$ and $\lambda=0$. The additive noise oracle is used. We show results averaged over $10$ replications. 

We compare the performance of averaged SGD (AvSGD), AccSGD (usual Nesterov acceleration for convex functions) and  our novel averaged accelerated SGD from Section~\ref{sec_result_acc} (AvAccSGD, which is not the averaging of AccSGD)   on two different problems: one deterministic ($\Vert\theta_{0}-\theta_{*}\Vert=1$, $\sigma^{2}=0$) which will illustrate how the bias term behaves, and one purely stochastic ($\Vert\theta_{0}-\theta_{*}\Vert=0$, $\sigma^{2}=1$)  which will illustrate how the variance term behaves.
\begin{figure}[!h]
\centering
\begin{minipage}[c]{.45\linewidth}
\includegraphics[width=\linewidth]{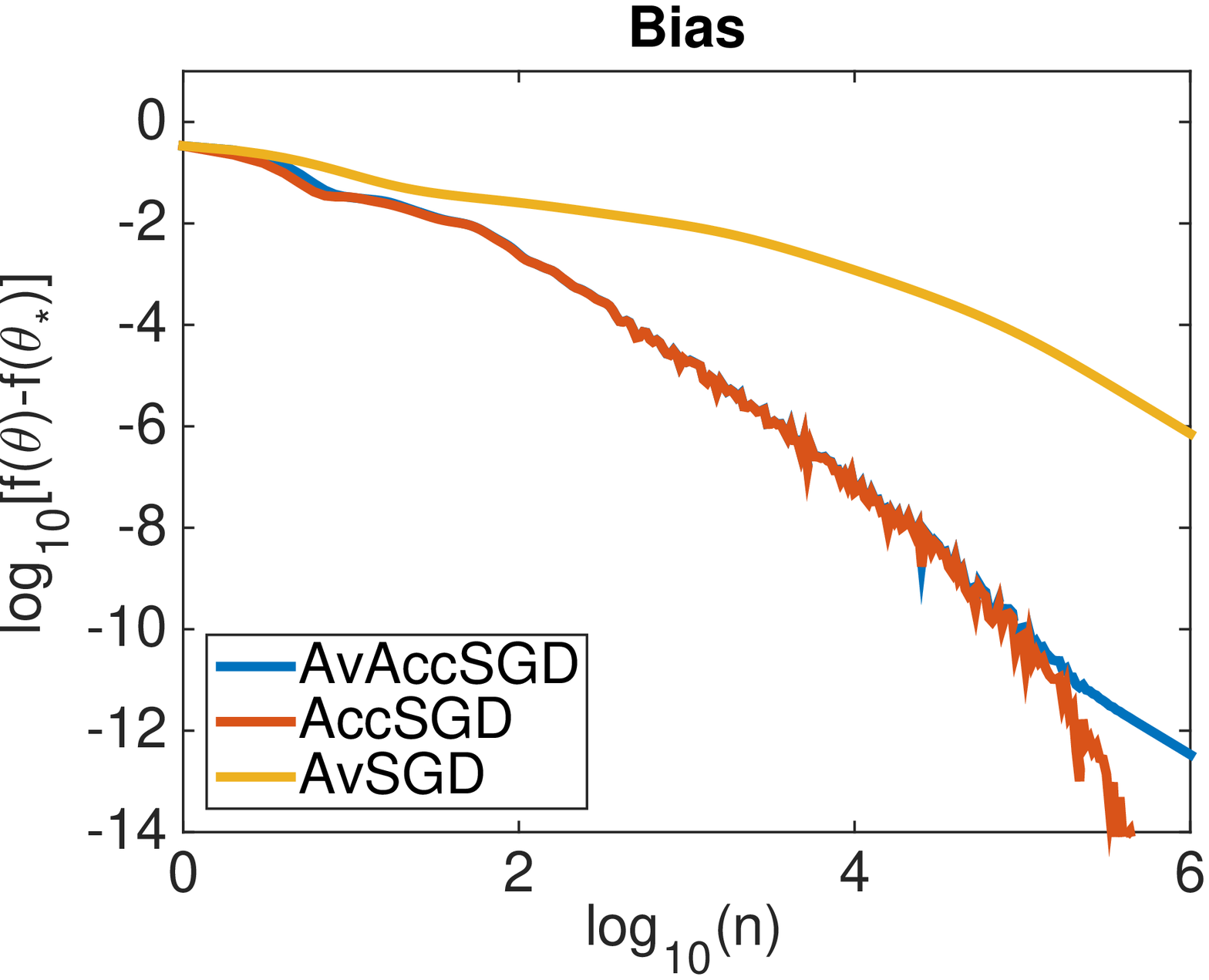}
   \end{minipage} \hspace*{.08\linewidth}
   \begin{minipage}[c]{.45\linewidth}
\includegraphics[width=\linewidth]{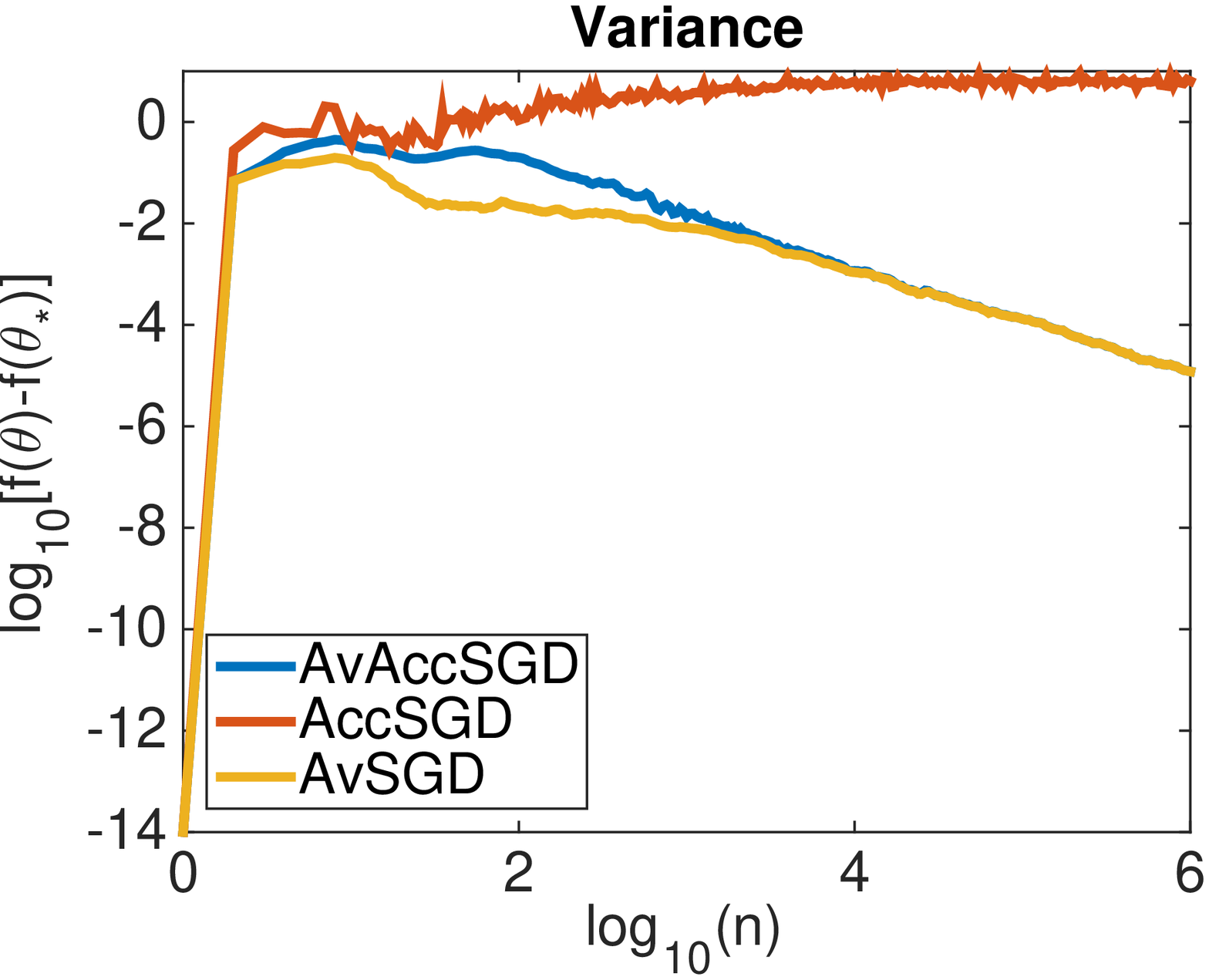}
   \end{minipage} 
   \caption{Synthetic problem ($d=25$) and $\gamma=1/R^{2}$. Left: Bias. Right: Variance.}
   \label{fig:plot1}
\end{figure}
For the bias (left plot of  \myfig{plot1}), AvSGD converges at speed $O(1/n)$, while AvAccSGD and AccSGD converge both at speed $O(1/n^{2})$. However, as mentioned in the observations following Corollary~\ref{Cor_main}, AccSGD takes advantage of the hidden strong convexity of the quadratic function and starts converging linearly at the end. For the variance (right plot of  \myfig{plot1}), AccSGD is not converging to the optimum and keeps oscillating whereas AvSGD and AvAccSGD both converge  to the optimum at a speed $O(1/n)$. However AvSGD remains slightly faster in the beginning. 

Note that for small $n$,  or when the bias $L \Vert \theta_{0}-\theta_{*}\Vert^{2} / n^2$ is much bigger than the variance $\sigma^{2}d/n$, the bias may have a stronger effect, although asymptotically, the variance always dominates. It is thus essential to have an algorithm which is optimal in both regimes, what is achieved by AvAccSGD.

%% file: conclusion.tex
\section{Conclusion}

%\note{FB: correct references below}

In this paper, we showed that stochastic \emph{averaged} accelerated gradient descent was robust to structured noise in the gradients present in least-squares regression. Beyond being the first algorithm which is jointly optimal in terms of both bias and finite-dimensional variance, it is also adapted to finer assumptions such as fast decays of the covariance matrices or optimal predictors with large norms.

Our current analysis is performed for least-squares regression. While it could be directly extended to smooth losses through efficient online  Newton methods~\citep{bm1}, an extension to  all smooth or self-concordant-like functions~\citep{adaptivityJMLR} would widen its applicability. Moreover, our accelerated gradient analysis is performed for additive noise (i.e., for least-squares regression, with knowledge of the population covariance matrix) and it would be interesting to study the robustness of our results in the context of least-mean squares. Finally, our analysis relies on single observations per iteration and could be made finer by using mini-batches~\citep{mini,mini2}, which should not change the variance term but could impact the bias term.

%% file: lemmasemisto.tex
\section{Proof of Section \ref{sec_result_av}}
\label{app:proofsemisto}
\subsection{Proof of Lemma \ref{lemma:sgdsemisto}}

We proof here  Lemma \ref{lemma:sgdsemisto} which is the extension of Lemma 2 of  \citet{bm1} for the regularized case. The proof technique relies on the fact that   recursions in \eq{lmssigma} are linear since the cost function is quadratic which allows us to obtain $\theta_{n}-\theta_{*}$ in closed form.

For any regularization parameter $\lambda\in\RR_{+}$ and any constant step-size $ \gamma (\Sigma +\lambda\idm ) \preccurlyeq \idm $ we may rewrite the regularized stochastic gradient recursion in \eq{lmssigma} as:
 \BEAS
 \theta_n  - \theta_\ast & = & \big[ \idm - \gamma \Sigma- \gamma \lambda \idm \big] ( \theta_{n-1}  - \theta_\ast) 
+ \gamma \xi_{n}
+\lambda \gamma (\theta_0-\theta_\ast).
 \EEAS
We thus get for $n\geq1$ the expansion
\BEAS
\theta_n  - \theta_\ast &=& (\idm -\gamma \Sigma -\gamma \lambda \idm)^{n}(\theta_0  - \theta_\ast )+ \gamma \sum_{k=1}^{n}(\idm -\gamma \Sigma -\gamma \lambda \idm)^{n-k}\xi_{k}\\
&&+\gamma \lambda  \sum_{k=1}^{n}(\idm -\gamma \Sigma -\gamma \lambda \idm)^{n-k} (\theta_0-\theta_\ast)\\
&=& (\idm -\gamma \Sigma -\gamma \lambda \idm)^{n}(\theta_0  - \theta_\ast )+ \gamma \sum_{k=1}^{n}(\idm -\gamma \Sigma -\gamma \lambda \idm)^{n-k}\xi_{k}\\
&&+\lambda  \big[ \idm -(\idm -\gamma \Sigma -\gamma \lambda \idm)^{n}\big]  (\Sigma + \lambda \idm)^{-1}  (\theta_0-\theta_\ast)\\
&=&(\idm -\gamma \Sigma -\gamma \lambda \idm)^{n}[\idm -\lambda (\Sigma + \lambda \idm)^{-1}  ](\theta_0  - \theta_\ast )+ \gamma \sum_{k=1}^{n}(\idm -\gamma \Sigma -\gamma \lambda \idm)^{n-k}\xi_{k}\\
&&+\lambda  (\Sigma + \lambda \idm)^{-1}  (\theta_0-\theta_\ast).
\EEAS
We then have  using the definition of the average
\BEAS
n(\bar \theta_{n-1}  - \theta_\ast)&=& \sum_{j=0}^{n-1}(\theta_{j}  - \theta_\ast) \\
&=&  \sum_{j=0}^{n-1} (\idm -\gamma \Sigma -\gamma \lambda \idm)^{j}[\idm -\lambda (\Sigma + \lambda \idm)^{-1}  ](\theta_0  - \theta_\ast ) +\gamma  \sum_{j=0}^{j-k}\sum_{k=1}^{n}(\idm -\gamma \Sigma -\gamma \lambda \idm)^{n-k}\xi_{k}\\
&&+n\lambda  (\Sigma + \lambda \idm)^{-1}  (\theta_0-\theta_\ast).
\EEAS
For which we will compute the two sums separately 
\begin{multline}\nonumber
\sum_{j=0}^{n-1} (\idm -\gamma \Sigma -\gamma \lambda \idm)^{j}[\idm -\lambda (\Sigma + \lambda \idm)^{-1}  ](\theta_0  - \theta_\ast )\\
=\frac{1}{\gamma}\big[ \idm -(\idm -\gamma \Sigma -\gamma \lambda \idm)^{n}\big]  (\Sigma + \lambda \idm)^{-1} [\idm -\lambda (\Sigma + \lambda \idm)^{-1}  ](\theta_0  - \theta_\ast ), 
\end{multline}
and 
\BEAS
\gamma  \sum_{j=0}^{n-1}\sum_{k=1}^{n}(\idm -\gamma \Sigma -\gamma \lambda \idm)^{j-k}\xi_{k} 
&=&\gamma \sum_{k=1}^{n}\Big( \sum_{j=k}^{n-1}(\idm -\gamma \Sigma -\gamma \lambda \idm)^{j-k}\Big)\xi_{k} \\
&=&\gamma \sum_{k=1}^{n}\Big( \sum_{j=0}^{n-1-k}(\idm -\gamma \Sigma -\gamma \lambda \idm)^{j}\Big)\xi_{k} \\
&=& \sum_{k=1}^{n}\big[ \idm -(\idm -\gamma \Sigma -\gamma \lambda \idm)^{n-k}\big]  (\Sigma + \lambda \idm)^{-1} \xi_{k}.
\EEAS
Gathering the three terms together, we thus have 
\BEAS
n(\bar \theta_{n-1}  - \theta_\ast)&=&\frac{1}{\gamma}\big[ \idm -(\idm -\gamma \Sigma -\gamma \lambda \idm)^{n}\big]  (\Sigma + \lambda \idm)^{-1} [\idm -\lambda (\Sigma + \lambda \idm)^{-1}  ](\theta_0  - \theta_\ast ) \\
&&+ \sum_{k=1}^{n}\big[ \idm -(\idm -\gamma \Sigma -\gamma \lambda \idm)^{n-k}\big]  (\Sigma + \lambda \idm)^{-1} \xi_{k}+n\lambda  (\Sigma + \lambda \idm)^{-1}  (\theta_0-\theta_\ast)\\
&=& \Big[\frac{1}{\gamma}\big[ \idm -(\idm -\gamma \Sigma -\gamma \lambda \idm)^{n}\big]   [\idm -\lambda (\Sigma + \lambda \idm)^{-1}  ]+n\lambda\idm\Big](\Sigma + \lambda \idm)^{-1}(\theta_0  - \theta_\ast )\\
&&+\sum_{k=1}^{n}\big[ \idm -(\idm -\gamma \Sigma -\gamma \lambda \idm)^{n-k}\big]  (\Sigma + \lambda \idm)^{-1} \xi_{k}.
\EEAS
Using standard martingale square moment inequalities which amount to consider $\xi_{i}$, $i = 1, \cdots, n$  independent, the variance of the sum is the sum of variances and we have for $V=\E \xi_{n}\otimes \xi_{n}$ 
\begin{multline}
\textstyle
n^{2}\E \Vert \Sigma^{1/2}(\bar \theta_{n-1}  - \theta_\ast )\Vert^{2}=
\sum_{k=1}^{n} \tr \big[ \idm -(\idm -\gamma \Sigma -\gamma \lambda \idm)^{n-k}\big]^2  \Sigma(\Sigma + \lambda \idm)^{-2} V 
\\+\Big\Vert  \Big[\frac{1}{\gamma}\big[ \idm -(\idm -\gamma \Sigma -\gamma \lambda \idm)^{n}\big]   [\idm -\lambda (\Sigma + \lambda \idm)^{-1}  ]+n\lambda\idm\Big]\Sigma^{1/2}(\Sigma + \lambda \idm)^{-1}(\theta_0  - \theta_\ast )\Big\Vert^{2}\label{eq:reso}.
\end{multline}

Since all the matrices in this equality are symmetric positive-definite we are allowed to bound
\BEA
  \Big[\frac{1}{\gamma}\big[ \idm -(\idm -\gamma \Sigma -\gamma \lambda \idm)^{n}\big]   [\idm -\lambda (\Sigma + \lambda \idm)^{-1}  ]+n\lambda\idm\Big]&\preccurlyeq& \Big(\frac{1}{\gamma}+n\lambda\Big)\idm \label{eq:caca}\\
   \big[ \idm -(\idm -\gamma \Sigma -\gamma \lambda \idm)^{n-k}\big]^2&\preccurlyeq& \idm \nonumber.
  \EEA
  This concludes the proof of the Lemma  \ref{lemma:sgdsemisto} 
  \begin{multline}\label{eq:coco}
\E \Vert \Sigma^{1/2}(\bar \theta_{n-1}  - \theta_\ast )\Vert^{2}\leq \Big(\frac{1}{n\gamma}+\lambda\Big)^{2}\Vert \Sigma^{1/2}(\Sigma + \lambda \idm)^{-1}(\theta_0  - \theta_\ast )\Vert^{2}\\
+ \frac{1}{n} \tr \Sigma(\Sigma + \lambda \idm)^{-2} V.
    \end{multline}

\subsection{Proof when only $\|\theta_0-\theta_{*}\|$ is finite}\label{app:semistonbiasbis}
% When  $\lambda=0$,  \eq{reso} becomes 
% $$
% n^{2}\E \Vert \Sigma^{1/2}(\bar \theta_{n-1}  - \theta_\ast )\Vert^{2}=\Big\Vert \frac{1}{\gamma}\big[ \idm -(\idm -\gamma \Sigma )^{n}\big]   \Sigma ^{-1/2}(\theta_0  - \theta_\ast )\Big\Vert^{2}
% + \sum_{k=1}^{n} \tr \big[ \idm -(\idm -\gamma \Sigma )^{n-k}\big]^2  \Sigma ^{-1} V .
% $$
Unfortunately  $\| \Sigma^{-1}(\theta_0-\theta_{*})\|$ may not be finite. However we can use that for all $u\in[0,1]$ we have $\frac{1-(1-u)^{n}}{n u}\leq 1$\footnote{since 
$\frac{1-(1-u)^{n}}{u}=\sum_{k=0}^{n}(1-u)^{k}\leq n$} and have therefore the  bound
\BEAS
  &&\Big[\frac{1}{\gamma}\big[ \idm -(\idm -\gamma \Sigma -\gamma \lambda \idm)^{n}\big]   [\idm -\lambda (\Sigma + \lambda \idm)^{-1}  ]+n\lambda\idm\Big][\Sigma+\lambda \idm]^{-1}\\
  &&\preccurlyeq \Big[\frac{1}{\gamma}\big[ \idm -(\idm -\gamma \Sigma -\gamma \lambda \idm)^{n}\big]  +n\lambda\idm\Big][\Sigma+\lambda \idm]^{-1}\\
    &&\preccurlyeq \Big[\frac{1}{\gamma}\big[ \idm -(\idm -\gamma \Sigma -\gamma \lambda \idm)^{n}\big][\Sigma+\lambda \idm]^{-1}  +n\lambda [\Sigma+\lambda \idm]^{-1}\Big]\\
        &&\preccurlyeq \idm +n\idm.
  \EEAS
 Combining with \eq{caca} we have 
 \begin{multline}\nonumber 
  \Big\Vert  \Big[\frac{1}{\gamma}\big[ \idm -(\idm -\gamma \Sigma -\gamma \lambda \idm)^{n}\big]   [\idm -\lambda (\Sigma + \lambda \idm)^{-1}  ]+n\lambda\idm\Big]\Sigma^{1/2}(\Sigma + \lambda \idm)^{-1}(\theta_0  - \theta_\ast )\Big\Vert^{2}\\ \leq (n+1)\Big(\frac{1}{\gamma}+n\lambda\Big)\Vert \Sigma^{1/2}(\Sigma + \lambda \idm)^{-1/2}(\theta_0  - \theta_\ast )\Vert^{2}
 \end{multline}

 which implies that 
  \begin{multline}
\E \Vert \Sigma^{1/2}(\bar \theta_{n-1}  - \theta_\ast )\Vert^{2}\leq 2\Big(\frac{1}{n\gamma}+\lambda\Big)\Vert \Sigma^{1/2}(\Sigma + \lambda \idm)^{-1/2}(\theta_0  - \theta_\ast )\Vert^{2}\\
+ \frac{1}{n} \tr \Sigma(\Sigma + \lambda \idm)^{-2} V.
    \end{multline}
which is interesting when only $\|\theta_0-\theta_{*}\|$ is finite.

\subsection{Proof when the noise is not structured}\label{app:semistonotstruc}
The bound in \eq{coco} becomes less interesting when the noise is not structured. However using the same technique we have that $\big[ \idm -(\idm -\gamma \Sigma -\gamma \lambda \idm)^{n-k}\big]^2 (\Sigma + \lambda \idm)^{-1}\preccurlyeq (n-k)\gamma \idm$ and we get the following upper-bound on the variance 
\BEAS
\sum_{k=1}^{n} \tr \big[ \idm -(\idm -\gamma \Sigma -\gamma \lambda \idm)^{n-k}\big]^2  \Sigma(\Sigma + \lambda \idm)^{-2} V 
&\leq &\gamma  \sum_{k=1}^{n}  (n-k) \tr \Sigma(\Sigma + \lambda \idm)^{-1} V\\
&\leq &\gamma  \frac{n(n+1)}{2} \tr \Sigma(\Sigma + \lambda \idm)^{-1} V.
\EEAS
Therefore we get
  \begin{multline}
\E \Vert \Sigma^{1/2}(\bar \theta_{n-1}  - \theta_\ast )\Vert^{2}\leq \Big(\frac{1}{n\gamma}+\lambda\Big)^{2}\Vert \Sigma^{1/2}(\Sigma + \lambda \idm)^{-1}(\theta_0  - \theta_\ast )\Vert^{2}\\
+ \gamma \tr \Sigma(\Sigma + \lambda \idm)^{-1} V,
    \end{multline}
which is meaningful when the noise is not structured.

%% file: appagd.tex
\section{Proof of Theorem~\ref{th:av_sgd} }\label{app:avsgd}

%\note{FB: TOO DRY!}
In this section, we will prove Theorem~\ref{th:av_sgd}. The proof relies on a decomposition of the error as the sum of three main terms which will be studied separately. We state decomposition in Section~\ref{app:expansion} then prove upper bounds for the different terms in Sections~\ref{app:regbased} and \ref{app:otherterms}.

\subsection{Expansion of the recursion} \label{app:expansion}

We may   rewrite the regularized stochastic gradient recursion as:
 \BEAS
 \theta_n & = & \big[ \idm - \gamma x_n\otimes x_n - \gamma \lambda \idm \big] \theta_{n-1}  
+ \gamma \varepsilon_n x_n
+ \gamma \langle x_n, \theta_\ast \rangle x_n+\lambda \gamma \theta_0
\\
 \theta_n  - \theta_\ast & = & \big[ \idm - \gamma x_n\otimes x_n - \gamma \lambda \idm \big] ( \theta_{n-1}  - \theta_\ast) 
+ \gamma \varepsilon_n x_n
+\lambda \gamma (\theta_0-\theta_\ast).
%\\
%\theta_n - \theta_\ast^\lambda & = & \big[ \idm - \gamma x_n\otimes x_n- \gamma \lambda \idm \big] ( \theta_{n-1} - \theta_\ast^\lambda)
%- \gamma \varepsilon_n x_n
%+ \gamma \big[ \Sigma - x_nx_n^\ast \big] \big[ \theta_\ast^\lambda - \theta_\ast \big] \\
%& = & \big[ \idm - \gamma x_n\otimes x_n- \gamma \lambda \idm \big] ( \theta_{n-1} - \theta_\ast^\lambda)
%+ \gamma \xi_n
 \EEAS
 For $i \geqslant k$, let 
$$M(i,k) = 
\big[ \idm - \gamma x_i \otimes x_i - \gamma \lambda \idm  \big] \cdots \big[ \idm - \gamma x_k\otimes x_k - \gamma \lambda \idm \big]$$
be an operator from $\mathcal{H}$ to $\mathcal{H}$.
We have the expansion
$$
\theta_n - \theta_\ast =    M(n,1) ( \theta_0 - \theta_\ast)
 + \gamma \sum_{k=1}^n  M(n,k+1) \varepsilon_k x_k + \gamma \sum_{k=1}^n  M(n,k+1) \lambda (\theta_0-\theta_\ast).
$$
Our goal is to study these three terms separately and bound $ \| \Sigma^{1/2} ( \bar{\theta}_n - \theta_\ast  ) \|$ for each of them.

\subsection{Regularization-based bias term} \label{app:regbased}
This is the term: 
$\theta_n - \theta_\ast =  \gamma \sum_{k=1}^n  M(n,k+1) \lambda (\theta_0- \theta_\ast)$, which corresponds to
the recursion \begin{eqnarray}
\theta_n - \theta_\ast = \big( \idm - \gamma x_n \otimes x_n - \gamma \lambda \idm \big)
(\theta_{n-1} - \theta_\ast) + \lambda \gamma  (\theta_0- \theta_\ast), \label{eq:storegbased}
\end{eqnarray} initialized with $\theta_0 = \theta_\ast$, and no noise.

Following the proof technique of \citet{bm1}, we are going to consider a  related recursion by replacing in Equation~\eqref{eq:storegbased} the operator $x_n \otimes x_n$ by its expectation $\Sigma$. Thus, we consider $\eta_n$ defined as
 $$
 \eta_n - \theta_\ast =   \gamma \sum_{k=1}^n 
  ( \idm - \gamma \Sigma - \lambda\gamma\idm)^{n-k} \lambda  (\theta_0- \theta_\ast),
 $$
 which satisfies the recursion  (with initialization $\eta_0 = \theta_\ast$) and
 $$
 \eta_n - \theta_\ast = \big[ \idm - \gamma \Sigma - \lambda\gamma\idm\big] ( \eta_{n-1} - \theta_\ast) + \lambda \gamma (\theta_0- \theta_\ast).
 $$
 In order to bound $\| \Sigma^{1/2 } (\theta_n - \theta_*)\| $, we will independently bound $\| \Sigma^{1/2 } (\eta_n - \theta_*)\| $ and $\| \Sigma^{1/2 } (\theta_n - \eta_n)\| $ using Minkowski's inequality.

 \paragraph{Bounding $\| \Sigma^{1/2 } (\theta_n - \eta_n)\| $.}
 
 We have $\theta_0 - \eta_0 = 0$, and
 $$
 \theta_n - \eta_n 
 = \big[ \idm - \gamma x_n\otimes x_n - \lambda\gamma\idm\big] ( \theta_{n-1} - \eta_{n-1}) 
 + \gamma \big[ \Sigma - x_n\otimes x_n \big] (\eta_{n-1} -\theta_\ast). $$
 
 We can now bound the recursion for $ \theta_n - \eta_n $ as follows, using standard online learning proofs \citep{nem}:
 \BEAS
 \|  \theta_n - \eta_n \|^2
 & \leqslant & 
  \|  \theta_{n-1}- \eta_{n-1} \|^2
  - 2\gamma \big\langle \theta_{n-1}- \eta_{n-1} ,  (  x_n\otimes x_n+ \lambda \idm)
  (  \theta_{n-1}- \eta_{n-1} ) \big \rangle \\
  & & 
    + 2\gamma \big\langle \theta_{n-1}- \eta_{n-1} , \big[ \Sigma - x_n\otimes x_n \big] (\eta_{n-1} -\theta_\ast)\big \rangle \\
    & & + \gamma^2 \big\|
   \big[ x_n\otimes x_n  + \lambda \idm\big] ( \theta_{n-1} - \eta_{n-1}) 
 -   \big[ \Sigma - x_n\otimes x_n \big] (\eta_{n-1} -\theta_\ast)
    \big\|^2.
 \EEAS
 By taking conditional expectations given $\F_{n-1}$, we get, using first the fact that $\E(\Sigma-x_n\otimes x_n|\mathcal{F}_{n-1})=0$ and the inequality $(a+b)^2\le 2(a^2+b^2)$, then developing and using $\E[(x_n\otimes x_n)^{2}]\le R^{2}\Sigma$, which is assumption $\mathcal{A}_1$.
 \BEAS
 \E \big( \|  \theta_n - \eta_n \|^2 | \F_{n-1} \big)
 & \leqslant & 
  \|  \theta_{n-1}- \eta_{n-1} \|^2
  - 2\gamma \big\langle \theta_{n-1}- \eta_{n-1} ,  (   \Sigma + \lambda \idm)
  (  \theta_{n-1}- \eta_{n-1} ) \big \rangle \\
     & & + 2 \gamma^2 \E \big( \big\|
   \big[ x_n\otimes x_n  + \lambda \idm\big] ( \theta_{n-1} - \eta_{n-1}) \big\|^2 | \F_{n-1} \big) \\
   & & 
   + 2\gamma^2 \E \big(\big\|
     \big[ \Sigma - x_n\otimes x_n \big] (\eta_{n-1} -\theta_\ast)
    \big\|^2 | \F_{n-1} \big)
\\
 & \leqslant & 
  \|  \theta_{n-1}- \eta_{n-1} \|^2
  - 2\gamma \big\langle \theta_{n-1}- \eta_{n-1} ,  (   \Sigma + \lambda \idm)
  (  \theta_{n-1}- \eta_{n-1} ) \big \rangle \\
     & & + 2 \gamma^2 \big\langle \theta_{n-1}- \eta_{n-1} , 
(     R^2 \Sigma + \lambda^2 \idm + 2 \lambda \Sigma )
       (  \theta_{n-1}- \eta_{n-1} ) \big \rangle  \\
   & & 
   + 2 \gamma^2 R^2 \langle \eta_{n-1} -\theta_\ast , \Sigma   \rangle
\\
 & \leqslant & 
  \|  \theta_{n-1}- \eta_{n-1} \|^2
  - 2\gamma \big[ 1 - \gamma( R^2 + 2\lambda) \big] \big\langle \theta_{n-1}- \eta_{n-1} , \Sigma 
  (  \theta_{n-1}- \eta_{n-1} ) \big \rangle \\
&&   + 2 \gamma^2 R^2 \langle \eta_{n-1} -\theta_\ast, \Sigma (\eta_{n-1} -\theta_\ast) \rangle.
 \EEAS
This leads by taking full expectations and moving terms to
\BEAS
\E
\big\langle \theta_{n-1}- \eta_{n-1} , \Sigma 
  (  \theta_{n-1}- \eta_{n-1} ) \big \rangle 
  & \leqslant & \frac{1}{2\gamma \big[ 1 - \gamma( R^2 + 2\lambda) \big] }\big[ \E   \|  \theta_{n-1}- \eta_{n-1} \|^2 - \E   \|  \theta_{n}- \eta_{n} \|^2 \big]  \\
  & & + \frac{   \gamma R^2} 
  {  1 - \gamma( R^2 + 2\lambda)   } \langle \eta_{n-1} -\theta_\ast, \Sigma (\eta_{n-1} -\theta_\ast) \rangle .
  \EEAS
 Thus, if $\gamma( R^2 + 2\lambda) \leqslant \frac{1}{2}$
  \BEAS
\E
\big\langle \theta_{n-1}- \eta_{n-1} , \Sigma 
  (  \theta_{n-1}- \eta_{n-1} ) \big \rangle 
  & \leqslant & \frac{1}{\gamma   }\big[ \E   \|  \theta_{n-1}- \eta_{n-1} \|^2 - \E   \|  \theta_{n}- \eta_{n} \|^2 \big]  \\
  & & +  { 2 \gamma R^2} 
  \E \langle \eta_{n-1} -\theta_\ast, \Sigma (\eta_{n-1} -\theta_\ast) \rangle.
 \EEAS
This leads to, summing and using initial conditions $\theta_0-\eta_0=0$, then using convexity to  upper bound $\big\langle \bar\theta_{n}- \bar\eta_{n} , \Sigma 
  (  \bar\theta_{n}- \bar\eta_{n} ) \big \rangle  \le \frac{1}{n+1}  \sum_{k=0}^{n} \big\langle \theta_{n}- \eta_{n} , \Sigma 
  (  \theta_{n}- \eta_{n} ) \big \rangle$,
 \BEAS
\E
\big\langle \bar\theta_{n}- \bar\eta_{n} , \Sigma 
  (  \bar\theta_{n}- \bar\eta_{n} ) \big \rangle 
  & \leqslant & \frac{ 2 \gamma R^2} {n+1}
   \sum_{k=0}^n \langle \eta_{k} -\theta_\ast, \Sigma (\eta_{k} -\theta_\ast) \rangle.
 \EEAS

\paragraph{Bounding $\| \Sigma^{1/2 } (\eta_n- \theta_*)\| $.}
Moreover we have:
\BEAS
\eta_n - \theta_\ast & = &   \lambda ( \Sigma + \lambda \idm)^{-1} (\theta_0-\theta_\ast) - ( \idm - \gamma \Sigma - \lambda \gamma \idm)^n \big[  \lambda ( \Sigma + \lambda \idm)^{-1}  (\theta_0-\theta_\ast) \big] \\
\bar\eta_n - \theta_\ast
& = &  \lambda ( \Sigma + \lambda \idm)^{-1}  (\theta_0-\theta_\ast)
- \frac{1}{n+1} \sum_{k=0}^n
 ( \idm - \gamma \Sigma - \lambda \gamma \idm)^k \big[  \lambda ( \Sigma + \lambda \idm)^{-1} (\theta_0-\theta_\ast) \big] 
\\
& = &   \lambda ( \Sigma + \lambda \idm)^{-1}  (\theta_0-\theta_\ast)
- \frac{1}{n+1}  \gamma^{-1} ( \Sigma + \lambda \idm)^{-1} \big[ \idm - 
( \idm - \gamma \Sigma - \lambda \gamma \idm)^{n+1}  \big] \big[  \lambda ( \Sigma + \lambda \idm)^{-1}  (\theta_0-\theta_\ast) \big] 
.
\EEAS
This leads using Minkowski inequality to
\BEAS
\big( \E \| \Sigma^{1/2} (  \eta_n - \theta_\ast)
\|^2 \big)^{1/2}
& \leqslant &  \| \lambda \Sigma^{1/2} ( \Sigma + \lambda \idm)^{-1} (\theta_0-\theta_\ast)  \| \\
\big( \E \| \Sigma^{1/2} (  \bar\eta_n - \theta_\ast)
\|^2 \big)^{1/2}
& \leqslant &  \| \lambda \Sigma^{1/2} ( \Sigma + \lambda \idm)^{-1}  (\theta_0-\theta_\ast)  \| .
\EEAS
Thus this part is such that 
\BEAS
\big( \E  \| \Sigma^{1/2} ( \bar\theta_n - \theta_\ast) \|^2 \big)^{1/2}
& \leqslant & \| \lambda \Sigma^{1/2} ( \Sigma + \lambda \idm)^{-1} (\theta_0-\theta_\ast)  \| 
+ \bigg(
  { 2 \gamma R^2} \| \lambda \Sigma^{1/2} ( \Sigma + \lambda \idm)^{-1}  (\theta_0-\theta_\ast) \| ^2
\bigg)^{1/2}
\\
& \leqslant & \| \lambda \Sigma^{1/2} ( \Sigma + \lambda \idm)^{-1}  (\theta_0-\theta_\ast) \| 
\big( 1 + 
  \sqrt{ 2 \gamma R^2}\big),
  \EEAS
that gives the first bound on the regularization-based bias
\BEQ
\label{eq:A}
  \E  \| \Sigma^{1/2} ( \bar\theta_n - \theta_\ast) \|^2 
    \leqslant  \| \lambda \Sigma^{1/2} ( \Sigma + \lambda \idm)^{-1}  (\theta_0-\theta_\ast)  \|^2 
\big( 1 + 
  \sqrt{ 2 \gamma R^2}\big)^2.
  \EEQ

 \subsection{Expansion without the regularization term} \label{app:otherterms}
 We will follow here the outline of the proof of \citet{gyo} which considers a full expansion of the  function value  $\Vert \Sigma^{1/2}(\bar \theta_{n}-\theta_{*})\Vert^{2}$.
 This corresponds to 
 $$
\theta_n - \theta_\ast =   M(n,1) ( \theta_0 - \theta_\ast)
 - \gamma \sum_{k=1}^n  M(n,k+1) \varepsilon_k x_k.
$$
 We have
 \BEAS
 \E \sum_{i=0}^n \sum_{j=0}^n \langle \theta_i - \theta_\ast, 
 \Sigma ( \theta_j - \theta_\ast) \rangle
 & = & \E \sum_{i=0}^n  \langle \theta_i - \theta_\ast ,
  \Sigma ( \theta_i - \theta_\ast) \rangle
 + 2 \E \sum_{i=0}^{n-1} \sum_{j=i+1}^n \langle  \theta_i - \theta_\ast,
 \Sigma ( \theta_j - \theta_\ast) \rangle.
 \EEAS
 
  Moreover,
 \BEAS
  & & \E \sum_{i=0}^{n-1} \sum_{j=i+1}^n\langle  \theta_i - \theta_\ast,
 \Sigma ( \theta_j - \theta_\ast) \rangle\\
 & = & 
 \E \sum_{i=0}^{n-1} \sum_{j=i+1}^n \bigg\langle  \theta_i - \theta_\ast,  \Sigma \bigg[
 M(j,i+1) ( \theta_i - \theta_\ast)
 - \sum_{k=i+1}^j M(j,k+1)  \gamma \varepsilon_k x_k  
 \bigg] \bigg\rangle
\\
& = & 
 \E \sum_{i=0}^{n-1} \sum_{j=i+1}^n \langle \theta_i - \theta_\ast, 
 \Sigma   
 M(j,i+1) ( \theta_i - \theta_\ast) \rangle \mbox{ because } \varepsilon_k x_k \mbox{ and } \theta_i \mbox{ are independent,}
 \\
& = & 
 \E \sum_{i=0}^{n-1} \sum_{j=i+1}^n \langle \theta_i - \theta_\ast, 
 \Sigma  
( \idm - \gamma \Sigma - \gamma \lambda \idm )^{j-i}
( \theta_i - \theta_\ast) \rangle \mbox{ because } M(j,i+1) \mbox{ and } \theta_i \mbox{ are independent,}
\\
& = & 
 \E \sum_{i=0}^{n-1} \Big\langle \theta_i - \theta_\ast,
 \gamma^{-1}  \Sigma ( \Sigma + \lambda \idm)^{-1}
\big[  ( \idm - \gamma \Sigma  - \gamma \lambda \idm ) - ( \idm - \gamma \Sigma  - \gamma \lambda \idm )^{n-i+1} \big]
( \theta_i - \theta_\ast) \Big\rangle
\\
& \leqslant & 
 \E \sum_{i=0}^{n}  \Big\langle \theta_i - \theta_\ast, 
 \gamma^{-1}   \Sigma ( \Sigma + \lambda \idm)^{-1}
  ( \idm - \gamma \Sigma  - \gamma \lambda \idm )   ( \theta_i - \theta_\ast)
  \Big\rangle \mbox{ using  } (\Sigma + \lambda \idm ) \preccurlyeq \idm,
\\
& = & 
 \gamma^{-1}  \E \sum_{i=0}^{n}  \langle \theta_i - \theta_\ast ,  \Sigma ( \Sigma + \lambda \idm)^{-1} 
 ( \theta_i - \theta_\ast  ) \rangle
  - \E \sum_{i=0}^{n}  \langle \theta_i - \theta_\ast , \Sigma ( \theta_i - \theta_\ast ) \rangle.
 \EEAS

   We thus simply need to bound $ \gamma^{-1} \E \sum_{i=0}^{n}  \langle \theta_i - \theta_\ast ,  \Sigma ( \Sigma + \lambda \idm)^{-1} 
 ( \theta_i - \theta_\ast  ) \rangle
$, to get a bound on
$n^2 \E \| \Sigma^{1/2} ( \bar\theta_n - \theta_\ast) \|^2$.

 \paragraph{Recursion on operators.}
 We   have:
 \BEAS
   \E \big[ M(i,k) \Sigma (\Sigma + \lambda \idm)^{-1} M(i,k)^\ast \big] 
 & = & \E
 \Big[
M(i,k+1)  \big[ \idm - \gamma\varphi(x_{k})\otimes \varphi(x_{k}) - \gamma \lambda \idm \big]  
\Sigma (\Sigma + \lambda \idm)^{-1} 
 \\
&&  \big[ \idm - \gamma\varphi(x_{k})\otimes \varphi(x_{k}) - \gamma \lambda \idm \big]  M(i,k+1)^\ast
 \Big]
\\
 & = & \E
 \Big[
 M(i,k+1) \Big( \Sigma (\Sigma + \lambda \idm)^{-1}  - 2\gamma \Sigma  + \gamma^2
 \big[ \varphi(x_{k}) \otimes \varphi(x_{k}) 
 \\
 &&+ \lambda \idm \big]  \Sigma (\Sigma + \lambda \idm)^{-1} 
 \big[ \varphi(x_{k}) \otimes \varphi(x_{k}) + \lambda \idm \big]   \Big) M(i,k+1)^\ast
 \Big]
\\
 & \preccurlyeq & \E
 \Big[
 M(i,k+1)\big[ \Sigma (\Sigma + \lambda \idm)^{-1}   - 2 \gamma \Sigma  
 \\
 &&+ \gamma^2 ( R^2 + 2 \lambda) \Sigma \big] M(i,k+1)^\ast
 \Big]
\\
& = & \E
 \Big[
 M(i,k+1)  \Sigma (\Sigma + \lambda \idm)^{-1}   M(i,k+1)^\ast
 \Big]
 \\
 &&
 - \gamma ( 2 - \gamma ( R^2 + 2 \lambda) ) \E
 \Big[
 M(i,k+1)   \Sigma  M(i,k+1)^\ast
 \Big],
 \EEAS
 which leads
 to
 \BEA
 \E
 \Big[
 M(i,k+1)  \Sigma  M(i,k+1)^\ast
 \Big]
 & \preccurlyeq & 
 \frac{1}{\gamma ( 2 - \gamma ( R^2 + 2 \lambda) )}
 \Big(
 E
 \Big[
 M(i,k+1)  \Sigma (\Sigma + \lambda \idm)^{-1}M(i,k+1)^\ast
 \Big]  \nonumber \\
 & & 
 -E
 \Big[
 M(i,k)  \Sigma (\Sigma + \lambda \idm)^{-1}  M(i,k)^\ast
 \Big]
 \Big) \label{eq:recuop}.
\EEA
 Using the operator $T$ on matrices defined below, this corresponds to showing
 $$
 ( \idm - \gamma T) \big[ \Sigma(\Sigma+\lambda \idm) \big] \preccurlyeq 
 \Sigma(\Sigma+\lambda \idm) - \gamma \Sigma.
 $$

\paragraph{Noise term.}
For $\theta_0 - \theta_\ast = 0$, we have:
 \BEAS
 & & 
  \E   \langle \theta_i - \theta_\ast ,  \Sigma ( \Sigma + \lambda \idm)^{-1} 
 ( \theta_i - \theta_\ast  ) \rangle \\
& = & \gamma^2 \E \sum_{k=1}^i \sum_{j=1}^i \varepsilon_j x_j^\ast M(i,j+1)^\ast 
 \Sigma ( \Sigma + \lambda \idm)^{-1} M(i,k+1) \varepsilon_k x_k \mbox{ by expanding all terms,}
\\
& = & \gamma^2 \E \sum_{k=1}^i  \varepsilon_k x_k^\ast M(i,k+1)^\ast
 \Sigma ( \Sigma + \lambda \idm)^{-1}  M(i,k+1) \varepsilon_k x_k
\mbox{ using independence,} \\
& = &\gamma^2 \tr\bigg( \sum_{k=1}^i  \E \varepsilon_k^2 x_k
 x_k
^\ast  \E M(i,k+1)^\ast  \Sigma ( \Sigma + \lambda \idm)^{-1} M(i,k+1) 
\bigg)
\\
& \leqslant &\gamma^2 \sigma^2 \tr\bigg( \sum_{k=1}^i     \E  M(i,k+1) \Sigma M(i,k+1)^\ast
 \Sigma ( \Sigma + \lambda \idm)^{-1} 
\bigg) \mbox{ using our assumption regarding the noise.}
\EEAS 
Using the recurrence between operators
\BEAS
& & 
  \E   \langle \theta_i - \theta_\ast ,  \Sigma ( \Sigma + \lambda \idm)^{-1} 
 ( \theta_i - \theta_\ast  ) \rangle \\
& \leqslant & \frac{\gamma \sigma^2}{ 2 - \gamma ( R^2 + 2 \lambda)}  \tr \sum_{k=1}^{i}    \bigg(   
 E
 \Big[
 M(i,k+1)  \Sigma ( \Sigma + \lambda \idm)^{-1} M(i,k+1)^\ast  \Sigma ( \Sigma + \lambda \idm)^{-1}
 \Big] \\
 & & 
 -E
 \Big[
 M(i,k)   \Sigma ( \Sigma + \lambda \idm)^{-1}  M(i,k)^\ast  \Sigma ( \Sigma + \lambda \idm)^{-1}
 \Big]
\bigg)\\
& \leqslant & 
\frac{\gamma \sigma^2}{2 - \gamma ( R^2 + 2 \lambda)} \tr   \bigg(   
 E
 \Big[
 M(i,i+1)   \Sigma ( \Sigma + \lambda \idm)^{-1} M(i,i+1)^\ast  \Sigma ( \Sigma + \lambda \idm)^{-1}
 \Big] \\
 & & 
 -E
 \Big[
 M(i,1)   \Sigma ( \Sigma + \lambda \idm)^{-1}   M(i,1)^\ast  \Sigma ( \Sigma + \lambda \idm)^{-1}
 \Big]
\bigg) \mbox{ by summing,} \\
& \leqslant & 
\frac{\gamma \sigma^2}{2 - \gamma ( R^2 + 2 \lambda)}   \tr  
\Sigma^2 ( \Sigma + \lambda \idm)^{-2}.
\EEAS

 This implies that for the noise process
 $$
 \E  \| \Sigma^{1/2} (  \bar{\theta}_n - \theta_\ast ) \|^2 \leqslant 
 \bigg(
 \frac{  \sigma^2}{n+1}   \tr  \big[ 
\Sigma^2 ( \Sigma + \lambda \idm)^{-2} \big]\bigg) \frac{1}{1 - \gamma ( R^2/2 +   \lambda)}.
 $$
 Note that when $\gamma$ tends to zero, we recover the optimal variance term.

 \paragraph{Noiseless term.}
 Without noise, we then need to bound:
 $$
 \gamma^{-1} \E \sum_{i=0}^{n}  \langle \theta_i - \theta_\ast ,  \Sigma ( \Sigma + \lambda \idm)^{-1} 
 ( \theta_i - \theta_\ast  ) \rangle,
 $$
 with $\theta_i - \theta_\ast = M(i,1) ( \theta_0 - \theta_\ast)$, that is
 $$
 \gamma^{-1} \E \sum_{i=0}^{n}  \tr  \Big[ M(i,1)^\ast  \Sigma ( \Sigma + \lambda \idm)^{-1} 
  M(i,1)  ( \theta_0 - \theta_\ast) ( \theta_0 - \theta_\ast)^\ast \Big].
 $$
 We follow here the proof of \citet{defossez2014constant} and consider the operator $T$ from symmetric matrices to symmetric matrices defined as
 $$
 TA = (\Sigma + \lambda \idm) A +   A (\Sigma + \lambda \idm)- \gamma 
 E \big[ (x_n \otimes x_n + \lambda \idm ) A  (x_n \otimes x_n + \lambda \idm )\big].
 $$
 of the form $TA = (\Sigma + \lambda \idm) A + (\Sigma + \lambda \idm)  A - \gamma  S A$.

 The operator $S$ is self-adjoint and positive. Moreover:
 \BEAS
 \langle A, SA \rangle
 & = &      \E \tr \big[ A  (x_n \otimes x_n + \lambda \idm )  A  (x_n \otimes x_n + \lambda \idm )  \big]
\\
& = &     \tr \big[  
  2 A^2   \lambda \Sigma  +    \lambda^2 A^2
  \big]
 +   \E \tr \big[ \langle x_n , A x_n\rangle^2 \big]
\\
& \leqslant &
   \tr \big[  
  2 A^2   \lambda \Sigma  +    \lambda^2 A^2
  \big]
 +      \E \tr \big[ \| x_n \|^2 \langle  x_n , A^2\rangle \big]
 \mbox{ using Cauchy-Schwarz inequality,}
\\
& \leqslant &
   \tr \big[  
  2 A^2   \lambda \Sigma  +    \lambda^2 A^2
  \big]
 +       R^2 \tr \Sigma A^2   
\\
& \leqslant & (R^2 + 2 \lambda)
   \tr \big[  \Sigma + \lambda \idm] A^2.
 \EEAS
  We have for any symmetric matrix $A$:
 $$
 \E M(i,1)^\ast A M(i,1) =  ( \idm - \gamma T )^i A.
 $$
 Thus, 
  $$
 \gamma^{-1} \E \sum_{i=0}^{n}  \tr  \Big[ M(i,1)^\ast  \Sigma ( \Sigma + \lambda \idm)^{-1} 
  M(i,1)  ( \theta_0 - \theta_\ast) ( \theta_0 - \theta_\ast)^\ast \Big]
  =   \gamma^{-1} \E \sum_{i=0}^{n}
  \langle
  ( \idm - \gamma T )^i A, E_0
  \rangle  $$
  with $E_0 =  ( \theta_0 - \theta_\ast) ( \theta_0 - \theta_\ast)^\ast $ and $A = \Sigma ( \Sigma + \lambda \idm)^{-1}$.
  This leads to
  $$
  \gamma^{-1} \E  
  \langle\langle
\gamma^{-1}  T^{-1} ( \idm - 
  ( \idm - \gamma  T )^{n+1} ) A, E_0
  \rangle\rangle,
  $$
  where $\langle \langle \cdot, \cdot \rangle \rangle$ denote the dot-product between self-adjoint operators.
  
The sum is less than its limit for $n \to \infty$, and thus, we can get rid of the term $ ( \idm - \gamma  T )^{n+1} $, and we need to bound
$$
 \gamma^{-2}   \langle\langle   M , E_0
  \rangle\rangle = \gamma^{-2}   \langle\langle   T^{-1} ( \Sigma(\Sigma+\lambda \idm)^{-1}), E_0
  \rangle\rangle  ,
  $$
with $M :=  T^{-1} \big[ \Sigma(\Sigma+\lambda \idm)^{-1}\big]$, i.e., such that
\BEA
\Sigma(\Sigma+\lambda \idm)^{-1} & = & (\Sigma + \lambda \idm) M + M ( \Sigma+\lambda \idm)
- \gamma \E ( x_n \otimes x_n  +\lambda \idm)
M  ( x_n \otimes x_n  +\lambda \idm) \nonumber \\
& = &  (\Sigma + \lambda \idm) M + M ( \Sigma+\lambda \idm)
- \gamma S M. \label{eq:SM}
\EEA
So that : 
\BEAS
M & = &  \big[ ( \Sigma + \lambda \idm) \otimes \idm + \idm \otimes  ( \Sigma + \lambda \idm)  \big]^{-1}
\big[ \Sigma (\Sigma + \lambda \idm)^{-1} \big]
+ \gamma  \big[ ( \Sigma + \lambda \idm) \otimes \idm + \idm \otimes  ( \Sigma + \lambda \idm)  \big]^{-1} SM \\
& = &  \frac{1}{2} \Sigma (\Sigma + \lambda \idm)^{-2} 
+ \gamma  \big[ ( \Sigma + \lambda \idm) \otimes \idm + \idm \otimes  ( \Sigma + \lambda \idm)  \big]^{-1} SM.
\EEAS
The operator $ ( \Sigma + \lambda \idm) \otimes \idm + \idm \otimes  ( \Sigma + \lambda \idm) $ is self adjoint, and so is its inverse, thus: 
\begin{eqnarray*}
  \gamma^{-2}   \langle\langle   M , E_0   \rangle\rangle & = &   \gamma^{-2}   \langle\langle  \frac{1}{2} \Sigma (\Sigma + \lambda \idm)^{-2} 
+ \gamma  \big[ ( \Sigma + \lambda \idm) \otimes \idm + \idm \otimes  ( \Sigma + \lambda \idm)  \big]^{-1} SM , E_0   \rangle\rangle \\
&=& \frac{1}{2} \gamma^{-2}   \langle\langle   \Sigma (\Sigma + \lambda \idm)^{-2} ,  E_0   \rangle\rangle 
+ \gamma^{-1}  \langle\langle SM , \big[  ( \Sigma + \lambda \idm) \otimes \idm + \idm \otimes  ( \Sigma + \lambda \idm)  \big]^{-1} E_0   \rangle\rangle\\
&=& \frac{1}{2} \gamma^{-2}   \tr(   \Sigma (\Sigma + \lambda \idm)^{-2} E_0  ) 
+ \gamma^{-1}  \langle\langle SM , \big[  ( \Sigma + \lambda \idm) \otimes \idm + \idm \otimes  ( \Sigma + \lambda \idm)  \big]^{-1} E_0   \rangle\rangle
  \end{eqnarray*}  
Moreover, \begin{eqnarray*}
E_0  &=&   ( \theta_0 - \theta_\ast) ( \theta_0 - \theta_\ast)^\ast \\
& = & ( \Sigma + \lambda \idm)^{1/2} ( \Sigma + \lambda \idm)^{-1/2}  ( \theta_0 - \theta_\ast) ( \theta_0 - \theta_\ast)^\ast ( \Sigma + \lambda \idm)^{-1/2} ( \Sigma + \lambda \idm)^{+1/2} \\
& \preccurlyeq  & [( \theta_0 - \theta_\ast)^\ast ( \Sigma + \lambda \idm)^{-1}  ( \theta_0 - \theta_\ast)]\ \    ( \Sigma + \lambda \idm), \\
&& \qquad\quad \text{ as }  ( \Sigma + \lambda \idm)^{-1/2}  ( \theta_0 - \theta_\ast) ( \theta_0 - \theta_\ast)^\ast ( \Sigma + \lambda \idm)^{-1/2} \preccurlyeq ( \theta_0 - \theta_\ast)^\ast ( \Sigma + \lambda \idm)^{-1}  ( \theta_0 - \theta_\ast) I.
\end{eqnarray*} 
Thus, as $ [( \Sigma + \lambda \idm) \otimes \idm + \idm \otimes  ( \Sigma + \lambda \idm)]^{-1} $ is an non-decreasing operator on $(S_n(\RR), \lecc)$ (see technical Lemma~\ref{lem:increasingop} in Appendix~\ref{app:techlem}):
 \begin{eqnarray*}
 \hspace{-4em}\big[  ( \Sigma + \lambda \idm) \otimes \idm + \idm \otimes  ( \Sigma + \lambda \idm)  \big]^{-1} E_0   & \lecc & \big[  ( \Sigma + \lambda \idm) \otimes \idm + \idm \otimes  ( \Sigma + \lambda \idm)  \big]^{-1} \left( [( \theta_0 - \theta_\ast)^\ast ( \Sigma + \lambda \idm)^{-1}  ( \theta_0 - \theta_\ast)]  ( \Sigma + \lambda \idm) \right)\\&=& \frac{( \theta_0 - \theta_\ast)^\ast ( \Sigma + \lambda \idm)^{-1}  ( \theta_0 - \theta_\ast)}{2} I.
\end{eqnarray*} 
Thus as $SM$ is positive : \begin{eqnarray*}
  \gamma^{-2}   \langle\langle   M , E_0   \rangle\rangle 
&\le & \frac{1}{2\gamma^{2} }   \tr(   \Sigma (\Sigma + \lambda \idm)^{-2} E_0  ) 
+ \frac{( \theta_0 - \theta_\ast)^\ast ( \Sigma + \lambda \idm)^{-1}  ( \theta_0 - \theta_\ast)}{2\gamma} \tr( SM ).\end{eqnarray*}
Moreover we can upper bound $\tr( SM ) :$ using Equation~\eqref{eq:SM} we have
$$ \tr ( \Sigma(\Sigma+\lambda \idm)^{-1} )
= 2 \tr (\Sigma + \lambda \idm) M
- \gamma \tr  \E ( x_n \otimes x_n  +\lambda \idm)
M  ( x_n \otimes x_n  +\lambda \idm)
 $$
then, using Assumption (\ref{eq:rr}) :
 $$
 \tr  \E ( x_n \otimes x_n  +\lambda \idm)
M  ( x_n \otimes x_n  +\lambda \idm)
\leqslant R^2 \tr M \Sigma + 2 \tr M\Sigma \lambda + \lambda^2 \tr M 
\leqslant ( R^2 + 2 \lambda) \tr M ( \Sigma + \lambda \idm).
$$
This implies
 \BEAS
 \tr \big[ \Sigma(\Sigma+\lambda \idm)^{-1} \big]
 & \geqslant & 
\big(  \frac{2} {R^2 + 2 \lambda} - \gamma \big) \tr  \E ( x_n \otimes x_n  +\lambda \idm)
M  ( x_n \otimes x_n  +\lambda \idm), 
\\
& \geqslant & 
 \frac{1} {R^2 + 2 \lambda}   \tr  \E ( x_n \otimes x_n  +\lambda \idm)
M  ( x_n \otimes x_n  +\lambda \idm) \mbox{ since } 
\gamma (R^2 + 2 \lambda) \leqslant 1,
\\
& \geqslant & 
 \frac{1} {R^2 + 2 \lambda}   \tr  S M
.
\EEAS
Thus finally:  
 \BEAS \gamma^{-2}   \langle\langle   M , E_0   \rangle\rangle 
&\le & \frac{1}{2\gamma^{2}} \tr E_0 \Sigma ( \Sigma + \lambda \idm)^{-2} +  \frac{( \theta_0 - \theta_\ast)^\ast ( \Sigma + \lambda \idm)^{-1}  ( \theta_0 - \theta_\ast)}{2\gamma}  ( R^2 + 2 \lambda) 
  \tr ( \Sigma(\Sigma+\lambda \idm)^{-1}),
\EEAS
which leads to the desired error term.

\subsection{Proof when only $\|\theta_0-\theta_{*}\|$ is finite}\label{app:stonbiasbis}
When $\lambda=0$, without noise, we then need to bound:
 $$
 \gamma^{-1} \E \sum_{i=0}^{n}  \langle \theta_i - \theta_\ast , 
 ( \theta_i - \theta_\ast  ) \rangle,
 $$
 with $\theta_i - \theta_\ast = M(i,1) ( \theta_0 - \theta_\ast)$, that is
 $$
 \gamma^{-1} \E \sum_{i=0}^{n}  \tr  \Big[ M(i,1)^\ast    M(i,1)  ( \theta_0 - \theta_\ast) ( \theta_0 - \theta_\ast)^\ast \Big].
 $$
 By definition of $M(i,1)$ we have that $ \E M(i,1)^\ast    M(i,1)\preccurlyeq \idm$ leading to 
 $$
 \gamma^{-1} \E \sum_{i=0}^{n}  \langle \theta_i - \theta_\ast , 
 ( \theta_i - \theta_\ast  ) \rangle \leq  \frac{(n+1)\Vert \theta_0 - \theta_\ast\Vert^{2}}{\gamma}.
 $$
 For the regularization-based bias we also have 
 $$
  \| \lambda \Sigma^{1/2} ( \Sigma + \lambda \idm)^{-1}  (\theta_0-\theta_\ast)  \|^2\leq \lambda  \|  \Sigma^{1/2} ( \Sigma + \lambda \idm)^{-1/2}  (\theta_0-\theta_\ast)  \|^2.
 $$
\subsection{Proof when the noise is not structured}\label{app:stonotstruc}

For $\Vert \theta_{0}-\theta_{*}\Vert=0$ we have $\theta_{n}-\theta_{*}=\gamma \sum_{k=1}^n  M(n,k+1) \varepsilon_k x_k$ which leads to 
$$
\E \Vert \Sigma^{1/2}(\theta_{n}-\theta_{*})\Vert^{2}=\gamma^{2}  \sum_{k=1}^n \tr\E M(n,k+1)^{*}\Sigma M(n,k+1) V,
$$
where $V=\E \varepsilon_k^{2} x_k x_k^{*}$.
And using the recursion on operators in \eq{recuop} by changing order of elements  we have 
\BEAS
 \E
 \Big[
 M(n,k+1)^*  \Sigma  M(n,k+1)
 \Big]
 & \preccurlyeq & 
 \frac{1}{\gamma ( 2 - \gamma ( R^2 + 2 \lambda) )}
 \Big(
 E
 \Big[
 M(n,k+1)^*  \Sigma (\Sigma + \lambda \idm)^{-1}M(n,k+1)
 \Big]  \nonumber \\
 & & 
 -E
 \Big[
 M(n,k)^*  \Sigma (\Sigma + \lambda \idm)^{-1}  M(n,k)
 \Big]
 \Big).
\EEAS
And by adding the terms 
$$
\E \Vert \Sigma^{1/2}(\theta_{n}-\theta_{*})\Vert^{2}=\frac{\gamma^{2}}{\gamma ( 2 - \gamma ( R^2 + 2 \lambda) )} \tr \Sigma (\Sigma + \lambda \idm)^{-1}V,
$$
We conclude by convexity 
$$
\E \Vert \Sigma^{1/2}(\bar \theta_{n}-\theta_{*})\Vert^{2}=\frac{\gamma^{2}}{\gamma ( 2 - \gamma ( R^2 + 2 \lambda) )} \tr \Sigma (\Sigma + \lambda \idm)^{-1}V.
$$

%% file: appaagd.tex
\section{Convergence of Accelerated Averaged Stochastic Gradient Descent } \label{app:agd}

We now prove Theorem~\ref{th:acc_sgd}. We thus consider iterates satisfying Eq.~\eqref{eq:accgrad}, under Assumptions~\eqref{eq:a4},~\eqref{eq:a5new}. We consider a fixed step size $\gamma$ such that $\gamma(\Sigma+\lambda I)\lecc I$. Seing Eq.~\eqref{eq:accgrad} as a linear second order for $\theta_n$, we will derive from exact calculations a decomposition of the errors a sum of three terms that will be studied independently. 
The proof is organized as follows: in Section~\ref{app:b1}, we state the formulation as a second order linear system and derive the three main terms that have to be studied (see Lemma~\ref{lem:3termsacc}). Section~\ref{app:b2} studies asymptotic behaviors of the three terms, ignoring some exponentially decreasing terms, in order to give insight of how they behave. This section is not necessary for the proof, indeed a direct and exact calculation in the eigenbasis of $\Sigma$, following~\cite{o2013adaptive}, is provided in Section~\ref{subsec:direct}. Results are summed up in Section~\ref{app:b4}.

\subsection{General expansion}\label{app:b1}
We study the regularized stochastic accelerated gradient descent recursion defined for $n\geq 1$ by
\BEAS
\theta_n & = & \nu_{n-1} -\gamma f'(\nu_{n-1}) - \gamma \lambda (\nu_n-\theta_0) + \gamma \xi_n 
\\
\nu_n&=&\theta_{n} +\delta (\theta_{n}-\theta_{n-1}),
\EEAS
starting from $\theta_{0}=\nu_{0}\in\mathcal{H}$. We may   rewrite it for a quadratic function $f: \theta\mapsto \frac{1}{2} \langle \theta-\theta_*,\Sigma(\theta-\theta_*)\rangle$  for $n\geq 2$ as
$$
 \theta_n  =  \big[ \idm - \gamma \Sigma - \gamma \lambda \idm \big] \big[\theta_{n-1} +\delta (\theta_{n-1}-\theta_{n-2})\Big] 
+ \gamma \xi_n +\gamma \lambda \theta_0
+ \gamma \Sigma \theta_\ast,
$$
with $\theta_{0}\in\mathcal{H}$ and $\theta_{1}= \big[ \idm - \gamma \Sigma - \gamma \lambda \idm \big] \theta_{0}
+ \gamma \xi_1 +\gamma \lambda \theta_0
+ \gamma \Sigma \theta_\ast$.

And by centering around the optimum, we get:  
$$
  \theta_n  - \theta_\ast  =  \big[ \idm - \gamma \Sigma - \gamma \lambda \idm \big]  \big[\theta_{n-1}-\theta_\ast  +\delta (\theta_{n-1}-\theta_\ast -\theta_{n-2}+\theta_\ast )\Big] 
+ \gamma \xi_n 
+ \lambda \gamma(\theta_0- \theta_\ast).
$$
 Thus this is a second order iterative system which is standard to cast in a linear form          
\BEQ\label{eq:thetalin}
\Theta_n=F\Theta_{n-1}+ \gamma \Xi_n+\gamma \lambda \Theta_{\lambda},
\EEQ
with $T=\idm - \gamma \Sigma - \gamma \lambda \idm $,  $F=\begin{pmatrix}
                 (1+\delta)T& -\delta T \\ I & 0
                \end{pmatrix}$, $\Theta_n=\begin{pmatrix} \theta_n  - \theta_\ast \\  \theta_{n-1}  - \theta_\ast \end{pmatrix}$, $\Theta_0= \begin{pmatrix}\theta_0-\theta_*\\ \theta_0-\theta_* \end{pmatrix}$, $\Xi_n= \begin{pmatrix} \xi_n\\ 0 \end{pmatrix}$ and $\Theta_{\lambda}= \begin{pmatrix}\theta_0-\theta_*\\ 0 \end{pmatrix}$.
                
We are interested in the behavior of the average $\bar \Theta_n = \frac{1}{n+1} \sum_{k=0} ^{n} \Theta_k$ for which we have the following general convergence result:

\begin{lemma}\label{lem:3termsacc}
For all $\lambda\in\RR_{+}$ and $\gamma$ such that $\gamma (\Sigma +\lambda \idm)\preccurlyeq \idm$ and any matrix $C$ the average of the iterate $\Theta_{n}$ defined by \eq{thetalin} satisfy for $P_{k}\overset{(def)}{=} C^{1/2} (I- F^{k})(I- F)^{-1}$, with $\tilde \Theta_0 = \Theta_0-\gamma \lambda (I-F)^{-1}\Theta_{\lambda} $,
\BEAS
\E \langle \bar {  \Theta}_n, C \bar {  \Theta}_n\rangle
&\le&2\left(\gamma\lambda\right)^{2 } \| C^{1/2 }(I- F)^{-1}\Theta_{\lambda}  \|^{2}+ \frac{2}{(n+1)^{2}} \|P_{ n+1} \tilde \Theta_0 \|^{2}  \\
&& +\frac{\gamma^2}{(n+1)^2}\sum_{j=1}^n \tr P_{ j}  V P_{ j}^{\top}.
\EEAS
\end{lemma}
Error thus decomposes as the sum of three main terms:
\begin{itemize}
 \item the two first ones are bias terms, one arising from the regularization (the first one), and one arising computation (the second one),
 \item a variance term. which is the last one.
\end{itemize}

We remark that as we have assumed that $\Sigma$ is invertible, the matrix $I-F$ can be shown to be invertible for all the considered $\delta$.

The regularization-based term will be studied directly whereas the two others will be studied in two stages. First an heuristic will lead to an asymptotic bound then an exact computation will give a non-asymptotic bound. Then using $C=H=\begin{pmatrix}\Sigma &0\\0&0\end{pmatrix}$ would give a convergence result on the function value and $C=\begin{pmatrix}\idm &0\\0&0\end{pmatrix}$ a  result on the iterate. The end of the section is devoted to the proof of this lemma.

%     
%  We decompose vectors in an eigenvector basis of $\Sigma$ with $\theta_{n}^{i}=p_{i}^{\top}\theta_{n}$ and $\eps_{n}^{i}=p_{i}^{\top}\eps_{n}$
% % \begin{equation*}
% %  \eta_{n+1}^{i}=(1-\gamma h_{i})\eta_n^{i} +\delta(1-\gamma h_{i})(\eta_n^{i}-\eta_{n-1}^{i})+\gamma \eps_{n+1}^{i}.
% % \end{equation*}
% %  
% %  
% % 
% % We denote by $\Theta^i_n=\begin{pmatrix}
% %                  \eta^i_n\\
% %                  \eta^i_{n-1}
% %                 \end{pmatrix} \mbox{ and }  F_i=\begin{pmatrix}
% %                 (1+\delta)(1-\gamma s_i) &-\delta(1-\gamma s_i) \\
% %                 1 &0 
% %                 \end{pmatrix}\in\mathbb R^{2\times2}$ and $\xi_{n+1}^{i}=\begin{pmatrix}
% %               \gamma \eps_{n+1}^{i}\\
% %               0
% %              \end{pmatrix}$
% and we have the reduced equation:
% \begin{equation*}
%  \Theta_{n+1}^{i}=F_{i}\Theta^i_n+\xi_{n+1}^{i}+\Theta_*^i.
% \end{equation*}with $\Theta_0^{i}= \begin{pmatrix}
%                  \theta_0^{i}\\
%                  \theta_0^{i}
%                 \end{pmatrix}$ and $\theta_0^{i} \in \mathbb{R}$, $F_i=  \begin{pmatrix}(1+\delta)T_i& -\delta T_i \\ 1 & 0
%                 \end{pmatrix}$, with $T_i = 1-\gamma s_i - \gamma \lambda$.
% 
%  
  \begin{proof}
The sequence $ \Theta_n$ satisfies a linear recursion, from which we get, for all $n\geq 1$:
\begin{eqnarray*}
  \Theta_n&=&F^n  \Theta_0+\gamma \sum_{k=1}^n F^{n-k} \Xi_{k}+\gamma \lambda \sum_{k=1}^n F^{n-k} \Theta_{\lambda}\\
&=&F^n  \Theta_0+\gamma \sum_{k=1}^n F^{n-k} \Xi_{k}+\gamma \lambda (I-F^{n})(I-F)^{-1}\Theta_{\lambda}.\\
\end{eqnarray*}
  
We study the  averaged sequence: $\bar \Theta_n = \frac{1}{n+1} \sum_{k=0} ^{n} \Theta_k$ . Using the identity $\sum_{k=0}^{n-1} F^k=(I-F^n)(I-F)^{-1}$, we get
 \BEAS
\bar  {  \Theta}_n&=& \frac{1}{n+1}\sum_{k=0}^n F^{k}   \Theta_0+ \frac{\gamma}{n+1}\sum_{k=1}^n\sum_{j=1}^k F^{k-j}   \Xi_j +\frac{\gamma \lambda }{n+1} \sum_{k=1}^{n} (I-F^{k})(I-F)^{-1}\Theta_{\lambda}.
\EEAS
With $$\tilde \Theta_0 = \Theta_0-\gamma \lambda (I-F)^{-1}\Theta_{\lambda}, $$ and $\sum_{k=1}^{n} (I-F^{k})=\sum_{k=0}^{n} (I-F^{k})=[n+1-(I- F^{n+1})(I- F)^{-1}]$. 

\ \\
Using summation formulas for geometric series, we derive:
\BEAS
\bar  {  \Theta}_n &=&\frac{1}{n+1} ( I- F^{n+1})(I- F)^{-1}  \tilde \Theta_0+\frac{\gamma}{n+1}\sum_{k=1}^n\sum_{j=1}^k F^{k-j} \Xi_j  +\gamma \lambda  (I- F)^{-1}\Theta_{\lambda}\\
&=&\frac{1}{n+1} (I- F^{n+1})(I- F)^{-1}   \tilde\Theta_0+\frac{\gamma}{n+1}\sum_{j=1}^n\big(\sum_{k=j}^n  F^{k-j}\big) \Xi_j+\gamma \lambda  (I- F)^{-1}\Theta_{\lambda} \ \ \text{(inverting sums)}\\
&=&\frac{1}{n+1} (I- F^{n+1})(I- F)^{-1}  \tilde \Theta_0+\frac{\gamma}{n+1}\sum_{j=1}^n\big(\sum_{k=0}^{n-j}  F^{k}\big) \Xi_j+\gamma \lambda  (I- F)^{-1}\Theta_{\lambda}\\
&=&\frac{1}{n+1} (I- F^{n+1})(I- F)^{-1}  \tilde \Theta_0+\frac{\gamma}{n+1}\sum_{j=1}^n(I- F^{n+1-j})(I- F)^{-1} \Xi_j+\gamma \lambda  (I- F)^{-1}\Theta_{\lambda} \\
&=&\frac{1}{n+1} (I- F^{n+1})(I- F)^{-1}  \tilde \Theta_0+\frac{\gamma}{n+1}\sum_{j=1}^n(I- F^{j})(I- F)^{-1} \Xi_j+\gamma \lambda  (I- F)^{-1}\Theta_{\lambda}.
\EEAS
Using martingale square moment inequalities which amount to consider $\Xi_i, i=1,...,n$ independent, so that the variance of the sum is the sum of variances, and denoting by $V=\E[ \Xi_n\otimes \Xi_n]$ we have for any positive semi-definite $C$, 
\BEAS
\E \langle \bar {  \Theta}_n, C \bar {  \Theta}_n\rangle &=& \Big\|C^{1/2} \left(\frac{1}{n+1} (I- F^{n+1})(I- F)^{-1}  \tilde \Theta_0 + \gamma \lambda (I- F)^{-1}\Theta_{\lambda} \right)\Big\|^{2} \\
& & +\frac{\gamma^2}{(n+1)^2}\sum_{j=1}^n \tr (I- F^{j})(I- F)^{-1}   V (I- F^\top)^{-1}(I- F^{j})^\top C,
\EEAS
where $C^{1/2}$ denotes a symmetric square root of $C$.
% If $Q_{i}$ is the transfer matrix into an eigenbasis of $F_{i}$, i.e., $F_{i}=Q_{i}D_{i}Q_{i}^{-1}$ with $D_i$ triangular, we have $(I- F_i^{n})(I- F_i)^{-1}  = Q_i (I-D_i^{n+1})(I-D_i)^{-1}Q_{i}^{-1} $. $P_{i, n}\overset{(def)}{=} C^{1/2} Q_i (I-D_i^{n})(I-D_i)^{-1}Q_{i}^{-1} $
Define   $P_{k}\overset{(def)}{=} C^{1/2} (I- F^{k})(I- F)^{-1} $, we have, Using Minkowski's inequality and inequality $(a+b)^{2}\le 2(a^2 +  b^2)$ for any $a,b\in \RR$, 
\BEAS
\E \langle \bar {  \Theta}_n, C \bar {  \Theta}_n\rangle &=&  \Big\|\frac{1}{n+1} P_{ n+1} \tilde \Theta_0 + \gamma \lambda C^{1/2 }(I- F)^{-1}\Theta_{\lambda}  \Big\|^{2} +\frac{\gamma^{2}}{(n+1)^2}\sum_{j=1}^n \tr P_{ j}  V P_{ j}^{\top} \\
&\le&2\left(\gamma\lambda\right)^{2 } \| C^{1/2 }(I- F)^{-1}\Theta_{\lambda}  \|^{2}+ \frac{2\|P_{ n+1} \tilde \Theta_0 \|^{2} }{(n+1)^{2}}  +\frac{\gamma^2}{(n+1)^2}\sum_{j=1}^n \tr P_{ j}  V P_{ j}^{\top}. 
\EEAS
Which concludes proof of Lemma~\ref{lem:3termsacc}.
\end{proof}

 \subsection{Asymptotic expansion}\label{app:b2}
To give the main terms that we expect, we first provide an asymptotic analysis, which shall only be understood as an insight and is not necessary for the proof.  Operator $F$ will have only eigenvalues smaller than 1, thus $\vertiii{ F^{j} } $ will decrease exponentially  to $0$ as $j\rightarrow \infty$ (even if $\vertiii{F}$\footnote{$\vertiii{F}$ denotes the operator norm of $F$, i.e., $\sup _{\|x\|\le 1} \|Fx\|.$} might be bigger than 1). The asymptotic analysis relies on ignoring all terms in which $F^j$ appears.
We  thus approximately have:
\BEAS
\E \langle \bar {  \Theta}_n, C \bar {  \Theta}_n\rangle &\leq &2\left(\gamma\lambda\right)^{2 } \| C^{1/2 }(I- F)^{-1}\Theta_{\lambda}  \|^{2}+2 \Big\|C^{1/2}\frac{1}{n+1} (I- F^{n+1})(I- F)^{-1}  \tilde \Theta_0 \Big\|^{2} \\
& & +\frac{\gamma^2}{(n+1)^2}\sum_{j=1}^n \tr (I- F^{j})(I- F)^{-1}   V (I- F^\top)^{-1}(I- F^{j})^\top C\\
&\approx&2\left(\gamma\lambda\right)^{2 } \| C^{1/2 }(I- F)^{-1}\Theta_{\lambda} \|^{2}+2 \Big\|C^{1/2}\frac{1}{n+1}(I- F)^{-1}  \tilde \Theta_0 \Big\|^{2}\\
&& +\frac{\gamma^2}{(n+1)^2}\sum_{j=1}^n \tr (I- F)^{-1}   V (I- F^\top)^{-1} C 
\EEAS
 where, as it has been explained $\approx$ stands for an equality up to terms that will decay exponentially. However, these terms have to be studied very carefully, what will be done in the Section~\ref{subsec:direct}.
 
Using the matrix inversion lemma we have for $C=\begin{pmatrix}
c& 0 \\ 0 & 0 
\end{pmatrix}$,
 \BEA
 \idm-F&=&  \begin{pmatrix}(1+\delta)(\gamma \Sigma+\gamma\lambda \idm) - \delta \idm & \delta (\idm-(\gamma \Sigma+\gamma\lambda\idm)) \\ -\idm & \idm
                \end{pmatrix} \nonumber \\
 (\idm-F)^{-1}&=&\begin{pmatrix}
                (\gamma \Sigma+ \gamma \lambda \idm)^{-1}& \delta \big(\idm -(\gamma \Sigma+ \gamma \lambda \idm)^{-1}\big)\\ (\gamma \Sigma+\gamma \lambda \idm)^{-1}& (1+\delta)\idm - \delta(\gamma \Sigma+ \gamma \lambda \idm)^{-1}
               \end{pmatrix} \label{eq:invImF} \\ \nonumber 
 C^{1/2}(\idm-F)^{-1}  &=&\begin{pmatrix}
 c^{1/2}(\gamma \Sigma+\gamma\lambda \idm)^{-1} & \delta c^{1/2} \big(\idm -(\gamma \Sigma+ \gamma \lambda \idm)^{-1}\big)\\ \nonumber 0& 0
  \end{pmatrix}.
 \EEA
 \paragraph{Regularization based term.}
 This gives for the regularization based term
\begin{eqnarray}
   \Big\| C^{1/2 }(\idm- F)^{-1}\Theta_{\lambda}  \Big\|^{2} & = &   \Bigg\| \begin{pmatrix}
 c^{1/2}(\gamma \Sigma+\gamma\lambda \idm)^{-1} & \delta c^{1/2} \big(\idm -(\gamma \Sigma+ \gamma \lambda \idm)^{-1}\big)\\ 0& 0
  \end{pmatrix} \begin{pmatrix}
  \theta_{0}-\theta_{*} \\ 0
  \end{pmatrix}  \Bigg\|^{2} \nonumber\\  & =&  \left(\frac{1}{\gamma}\right)^{2 }    \| (c^{1/2} (\Sigma+\lambda I)^{-1} (\theta_{0}-\theta_{*})) \|^{2}. \label{eq:regbased}
\end{eqnarray}
The computation of this term is exact (not asymptotic).

\paragraph{Bias term.}
For the bias term we have 
\BEAS
\tilde \Theta_0 &=& \Theta_0-\gamma \lambda (I-F)^{-1}\Theta_{\lambda} \\
&=& \begin{pmatrix}
\theta_{0}-\theta_{*}\\ \theta_{0}-\theta_{*}\end{pmatrix} -\gamma \lambda \begin{pmatrix}
                (\gamma \Sigma+ \gamma \lambda \idm)^{-1}& \delta \big(\idm -(\gamma \Sigma+ \gamma \lambda \idm)^{-1}\big)\\ (\gamma \Sigma+\gamma \lambda \idm)^{-1}& (1+\delta)\idm - \delta(\gamma \Sigma+ \gamma \lambda \idm)^{-1}
               \end{pmatrix}\begin{pmatrix}
\theta_{0}-\theta_{*}\\0\end{pmatrix} \\
&=& \begin{pmatrix}
\theta_{0}-\theta_{*}\\ \theta_{0}-\theta_{*}\end{pmatrix} -\gamma \lambda  \begin{pmatrix}    (\gamma \Sigma+ \gamma \lambda \idm)^{-1} (\theta_{0}-\theta_{*}) \\(\gamma \Sigma+ \gamma \lambda \idm)^{-1}( \theta_{0}-\theta_{*})
\end{pmatrix}\\
 &=&\begin{pmatrix}   [ \idm - \lambda ( \Sigma + \lambda \idm )^{-1} ] (\theta_{0}-\theta_{*}) \\ [\idm-\lambda (\Sigma+ \lambda \idm)^{-1}](\theta_{0}-\theta_{*})
\end{pmatrix}.
\EEAS
Thus this gives for the dominant term
 \begin{eqnarray*}
   \Big\| C^{1/2 }(\idm- F)^{-1}\tilde\Theta_{0}  \Big\|^{2} & = &   \Bigg\| \begin{pmatrix}
 c^{1/2}(\gamma \Sigma+\gamma\lambda \idm)^{-1} & \delta c^{1/2} \big(\idm -(\gamma \Sigma+ \gamma \lambda \idm)^{-1}\big)\\ 0& 0
  \end{pmatrix} \tilde\Theta_{0}\Bigg\|^{2}\\  
  & =&      \| (c^{1/2}[(1-\delta)(\gamma \Sigma+\gamma\lambda \idm)^{-1}+\delta \idm] [\idm-\lambda (\Sigma+ \lambda \idm)^{-1}](\theta_{0}-\theta_{*}) \|^{2}.
\end{eqnarray*}
And if $c$ commutes with $\Sigma$ we have the bound for $\delta\in[\frac{1-\sqrt{\gamma\lambda}}{1+\sqrt{\gamma\lambda}},1]$
\begin{eqnarray*}
\Big\| C^{1/2 }(\idm- F)^{-1}\tilde\Theta_{0}  \Big\|^{2} & \leq &(\frac{(1-\delta)}{\gamma \lambda}+\delta) \| (c^{1/2} [\idm-\lambda (\Sigma+ \lambda \idm)^{-1}](\theta_{0}-\theta_{*}) \|^{2}\\
 & \leq &(\frac{2}{\sqrt{\gamma \lambda}}+1) \| (c^{1/2} [\idm-\lambda (\Sigma+ \lambda \idm)^{-1}](\theta_{0}-\theta_{*}) \|^{2}.
\end{eqnarray*}

\paragraph{Variance term.}
And for the variance term with $V=\begin{pmatrix}
v&0\\0&0
\end{pmatrix}$, we have
$ C^{1/2}(\idm-F)^{-1}V^{1/2}  =\begin{pmatrix}
 c^{1/2}(\gamma \Sigma+\gamma\lambda \idm)^{-1} v^{1/2} &0\\ 0& 0
  \end{pmatrix}$, and 
$$
\tr C^{1/2}(I- F)^{-1}   V (I- F^\top)^{-1} C^{1/2}=\tr  c(\gamma \Sigma+\gamma\lambda \idm)^{-1} v  (\gamma \Sigma+\gamma\lambda \idm)^{-1}.
$$

This gives the three dominant terms. However in order to control the remainders we have to compute the eigenvalues more carefully, as done in the next section. 

\subsection{Direct computation without the regularization based term} \label{subsec:direct}
The computation of the regularization based term being exact we derive now direct computation for the remainders.
Following \citet{o2013adaptive} we consider an eigen-decomposition of the matrix $F$, in order to study independently the recursion on eigenspaces.
We assume $\Sigma$ has eigenvalues $(s_i)$ and we decompose vectors in an eigenvector basis of $\Sigma$ with $\theta_{n}^{i}=p_{i}^{\top}\theta_{n}$ and $\xi_{n}^{i}=p_{i}^{\top}\xi_{n}$
and we have the reduced equation:
\begin{equation*}
\Theta_{n+1}^{i}=F_{i}\Theta^i_n+\gamma \Xi_{n+1}^{i}.
\end{equation*}with $\Theta_0^{i}=\tilde \Theta_0^{i} $, $F_i=  \begin{pmatrix}(1+\delta)T_i& -\delta T_i \\ 1 & 0
                \end{pmatrix}$, with $T_i = 1-\gamma s_i - \gamma \lambda$.

\paragraph{Computing initial point $\tilde\Theta_0^{i}$. }
$\tilde \Theta_0^{i} = \Theta_0^{i} - \gamma \lambda (I-F_i)^{-1}\Theta^{i}_{\lambda}$, with $\Theta_0^{i}=\begin{pmatrix}
\theta_0^i-\theta_*^i\\ \theta_0^i-\theta_*^i
\end{pmatrix}$,  $\Theta^{i}_{\lambda}=\begin{pmatrix}  \theta^{i}_0- \theta^{i}_*\\ 0 \end{pmatrix}$
and $(I-F_i)^{-1}$ given in \eq{invImF}. Thus 
\BEA
\tilde \Theta_0^{i} &= &\begin{pmatrix}
 \theta^{i}_0- \theta^{i}_* \\  \theta^{i}_0- \theta^{i}_*\end{pmatrix} - \frac{\gamma \lambda}{	(\gamma s_i+\gamma\lambda)}\begin{pmatrix}1 & \delta (1-(\gamma s_i+\gamma\lambda)) \\ 1 & -(1+\delta)(\gamma s_i+\gamma\lambda) - \delta\end{pmatrix} \begin{pmatrix}  \theta^{i}_0- \theta^{i}_*\\ 0 \end{pmatrix}\nonumber \\  &= &\begin{pmatrix}
(1- \frac{\lambda}{\lambda+s_i} )(\theta_0^i-\theta^{i}_* )  \\ (1- \frac{\lambda}{\lambda+s_i} )(\theta_0^i-\theta^{i}_* )\end{pmatrix}. \label{eq:initialpoint}
\EEA
\paragraph{Study of  spectrum of $F_i$.}

Depending on $\delta$, $F_i$ may have two distinct complex eigenvalues of same modulus, only one (double) eigenvalue, or two real eigenvalues. We only consider the two former cases, which we detail bellow.

 Indeed, the characteristic polynomial $$\chi_{F_i}(X)\overset{def}{=} \det (X\idm- F_i) = X^{2}-  (1+\delta )(1-\gamma  (s_i+\lambda)) X +\delta(1-\gamma   (s_i+\lambda))$$
 has discriminant $\Delta_i = (1-\gamma (s_i+\lambda)) ((1+\delta )^{2}(1-\gamma (s_i+\lambda)) - 4 \delta)$ which is non positive as far as $\delta \in [\delta_{-}; \delta_{+}]$, with $\delta_{-} = \frac{1-\sqrt{\gamma (s_i+\lambda)}}{1+\sqrt{\gamma (s_i+\lambda)}}$, $\delta_{+} = \frac{1+\sqrt{\gamma (s_i+\lambda)}}{1-\sqrt{\gamma (s_i+\lambda)}}$.
 
  \subsubsection{Two distinct eigenvalues}

We first assume that $F_i$ has two distinct complex eigenvalues $r_{\pm}=\frac{(1+\delta)(1-\gamma (s_i+\lambda))\pm \sqrt{-1}\sqrt{-\Delta_i}}{2}$ which are conjugate. Thus the roots are of the form $\rho_i e^{\pm i \omega_i}$ with $\rho_i=\sqrt{\delta(1-\gamma (s_i+\lambda))}$, $\cos(\omega_i)=\frac{(1+\delta)(1-\gamma (s_i+\lambda))}{2\rho_i}
$, $\omega_{i}\in[-\pi/2;\pi/2]$ and $\sin(\omega_i)=\frac{\sqrt{-\Delta_i}}{2\rho_i}$.

Let $Q_{i}=\begin{pmatrix}
             r_{i}^{-} &r_{i}^{+} \\
              1&1
             \end{pmatrix}$ be the transfer matrix into an eigenbasis of $F_{i}$, i.e., $F_{i}=Q_{i}D_{i}Q_{i}^{-1}$ with \mbox{$D_i=\begin{pmatrix}
             r_{i}^{-} &0 \\
              0&r_{i}^{+}
             \end{pmatrix}$}
             and 
             \mbox{$Q_{i}^{-1}=\frac{1}{r_{i}^{-}-r_{i}^{+}}\begin{pmatrix}
             1 &-r_{i}^{+} \\
              -1&r_{i}^{-}
             \end{pmatrix}$}.
\paragraph{Computing $P_{i,k}$.}

We first compute the matrix $P_{i,k}$: 
With $$C_i^{1/2}=\begin{pmatrix}
         \sqrt{ c_{i}}& 0 \\0&0
          \end{pmatrix}, C_i^{1/2}Q_i=  \begin{pmatrix}
          r_{i}^{-} \sqrt{c_i}&r_{i}^{+}\sqrt{c_i}\\0&0
          \end{pmatrix} $$ we have
  $$
   C_i^{1/2}Q_i(I-D_i^{k})(I-D_i)^{-1}=\sqrt{c_i}\begin{pmatrix}
                                         \frac{1- (r_{i}^{-})^k}{1-r_{i}^{-}}r_{i}^{-}&\frac{1- (r_{i}^{+})^k}{1-r_{i}^{+}}r_{i}^{+}\\
0&0
                                        \end{pmatrix},
  $$
and, when developing and regrouping terms which depend on $k$, we get : 
\BEAS
  P_{i,k} &=& C_i^{1/2}Q_i(I-D_i^{k})(I-D_i)^{-1}Q_i^{-1}\\ &=& \frac{\sqrt{c_i}}{r_{i}^{-}-r_{i}^{+}} 
                                        \begin{pmatrix}\frac{1-(r_{i}^{-})^k}{1-r_{i}^{-}}r_{i}^{-}-\frac{1-(r_{i}^{+})^k}{1-r_{i}^{+}}r_{i}^{+} & \frac{1-(r_{i}^{+})^k}{1-r_{i}^{+}}r_{i}^{-}r_{i}^{+} -\frac{1-(r_{i}^{-})^k}{1-r_{i}^{-}}r_{i}^{+}r_{i}^{-}\\
                                         0&0
                                        \end{pmatrix}   \\
                                        &=&
                                       {\sqrt{c_i}}
                                        \begin{pmatrix}\frac{1}{(1-r_{i}^{-})(1-r_{i}^{+})} & \frac{-r_{i}^{+}r_{i}^{-}}{(1-r_{i}^{-})(1-r_{i}^{+})}\\
                                         0&0
                                        \end{pmatrix} \\
                                        &&- \frac{\sqrt{c_i}}{r_{i}^{-}-r_{i}^{+}} 
                                        \begin{pmatrix}\frac{(r_{i}^{-})^{k+1}}{1-r_{i}^{-}}-\frac{(r_{i}^{+})^{k+1}}{1-r_{i}^{+}} & \frac{(r_{i}^{+})^{k+1}}{1-r_{i}^{+}}r_{i}^{-} -\frac{(r_{i}^{-})^{k+1}}{1-r_{i}^{-}}r_{i}^{+}\\
                                         0&0
                                        \end{pmatrix}.  
 \EEAS
 We also have 
$P_{i,k}=C_i^{1/2}Q_i(I-D_i^{k})(I-D_i)^{-1}Q_i^{-1}
= \sum_{j=0}^{k-1}R_{i,j}$
 with 
 \BEAS R_{i,j}
 &=& C_i^{1/2}Q_iD_i^{j} Q_i^{-1} \\
 &=& \sqrt{c_i}\begin{pmatrix}
                                         (r_{i}^{-}) ^{j+1}&(r_{i}^{+})^{j+1}\\
0&0   \end{pmatrix}Q_i^{-1}\\
&=& \frac{\sqrt{s_i}}{r_{i}^{-}-r_{i}^{+}} \begin{pmatrix}
         (r_{i}^{-})^{j+1}-(r_{i}^{+})^{j+1}&-r_{i}^{+}(r_{i}^{-})^{j+1}+r_{i}^{-}(r_{i}^{+})^{j+1}\\0&0
          \end{pmatrix},
  \EEAS
 but computing error terms based in $R_{i,j}$ before summing these errors gives a looser error bound than a tight calculation using $P_{i,k}$. More precisely, if we use  $P_{{i,k}}\Theta^{i}_0= \sum_{j=0}^{k-1}R_{i,j}\Theta^{i}_0$ to upper bound $\Vert P_{{i,k}}\Theta^{i}_0 \Vert \le \sum_{j=0}^{k-1} \Vert R_{i,j}\Theta^{i}_0 \Vert$, we end up with a worse bound.

 \paragraph{Bias term.}
 Thus, for the bias term: 
 \BEAS
 P_{{i,k}}\Theta^{i}_0
                                        &=&
                                       {\sqrt{c_i}\theta^i_0}\frac{1-r_{i}^{+}r_{i}^{-}}{(1-r_{i}^{-})(1-r_{i}^{+})}
- \frac{\sqrt{c_i}\theta^i_0}{r_{i}^{-}-r_{i}^{+}} \begin{pmatrix}\Big[(r_{i}^{-})^{k+1}\frac{1-r_{i}^{+}}{1-r_{i}^{-}}-(r_{i}^{+})^{k+1}\frac{1-r_{i}^{-}}{1-r_{i}^{+}}\Big]\\0\end{pmatrix}\\     
&=& \frac{\sqrt{c_i}\theta^i_0}{\sqrt{(1-r_{i}^{-})(1-r_{i}^{+})}}\begin{pmatrix}\frac{\big[ ({1-r_{i}^{+}r_{i}^{-}})- \rho_i^{k}A_1\big]}{\sqrt{(1-r_{i}^{-})(1-r_{i}^{+})}}\\0\end{pmatrix},
 \EEAS
 where 
$$
 \rho_i^{k}A_1=\frac{(r_{i}^{-})^{k+1}(1-r_{i}^{+})^2-(r_{i}^{+})^{k+1}(1-r_{i}^{-})^2}{r_{i}^{-}-r_{i}^{+}}.
$$
 This can be bound with the following lemma 
 \begin{lemma} For all  $\rho\in(0,1)$ and $\omega \in[-\pi/2;\pi/2]$ and $r^{\pm}=\rho(\cos(\omega)\pm\sqrt{-1}\sin(\omega))$ we have:
 \begin{equation}
 \bigg|\frac{1-r^{+}r^{-}-\rho^{k} |A_1|}{|1-r^{+}|} \bigg| \le 3+3\rho^{k} \leq 6
 \end{equation}
 \end{lemma}
 We note that the exact constant seems empirically to be $2$. This lemma is  proved as Lemma~\ref{lem:maj1} in Appendix~\ref{app:techlem}.
This gives for the bias term 
 \BEAS
  \Vert P_{i,k}\Theta^{i}_0\Vert   &=& \frac{\sqrt{c_i}(\theta^i_0)}{\sqrt{(1-r_{i}^{-})(1-r_{i}^{+})}}\big[ \frac{1}{\sqrt{(1-r_{i}^{-})(1-r_{i}^{+})}}\left(({1-r_{i}^{+}r_{i}^{-}})- \rho_i^{k}A_1\right)\big]\\
  &\le&6 \frac{\sqrt{c_i}(\theta^i_0)}{\sqrt{\gamma (s_i+\lambda)}},
  \EEAS
since: 
 \BEAS
 (1-r_{i}^{-})(1-r_{i}^{+})&=&1-2\,\mathfrak{Re}\,( r_{i}^{+})+\vert r_{i}^{+} \vert^2\\
 &=&1-(1+\delta)(1-\gamma (s_i+\lambda))+\delta(1-\gamma (s_i+\lambda))\\
&=&\gamma (s_i+\lambda).
 \EEAS
 We also have a looser bound using $P_{{i,k}}\Theta^{i}_0= \sum_{j=0}^{k-1}R_{i,j}\Theta^{i}_0$. 
\BEAS
R_{i,j}\Theta^{i}_0
&=&
\frac{\sqrt{c_i}\theta^{i}_0}{r_{i}^{-}-r_{i}^{+}} \big((1-r_{i}^{+})(r_{i}^{-})^{j+1}-(1-r_{i}^{-})(r_{i}^{+})^{j+1}\big)\\
&=&\sqrt{c_i}\theta^{i}_0\bigg( \frac{(r_{i}^{-})^{j+1}-(r_{i}^{+})^{j+1}}{r_{i}^{-}-r_{i}^{+}}- \frac{r_{i}^{+}(r_{i}^{-})^{j+1}-r_{i}^{-}(r_{i}^{+})^{j+1}}{r_{i}^{-}-r_{i}^{+}}\bigg) \text{ using De Moivre's formula,}	\\
&=&\sqrt{c_i}\theta^{i}_0\bigg( \frac{\rho_i^{j+1}\sin(\omega_i (j+1)) }{\rho_i\sin(\omega_i)}- \frac{\rho_ie^{i\omega_i}\rho_i^{j+1}e^{-i\omega_i(j+1)}-\rho_ie^{-i\omega_i}\rho_i^{j+1}e^{+i\omega_i(j+1)}}{\rho_ie^{-i\omega_i}-\rho_ie^{i\omega_i}}\bigg)\\
&=&\sqrt{c_i}\theta^{i}_0\bigg( \frac{\rho_i^{j+1}\sin(\omega_i (j+1)) }{\rho_i\sin(\omega_i)}- \rho_i^{j+1}\frac{e^{-i\omega_ij}-e^{+i\omega_ij}}{e^{-i\omega_i}-e^{i\omega_i}}\bigg)\\
&=&\sqrt{c_i}\theta^{i}_0\bigg( \frac{\rho_i^{j}\sin(\omega_i (j+1)) }{\sin(\omega_i)}- \rho_i^{j+1}\frac{\sin(\omega_ij)}{\sin(\omega_i)}\bigg)\\
&\leq&  (1+e^{-1}) \sqrt{c_i}\theta^{i}_0 \quad \text{ using Lemma~\ref{lem:majcde} (see proof in Appendix~\ref{app:techlem})},
\EEAS
which also gives for the bias term 
$$
\Vert P_{{i,k}}\Theta^{i}_0 \Vert \leq  (1+e^{-1}) \sqrt{c_i}\theta^{i}_0  k.
$$
Thus we have the final bound: 
\begin{eqnarray}
\Vert P_{{i,k}}\Theta^{i}_0 \Vert^{2} \leq \min \left\lbrace 36 \frac{{c_i}(\theta^i_0)^{2}}{\gamma (s_i+\lambda)}, 6n (1+e^{-1} )\frac{{c_i}(\theta^i_0)^{2}}{\sqrt{\gamma (s_i+\lambda)}}, n^{2}(1+e^{-1})^{2}{c_i}(\theta^i_0)^{2}  \right\rbrace.
\label{eq:disctinc_bias}
\end{eqnarray}

\paragraph{Variance term.}
As for the variance term, with  $V_i=\begin{pmatrix}
              v_i&0\\ 0 &0
             \end{pmatrix}$, we have
 $
 \tr P_{i,k }V_iP_{i,k} = \Big\Vert  P_{i,k } \begin{pmatrix}
              \sqrt{v_i}\\ 0
             \end{pmatrix}\Big\Vert^2
 $.
 
\BEAS
\Big\Vert  P_{i,k } \begin{pmatrix}
              \sqrt{v_i}\\ 0
             \end{pmatrix}\Big\Vert&=& \frac{\sqrt{v_ic_i}}{(1-r_{i}^{-})(1-r_{i}^{+})}\Bigg[ 1+\frac{(r_{i}^{-})^{k+1}(1-r_{i}^{+})-(r_{i}^{+})^{k+1}(1-r_{i}^{-})}{r_{i}^{+}-r_{i}^{-}}\Bigg]\\
                                        &=& \frac{\sqrt{v_ic_i}}{\gamma (s_i+\lambda)}\Bigg[ 1-\rho_i^{k}B_{i,k}\Bigg],
\EEAS
where 
\BEAS
\rho_i^{k}B_{i,k}&=&-\frac{(r_{i}^{-})^{k+1}(1-r_{i}^{+})-(r_{i}^{+})^{k+1}(1-r_{i}^{-})}{r_{i}^{+}-r_{i}^{-}},
\EEAS
which we can bound using the following Lemma:
 \begin{lemma} For all  $\rho\in(0,1)$ and $\omega \in[-\pi/2;\pi/2]$ and $r^{\pm}=\rho(\cos(\omega)\pm\sqrt{-1}\sin(\omega))$ we have:
 \begin{equation*}
\bigg| \rho^{k}B_{k} \bigg| \leq  1.75.
 \end{equation*}
 \end{lemma}
Where we note that the exact majoration seems to be 1.3. This Lemma is proved as Lemma~\ref{lemmajvar} in Appendix~\ref{app:techlem}.

We can also have a looser bound using 
  $P_{{i,k}}\begin{pmatrix}
             v_i^{1/2}\\0
            \end{pmatrix}
= \sum_{j=0}^{k-1}R_{i,j}\begin{pmatrix}
             v_i^{1/2}\\0
            \end{pmatrix}$ and  

\BEAS
R_{i,j}\begin{pmatrix}
             v_i^{1/2}\\0
            \end{pmatrix}
&=&
\frac{\sqrt{c_i v_i}}{r_{i}^{-}-r_{i}^{+}} \ \ \big((r_{i}^{-})^{j+1}-(r_{i}^{+})^{j+1}\big)\\
&=&\sqrt{c_i v_i} \ \ \frac{\rho_i^{j+1}\sin(\omega_i (j+1)) }{\rho_i\sin(\omega_i)}\\
&\leq& (j+1)  \sqrt{c_i v_i} , \text{ using the inequality }  |\sin(k\omega_i)| \le k|\sin(\omega_i)|
\EEAS
and $\big\Vert P_{{i,k}}\begin{pmatrix}
             v_i^{1/2}\\0
            \end{pmatrix} \big\Vert\leq \frac{\sqrt{c_i v_i}(k+1)k}{2}$.\\
And this gives for the Variance term
 \BEA
\sum_{k=1}^n \tr  P_{i,k }V_iP_{i,k}
&\leq&{v_i c_i }\sum_{k=1}^n \min\Bigg\{ \frac{\big[ 1-\rho_i^{k} B_{1,k}\big]^2}{{\gamma^2 (s_i+\lambda)^{2}}}, \frac{\big[ 1-\rho_i^{k} B_{1,k}\big]k(k+1)}{2{\gamma (s_i+\lambda)}}, \frac{k^2(k+1)^2}{4} \Bigg\} \nonumber\\
&\le & {v_i c_i } \min\Bigg\{ \frac{8n}{{\gamma^2 (s_i+\lambda)^{2}}}, \frac{(n+1)^3}{2{\gamma (s_i+\lambda)}}, \frac{(n+1)^5}{20} \Bigg\}. \label{eq:disctinc_var}
\EEA

%\begin{lemma}\label{lem_disctint_eigen}
%Finally:
%\BEAS
%\E \langle \bar { \Theta}^{i}_n, C \bar { \Theta}^{i}_n\rangle &\le & K_2 \frac{s_i(\theta^i_0)^2}{(s_i+\lambda)\gamma (n+1)^2} 
%+ K_1 \frac{v s_i n}{\gamma^2 (s_i+\lambda)^{2} (n+1)^{2}} 
%\EEAS
%With $K_1 = 2.75^{2}\le 8, K_2=(3+3\rho^{k})^{2}\le 36$.
%\end{lemma}

\subsubsection{One coalescent eigenvalue}
We now turn to the case where $F$ has two coalescent eigenvalues, which happens when the discriminant $\Delta=0$. We assume that $F_i$ has one coalescent eigenvalue $r_i=\frac{(1+\delta)(1-\gamma (s_i+\lambda))}{2}$. Then,  with $\delta= \frac{1-\sqrt{\gamma (s_i+\lambda)}}{1+\sqrt{\gamma (s_i+\lambda)}}$, $r_i=\frac{(1+\delta)(1-\gamma (s_i+\lambda))}{2} = 1 -\sqrt{\gamma (s_i+\lambda)}$.  Then $F_i$ can be trigonalized as $F_{i}=Q_{i}D_{i}Q_{i}^{-1}$ with   $Q_{i}=\begin{pmatrix}
             r_{i} &1 \\
              1&0
             \end{pmatrix}$, $D_i=\begin{pmatrix}
             r_{i} &1 \\
              0&r_{i}
             \end{pmatrix}$
             and 
             \mbox{$Q_{i}^{-1}=\begin{pmatrix}
             0 &1 \\
              1&-r_{i}
             \end{pmatrix}$}.
 We note that for all $k\geq0$, then $D_i^k=r_i^{k-1}\begin{pmatrix}
                                                      r_i&k\\0&r_i
                                                     \end{pmatrix}$.
                                                     
\paragraph{Computing $P_{i,k}$.}
 We first compute $P_{i,k}$:  
$$(I_2-D_i)^{-1}=\begin{pmatrix}
                      \frac{1}{1-r_i}&\frac{1}{(1-r_i)^2}\\ 0&\frac{1}{1-r_i}
                     \end{pmatrix}$$ and 
                     $$(I_2-D_i^k)(I_2-D_i)^{-1}=\begin{pmatrix}
                                                      \frac{1-r_i^k}{1-r_i}&\frac{1-r_i^k}{(1-r_i)^2}-\frac{k r_i^{k-1}}{1-r_i}\\
                                                      0& \frac{1-r_i^k}{1-r_i}
                                                 \end{pmatrix}.$$                   
Thus with $C_i^{1/2} Q_i=\begin{pmatrix}
                \sqrt{c_i} r_i &\sqrt{c_i}\\0&0
               \end{pmatrix}$ we have 
 $$C_i^{1/2} Q_i(I_2-D_i^k)(I_2-D_i)^{-1}=\sqrt{c_i}\begin{pmatrix}
                    \frac{1-r_i^k}{1-r_i}r_i &\frac{1-r_i^k}{(1-r_i)^2}-\frac{k r_i^{k}}{1-r_i}\\0&0
                    \end{pmatrix}.
$$              
And, computing as previously the matrices products, we derive:
\BEAS
P_{i,k} &=&C_i^{1/2} Q_i(I_2-D_i^k)(I_2-D_i)^{-1}Q_{i}^{-1}\\&=&\sqrt{c_i}
\begin{pmatrix}
 \frac{1-r_i^k}{(1-r_i)^2}-\frac{k r_i^{k}}{1-r_i}& \frac{1-r_i^k}{1-r_i}r_i-(\frac{1-r_i^k}{(1-r_i)^2}-\frac{k r_i^{k}}{1-r_i})r_i\\
 0&0
\end{pmatrix}\\
&=&\sqrt{c_i}\begin{pmatrix}
 \frac{1-r_i^k}{(1-r_i)^2}-\frac{k r_i^{k}}{1-r_i}& \frac{1-r_i^k}{(1-r_i)^2}(r_i)^2+\frac{k r_i^{k+1}}{1-r_i}\\
 0&0
\end{pmatrix}\\
&=&\frac{\sqrt{c_i}}{1-r_i}\begin{pmatrix}
 \frac{1-r_i^k}{1-r_i}-{k r_i^{k}}& -\frac{1-r_i^k}{1-r_i}(r_i)^2+{k r_i^{k+1}}\\
 0&0
\end{pmatrix}.
\EEAS
\paragraph{Bias term.}
We thus have:\BEAS
P_{i,k} \Theta^{i}_0
&=&\frac{\sqrt{c_i}}{1-r_i}\begin{pmatrix}
 \frac{1-r_i^k}{1-r_i}-{k r_i^{k}}& -\frac{1-r_i^k}{1-r_i}(r_i)^2+{k r_i^{k+1}}\\
 0&0
\end{pmatrix}\begin{pmatrix}
               \theta^i_0 \\ \theta_0^i
	     \end{pmatrix}\\
&=&{\theta^i_0 \sqrt{c_i}}\begin{pmatrix}
                           (1-r_i^k)\frac{1+r_i}{1-r_i}-kr_i^k \\0
                          \end{pmatrix}.
\EEAS                                                
and this gives for the bias term:
\BEA
  \Vert P_{i,k}\Theta^{i}_0\Vert^2 
    &=& (\theta^i_0)^2 c_i\Big[ (1-r_i^k)\frac{1+r_i}{1-r_i}-kr_i^k\Big]^2\nonumber\\
    &=& (\theta^i_0)^2 c_i\Big[ \frac{1+r_i}{1-r_i}-\Big(k+\frac{1+r_i}{1-r_i}\Big) r_i^k\Big]^2 \text{ developing the product, then using formulas for } r_i\nonumber\\
    &=& (\theta^i_0)^2 c_i\Big[ \frac{2-\sqrt{\gamma (s_i +\lambda)}}{\sqrt{\gamma (s_i +\lambda)}}-\Big(k+\frac{2-\sqrt{\gamma (s_i +\lambda)}}{\sqrt{\gamma (c_i +\lambda)}}\Big) (1-\sqrt{\gamma (s_i +\lambda)})^k\Big]^2\nonumber\\
    &=& \frac{(\theta^i_0)^2 c_i}{\gamma (s_i+\lambda)} \Big[ {2-\sqrt{\gamma (s_i +\lambda)}}-\big(k \sqrt{\gamma (s_i +\lambda)}+{2-\sqrt{\gamma (s_i +\lambda)}}\big) (1-\sqrt{\gamma (s_i +\lambda)})^k\Big]^2\nonumber\\
        &=& \frac{(\theta^i_0)^2c_i}{\gamma (s_i+\lambda)} \Big[ {2-\sqrt{\gamma (s_i +\lambda)}}-\big(2+(k-1) \sqrt{\gamma (s_i +\lambda)}\big) (1-\sqrt{\gamma (s_i +\lambda)})^k\Big]^2\nonumber\\
         &\le& 4 \frac{(\theta^i_0)^2 c_i}{\gamma (s_i+\lambda)}, \text{ using Lemma~\ref{lem:techtriv} in Appendix~\ref{app:techlem}}. \label{eq:coal_bias}
 \EEA

 \paragraph{Variance term.} 
 With  $V=\begin{pmatrix}
              v_i&0\\ 0 &0
             \end{pmatrix}$,
\BEAS
 &&\tr P_{i,k}V P_{i,k}\\ &=&\frac{{s_i}}{(1-r_i)^2}\begin{pmatrix}
 \frac{1-r_i^k}{1-r_i}-{k r_i^{k}}& -\frac{1-r_i^k}{1-r_i}(r_i)^2+{k r_i^{k+1}}\\
 0&0
\end{pmatrix} \begin{pmatrix} v_i&0\\0&0 \end{pmatrix}
\begin{pmatrix}
 \frac{1-r_i^k}{1-r_i}-{k r_i^{k}}& -\frac{1-r_i^k}{1-r_i}(r_i)^2+{k r_i^{k+1}}\\
 0&0
\end{pmatrix}^{\top}\\
&=&\frac{{s_i}v_i}{(1-r_i)^2} \Big[ \frac{1-r_i^k}{1-r_i}-{k r_i^{k}}\Big]^2\\
&=&\frac{v_i h_{i}}{\gamma (s_i +\lambda)} \Big[ \frac{1-r_i^k}{1-r_i}-{k r_i^{k}}\Big]^2\\
&=&\frac{v_ih_{i}}{\gamma (s_i +\lambda)(1-r_i)^2} \Big[ {1-r_i^k}-(1-r_i){k r_i^{k}}\Big]^2\\
&=&\frac{v_ih_{i}}{\gamma^2 (s_i +\lambda)^{2} } \Big[ 1- (1+k\sqrt{\gamma (s_i +\lambda)})(1-\sqrt{\gamma (s_i +\lambda)})^{k}\Big]^2\\
\EEAS
And \BEA
\sum_{k=1}^n \tr P_{i,k}V P_{i,k}
&=&\frac{v_i s_i}{\gamma^2 (s_i +\lambda)^{2}}\sum_{k=1}^n  \Big[ 1- (1+k\sqrt{\gamma (s_i +\lambda)})(1-\sqrt{\gamma (s_i +\lambda)})^{k}\Big]^2 \nonumber\\
&\le & n \frac{v_i s_i}{\gamma^2 (s_i +\lambda)^{2}} \text{ using Lemma~\ref{lem:techtriv} in Appendix~\ref{app:techlem}.} \label{eq:coal_var}
\EEA

Alternative bounds for the bias and the variance term, as in Equations\eqref{eq:regbased},~\eqref{eq:disctinc_bias} may be derived as well.
Combining all these results, we are now able to state Theorem~\ref{th:acc_sgd}.

\subsection{Conclusion}\label{app:b4}

Combining results from Lemma~\ref{lem:3termsacc}, and Equations~\eqref{eq:regbased},~\eqref{eq:disctinc_bias},~\eqref{eq:disctinc_var}, with $c=\Sigma$, and using the following simple facts: 
\begin{itemize}
 \item For the least squares regression function, with $c=\Sigma$,
 $\E \langle \bar {  \Theta}_n, C \bar {  \Theta}_n\rangle= \E f(\bar \theta_n )- f ( \theta_* ) $.
 \item Under assumption~\ref{eq:a4},~\ref{eq:a5new}, we have $V \lecc \tau^{2} \Sigma$.
 \item The squared norm of a vector is the sum of its squared components on the orthonormal eigenbasis. For example $\| P_{n+1} \Theta_0\|^{2} = \sum _{i=1}^{d} \| P_{i,n+1} \Theta_0^{i}\|^{2}$. 
 \item For any regularization parameter $\lambda\in\RR_{+}$ and for any constant step-size $ \gamma (\Sigma +\lambda\idm ) \preccurlyeq \idm $, for any $\delta\in\big[\frac{1-\sqrt{\gamma \lambda}}{1+\sqrt{\gamma \lambda}},1\big]$, matrix $F$ will have only two distinct complex eigenvalues or two coalescent eigenvalues.
 \end{itemize}

\begin{proposition}
Under ($\mathcal{A}_{4,5}$), for any regularization parameter $\lambda\in\RR_{+}$ and for any constant step-size $ \gamma (\Sigma +\lambda\idm ) \preccurlyeq \idm $ we have  for any $\delta\in\big[\frac{1-\sqrt{\gamma \lambda}}{1+\sqrt{\gamma \lambda}},1\big]$, for the recursion in Eq.~(\ref{eq:accgrad}):
\BEAS
\E f(\bar \theta_n )- f ( \theta_* )  &\le&  2\lambda \| \lambda^{1/2}\Sigma^{1/2}(\Sigma+\lambda I)^{-1}(\theta_0-\theta_*)\|^{2} \\
&& + \sum _{i=1}^{d}\frac{2}{(n+1)^{2}}  \min \left\lbrace 36 \frac{{c_i}(\tilde \theta^i_0)^{2}}{\gamma (s_i+\lambda)}, 6n (1+e^{-1} )\frac{{c_i}(\tilde \theta^i_0)^{2}}{\sqrt{\gamma (s_i+\lambda)}}, n^{2}(1+e^{-1})^{2}{c_i}(\tilde \theta^i_0)^{2}  \right\rbrace\\
&& + \sum _{i=1}^{d} \frac{\gamma^2}{(n+1)^2} {v_i c_i } \min\Bigg\{ \frac{8n}{{\gamma^2 (s_i+\lambda)^{2}}}, \frac{(n+1)^3}{2{\gamma (s_i+\lambda)}}, \frac{(n+1)^5}{20} \Bigg\}.
\EEAS
This implies, using the Equation~\eqref{eq:initialpoint} for the initial point, using $c_i=\sigma_i$ and regrouping sums as traces or norms: 
\BEAS
\E f(\bar \theta_n )- f ( \theta_* )  &\le&  2\lambda \| \lambda^{1/2}\Sigma^{1/2}(\Sigma+\lambda I)^{-1}(\theta_0-\theta_*)\|^{2} \\
&+& 2  \min \left\lbrace \frac{36 \|\Sigma^{1/2} (\Sigma+\lambda I)^{-1/2} (\theta_0-\theta_*) \|^{2}}{\gamma (n+1)^{2}}, 
%\frac{6 (1+e^{-1}) \|\Sigma^{1/4 }  (\theta_0-\theta_*) \|^{2}}{\gamma^{3/2} (n+1)}, 
(1+e^{-1})^{2} \|\Sigma^{1/2} (\theta_0-\theta_*) \|^{2}  \right\rbrace\\
& + & \min\Bigg\{ \frac{8\tr(V \Sigma (\Sigma+\lambda I)^{-2})}{n+1} , n{ \gamma \tr(V \Sigma (\Sigma+\lambda I)^{-1})} \Bigg\},
\EEAS
which gives exactly Theorem~\ref{th:acc_sgd} using $V\lecc \tau^{2} \Sigma$ in the Variance term, and $\lambda^{1/2} (\Sigma+\lambda I)^{-1/2} \lecc I$ in the first term.
\end{proposition}

%% file: proofHShyp.tex
\section{Tighter bounds}\label{app:tighter}

\subsection{Simple upper-bounds}

In this section, we chow how tighter bounds naturally appear from the regularized quantities appearing in Theorems. It only relies on simple algebraic majorations, even if one has to be careful with the allowed intervals for $r,b$.

\begin{lemma}\label{lem_rkhsasstrace}
For any $\lambda\geq 0$, for any $b\in [0;1]$, if $\tr(\Sigma^{b})$ exists, we have : 
\begin{eqnarray*}
 \tr(\Sigma(\Sigma+\lambda I)^{-1}) &\le & \frac{\tr(\Sigma^{b})}{\lambda^{b}}\\
\tr(\Sigma^{-2}(\Sigma+\lambda I)^{-2}) &\le & \frac{\tr(\Sigma^{b})}{\lambda^{b}}
\end{eqnarray*}
\end{lemma}

 \begin{proof}
 As all operators can be diagonalized in a same eigenbasis with positive eigenvalues, we have, 
\begin{eqnarray*}
    \tr(\Sigma(\Sigma+\lambda I)^{-1}) &\le& \vertiii {\Sigma^{1-b } (\Sigma+\lambda I)^{-1} } \tr(\Sigma^{b}) \\
    ||| \Sigma^{1-b } (\Sigma+\lambda I)^{-1} ||| &\le& \sup_{0\le x} \frac{x^{1-b}}{(x+\lambda )}\\ 
    &\le & \sup_{0\le x} {x^{1-b}}\left(\frac{1}{\lambda} \wedge \frac{1}{x}\right) \\
& \le & \sup_{0\le x} {x^{1-b}}\left(\frac{1}{\lambda}\right)^{b}  \left(\frac{1}{x}\right)^{1-b} = \lambda^{-b}
  \end{eqnarray*}
And the calculations are exactly the same for  $ \tr(\Sigma^{-2}(\Sigma+\lambda I)^{-2}) \le  \frac{\tr(\Sigma^{b})}{\lambda^{b}}$. 
  \end{proof}
\ \\
As for the bias term, we need to bound the following quantities :

\begin{lemma}\label{lem_rkhsassbias}
For any $\lambda\geq 0$, for any $r \in [-1;1] $, we have : 
\begin{eqnarray*}
 \big\| \Sigma^{1/2}(\Sigma+\lambda \idm)^{-1} (\theta_0 - \theta_\ast)\big\|^{2 }  &\le & \lambda^{ -({1+r})}  \big\| \Sigma^{r/2}  (\theta_0 - \theta_\ast)\big\|^{2} 
  \end{eqnarray*}
  For any $\lambda\geq 0$, for any $r \in [-1;0] $, we have : 
\begin{eqnarray*}
     \big\|(\Sigma+\lambda \idm)^{-1/2} (\theta_0 - \theta_\ast)\big\|^{2 }  &\le & \lambda^{ -({1+r})}  \big\| \Sigma^{r/2}  (\theta_0 - \theta_\ast)\big\|^{2} 
  \end{eqnarray*}
  \label{lem_acc_rkhsass1}
For any $\lambda\geq 0$, for any $r \in [0;1] $, we have : 
\begin{eqnarray*} 
\|\Sigma^{1/2}(\Sigma+\lambda I)^{-1/2}(\theta_0 - \theta_\ast)\|^{2} &\le & \lambda^{ -r}  \big\| \Sigma^{r/2}  (\theta_0 - \theta_\ast)\big\|^{2} 
  \end{eqnarray*}
  (No result when $r\le 0$ because of saturation effect)
\end{lemma}
\begin{proof} Proof relies of simple following calculations: 
 \BEAS \big\| \Sigma^{1/2}(\Sigma+\lambda \idm)^{-1} (\theta_0 - \theta_\ast)\big\|  &\le & \vertiii{ \Sigma^{1/2-r/2}(\Sigma+\lambda \idm)^{-1} }\ \big\| \Sigma^{r/2}  (\theta_0 - \theta_\ast)\big\| \\
 &\le &  \left(\frac{1}{\lambda}\right)^{1-(1/2-r/2)} \big\| \Sigma^{r/2}  (\theta_0 - \theta_\ast)\big\| \\
 &\le & \lambda^{ -\frac{1+r}{2}}  \big\| \Sigma^{r/2}  (\theta_0 - \theta_\ast)\big\|  \EEAS 

\begin{eqnarray*}
\big\|(\Sigma+\lambda \idm)^{-1/2} (\theta_0 - \theta_\ast)\big\|  &\le &  \vertiii{ \Sigma^{-r/2}(\Sigma+\lambda \idm)^{-1/2}} \big\| \Sigma^{r/2}  (\theta_0 - \theta_\ast)\big\| \\
 &\le &  \left(\frac{1}{\lambda}\right)^{\frac{1+r}{2}} \big\| \Sigma^{r/2}  (\theta_0 - \theta_\ast)\big\| \\
 &\le & \lambda^{ -\frac{1+r}{2}}  \big\| \Sigma^{r/2}  (\theta_0 - \theta_\ast)\big\| 
\end{eqnarray*}

\begin{eqnarray*}
\|\Sigma^{1/2}(\Sigma+\lambda I)^{-1/2}(\theta_0 - \theta_\ast)\| &\le &  \vertiii{ \Sigma^{1/2-r/2}(\Sigma+\lambda \idm)^{-1/2}} \big\| \Sigma^{r/2}  (\theta_0 - \theta_\ast)\big\| \\
 &\le &  \left(\frac{1}{\lambda}\right)^{\frac{1-(1-r)}{2}} \big\| \Sigma^{r/2}  (\theta_0 - \theta_\ast)\big\| \\
 &\le & \lambda^{ -\frac{r}{2}}  \big\| \Sigma^{r/2}  (\theta_0 - \theta_\ast)\big\|  
\end{eqnarray*}
\end{proof}

\subsection{Theorem~\ref{th:tighter_acc} and Equation~\eqref{eq:tighter_av}}

Theorem~\ref{th:tighter_acc} and Equation~\eqref{eq:tighter_av} are directly derived from Theorem~\ref{th:av_sgd} and Theorem~\ref{th:acc_sgd}, using Lemmas~\ref{lem_rkhsasstrace} and~\ref{lem_rkhsassbias}.

To derive corollaries for the optimal $\gamma$, one has to find the $\gamma$ that balances the bias and variance term and to compute the products for  such a step size.

\subsubsection{Equation~\eqref{eq:tighter_av}}

We derive from Theorem~\ref{th:av_sgd}, when choosing $\gamma =(\lambda n)^{-1} $, and using Lemmas~\ref{lem_rkhsasstrace} and~\ref{lem_rkhsassbias}, the following bound, under assumptions of Theorem~\ref{th:av_sgd} :
\begin{eqnarray*}
\E f(\bar \theta_n )- f ( \theta_* )\le \frac{(18 +  \text{Res}(n, b, r, \gamma)) {\|\Sigma^{r/2} (\theta_0-\theta_*) \|^{2}}}{(\gamma n)^{\frac{1-r}{2}}}  +  \frac{6 \sigma^2 \tr(\Sigma^{b}) \gamma^{b}}{n^{1-b}} .
\end{eqnarray*}
Where $\text{Res}(n, b, r, \gamma): = 3  \gamma^{1+b} n^{b} \tr(\Sigma^b)$ if $-1\le r\le 0$ and $\text{Res}(n, b, r, \gamma): = 0$ if $0\le r\le 1$.
When choosing the optimal $\gamma \varpropto n^{\frac{-b+r}{b+1-r}}$, we have that $\gamma^{1+b} n^{b}=  n^{-1+\frac{1+b}{1+b-r}} = n^{\chi}$, with $\chi =\frac{-r}{1+b-r} \geq 0$ if  $r\le 0$. Thus the residual term is always vanishing for $r\le 0$ and does not exist for $r\geq 0$. 

\subsubsection{Theorem~\ref{th:tighter_acc}}

Theorem~\ref{th:tighter_acc} directly follows from Lemmas~\ref{lem_rkhsasstrace} and~\ref{lem_rkhsassbias} and the choice of $\gamma\varpropto n^{\frac{-2b+2r-1}{b+1-r}}$.

%% file: technical.tex
\section{Technical Lemmas}\label{app:techlem}

The following sequence of Lemmas appear in the proof. They are mostly independent and  rely on simple calculations.

\begin{lemma} \label{lem:increasingop}
The operator $\big[ ( \Sigma + \lambda \idm) \otimes \idm + \idm \otimes  ( \Sigma + \lambda \idm)  \big]^{-1}$ is a non-decreasing operator on $(S_n, \lecc)$
\end{lemma}

\begin{proof}
Lemma  means that for two matrices $M,N \in S_n(\RR)$ such that $M\lecc N$, then $$\big[ ( \Sigma + \lambda \idm) \otimes \idm + \idm \otimes  ( \Sigma + \lambda \idm)  \big]^{-1} M\lecc \big[ ( \Sigma + \lambda \idm) \otimes \idm + \idm \otimes  ( \Sigma + \lambda \idm)  \big]^{-1} N.$$
It is equivalent to show that for any symmetric positive matrix $A \in S_n^{+}$, $$\big[ ( \Sigma + \lambda \idm) \otimes \idm + \idm \otimes  ( \Sigma + \lambda \idm)  \big]^{-1} A \in S_n^{+} (\RR)$$

We consider a matrix $A \in S_n^{+} (\RR)$. $A$ can be decomposed as a sum of (at most) $n$ rank one matrices $A= \sum_{i=1}^{n} \omega_i \omega_i^{\top}$, with $\omega_i \in \RR^{n}$. We thus just have to prove that for some $\omega \in \RR^{n}$, $\big[ ( \Sigma + \lambda \idm) \otimes \idm + \idm \otimes  ( \Sigma + \lambda \idm)  \big]^{-1} \omega \omega^{\top} \in S_n^{+} (\RR)$.

 Let $\Sigma = \sum_{i \geqslant 0} \mu_i e_i \otimes e_i$ is the eigenvalue decomposition of $\Sigma$, then
$$
\big[ ( \Sigma + \lambda \idm) \otimes \idm + \idm \otimes  ( \Sigma + \lambda \idm)  \big]^{-1}  \omega \omega^{\top}
= \sum_{i,j \geqslant 0 } \frac{\langle \omega, e_i \rangle 
\langle \omega, e_j \rangle }{\mu_i + \mu_j + 2\lambda} e_i \otimes e_j.
$$

Thus, in the orthonormal basis of eigenvectors, this is thus Hadamard product between $$ \sum_{i,j \geqslant 0 } {\langle \omega, e_i \rangle 
\langle \omega, e_j \rangle } e_i \otimes e_j = \omega \omega^{\top}  $$ and the  matrix $ C= \left( \big( \frac{1}{\mu_i + \mu_j + 2\lambda} \big)_{i,j \geqslant 0} \right)$.  Matrix $C$ is a Cauchy matrix and is thus positive. Moreover the Hadamard product of two positive matrices is positive, which concludes the proof.
\end{proof}

Remark: surprisingly, the inverse operator $( \Sigma + \lambda \idm) \otimes \idm + \idm \otimes  ( \Sigma + \lambda \idm) $ is not non-decreasing. Indeed, $\lecc$ is not a total order on $S_n$ so we may have that an operator is non-decreasing and its inverse is not.

\begin{lemma}\label{lem:maj1} For all  $\rho\in(0,1)$ and $\omega \in[-\pi/2;\pi/2]$ and $r^{\pm}=\rho(\cos(\omega)\pm\sqrt{-1}\sin(\omega))$ we have:
 \begin{equation}
 \bigg|\frac{1-r^{+}r^{-}-\rho^{k} |A_1|}{|1-r^{+}|} \bigg| \le \min\{ 1 +\rho +e^{-1} +4 \rho^{k}, 2 +\rho  +\sqrt{5} \rho^{k+1}\}\leq 6
 \end{equation}
\end{lemma} 
 
 \begin{proof}
 We note that $\rho_i^{k}A_1$ is a real number as is is a quotient of pure complex numbers, which come from the difference between a complex and its conjugate. We first write $A_1$ as a combination of sine and cosine functions:
 \begin{eqnarray*}
 \rho_i^{k}A_1&=&\frac{(r_{i}^{-})^{k+1}(1-r_{i}^{+})^2-(r_{i}^{+})^{k+1}(1-r_{i}^{-})^2}{r_{i}^{-}-r_{i}^{+}}    \\
 &=&- \frac{(r_{i}^{-})^{k+1} - (r_{i}^{+})^{k+1} -2 r_{i}^{-}r_{i}^{+} ((r_{i}^{-})^{k} - (r_{i}^{+})^{k}  ) +(r_{i}^{-}r_{i}^{+}) (r_{i}^{-})^{k-1} - (r_{i}^{+})^{k-1}  )}{ \rho_i\sin{\omega_i}}\\
  &=& -\frac{\rho_i^{k+1 }\sin((k+1)\omega_i) - 2 \rho_i^{k+2 }\sin(k\omega_i) +\rho_i^{k+3 }\sin((k-1)\omega_i)}{\rho_i\sin{\omega_i}}.
 \end{eqnarray*}

This quantity can be simplified when $\rho\rightarrow 1$ or $\omega\rightarrow 0$. We thus modify the expression of $A_1$ to make these  dependencies clearer: 
 \begin{eqnarray}
\hspace{-3em}- A_1  &=& \frac{\sin((k+1)\omega_i) - 2 \rho_i\sin(k\omega_i) +\rho_i^{2 }\sin((k-1)\omega_i)}{\sin{\omega_i}} \nonumber \\
&=& \frac{(\cos(\omega)-\rho) (\sin(k\omega) - \rho \sin((k-1)\omega))  + \cos(k\omega)\sin(\omega) - \rho \cos((k-1)\omega)\sin(\omega)}{\sin{\omega_i}}\nonumber\\
&& \text{developing} \sin(a+b) = \sin(a)\cos(b) + \cos(a) \sin (b) \text{ and regrouping terms, }\nonumber\\
&=& \frac{(\cos(\omega)-\rho)^{2} \sin((k-1)\omega) +(\cos(\omega)-\rho)\sin(\omega)\cos((k-1)\omega) + \cos(k\omega)\sin(\omega) - \rho \cos((k-1)\omega)\sin(\omega)}{\sin{\omega_i}}\nonumber\\
&=& \frac{(\cos(\omega)-\rho)^{2} \sin((k-1)\omega)}{\sin{\omega_i}} +(\cos(\omega)-\rho)\cos((k-1)\omega) + \cos(k\omega)- \rho \cos((k-1)\omega)\nonumber\\
&& \text{ simplifying expression, then developing the cosine,}\nonumber\\
&=& \frac{(\cos(\omega)-\rho)^{2} \sin((k-1)\omega)}{\sin{\omega_i}} +2 (\cos(\omega)-\rho)\cos((k-1)\omega) + \sin(\omega)\sin((k-1)\omega) \label{eq:compl}.
% |A_1|  &\le & k+2 \quad \text{but also}\\
 %|A_1|  &\le & (1-\rho) (k-1)+4 \quad \text{thus}\\
 % \rho_i^{k} |A_1|^{2}  &\le& 2 \rho_i^{k} (1-\rho)^{2} (k-1)^{2}+ 32 \le 34 
  \end{eqnarray}
So that in that final expression all the terms behave relatively simply when $\rho\rightarrow 1$ or $\omega\rightarrow 0$. We want to upper bound: \begin{equation*}
 \bigg|\frac{1-r^{+}r^{-}-\rho^{k} |A_1|}{|1-r^{+}|} \bigg|
 \end{equation*}
 
 We thus consider separately the first and second term.   
  \begin{eqnarray*}
\frac{1-r_i^{+}r_i^{-}}{|1-r_i^{+}|} &=&  \frac{1-\rho^{2}}{|1-r_i^{+}|} \le 1+\rho \qquad \text{(exact if } \ \omega=0). \\
\EEAS
Then, using Equation~\eqref{eq:compl}: 
\BEAS
\frac{-\rho_i^{k} |A_1|}{|1-r_i^{+}|} &=& \rho^{k} \frac{\frac{(\cos(\omega)-\rho)^{2} \sin((k-1)\omega)}{\sin{\omega_i}} +2 
(\cos(\omega)-\rho)\cos((k-1)\omega) + \sin(\omega)\sin((k-1)\omega)}{\sqrt{(1-\rho \cos\omega)^{2} + \rho^{2}\sin^{2}(\omega)}}\\
\EEAS
And considering separately the three terms in the numerator, using numerous times that for any $a,b\in [0;1]$, $|a-b| \le 1-ab$:
\BEAS
\left|\rho^{k} \frac{\frac{(\cos(\omega)-\rho)^{2} \sin((k-1)\omega)}{\sin{\omega_i}}}{\sqrt{(1-\rho \cos\omega)^{2} + \rho^{2}\sin^{2}(\omega)}} \right| &\le& \rho^{k} {\frac{(\cos(\omega)-\rho) \sin((k-1)\omega)}{\sin{\omega_i}}} \\
&& \qquad \qquad \text{as } |(\cos(\omega)-\rho)| \le 1-\rho\cos(\omega),  \\
&\le & \rho^{k} {\frac{(\cos(\omega)-1) \sin((k-1)\omega)}{\sin{\omega_i}}}  + \rho^{k}  {\frac{(1-\rho) \sin((k-1)\omega)}{\sin{\omega_i}}} \\
&& \qquad \qquad \text{  writing } \cos(\omega)-\rho = \cos(\omega)-1+1-\rho  \\
&\le & \rho^{k} (1-\rho) (k-1) + \rho^{k} {\frac{(\cos(\omega)-1) \sin((k-1)\omega)}{\sin{\omega_i}}}\\
&& \qquad \qquad \text{as } |\sin((k-1)\omega) |\le |(k-1)\sin(\omega)|,  \\
&\le& \rho^{k} (1-\rho) k-(1-\rho) \rho^{k} + \rho^{k} {\frac{(\cos(\omega)-1) \sin((k-1)\omega)}{\sin{\omega_i}}}\\
&& \qquad \qquad \text{  writing } \cos(\omega)-1= 2 \sin ^{2}(\omega/2),\\
&\le& \rho^{k} (1+(1-\rho))^{k} -\rho^{k}- (1-\rho)\rho^{k} + \rho^{k} \frac{2\sin^{2}(\omega/2)}{\sin{\omega_i}}\\
&& \qquad \qquad \text{  using } 1 + (1-\rho) k\le (1+(1-\rho))^{k},\\
&\le& \rho^{k} (1+(1-\rho))^{k} -\rho^{k}- (1-\rho)\rho^{k} + \rho^{k} \tan(\omega/2)\\
&& \qquad \qquad \text{  and as }  \tan(\omega/2) \le 1 \text{ for } |\omega|\le \pi/2,\\
&\le& 1 -(1-\rho)\rho^{k} \\
&& \qquad \qquad \text{  using }\rho^{k} (1+(1-\rho))^{k}  = (1-(1-\rho)^{2})^{k} \le 1,
\EEAS
And for the second and third term:
\BEAS
2\left|\rho^{k} \frac{(\cos(\omega)-\rho)\cos((k-1)\omega)}{\sqrt{(1-\rho \cos\omega)^{2} + \rho^{2}\sin^{2}(\omega)}}\right| &\le& 2 \rho^{k}, \\
\left|\rho^{k} \frac{+ \sin(\omega)\sin((k-1)\omega)}{\sqrt{(1-\rho \cos\omega)^{2} + \rho^{2}\sin^{2}(\omega)}}\right| &\le & \rho^{k}.\\
 \end{eqnarray*}
Thus: 
 \begin{eqnarray*}
\bigg|\frac{1-r_i^{+}r_i^{-}-\rho_i^{k} |A_1|}{|1-r_i^{+}|} \bigg| &\le& 1 +\rho +1 +3 \rho^{k}
 \end{eqnarray*}

We also have 
\begin{eqnarray*}
|\rho^{k} \frac{\frac{(\cos(\omega)-\rho)^{2} \sin((k-1)\omega)}{\sin{\omega_i}}}{\sqrt{(1-\rho \cos\omega)^{2} + \rho^{2}\sin^{2}(\omega)}} | &\le& \rho^{k} {\frac{(\cos(\omega)-\rho) \sin((k-1)\omega)}{\sin{\omega_i}}}  \\
&\le& \rho^{k} (1-\rho) (k-1) + \rho^{k} {\frac{(\cos(\omega)-1) \sin((k-1)\omega)}{\sin{\omega_i}}}\\
&\le& (1-\frac{1}{k+1})^{k+1} -(1-\rho) \rho^{k} + \rho^{k} {\frac{(\cos(\omega)-1) \sin((k-1)\omega)}{\sin{\omega_i}}}\\
&\le& e^{-1} - (1-\rho)\rho^{k} + \rho^{k} \frac{\sin^{2}(\omega/2)}{\sin{\omega_i}}\\
\end{eqnarray*}
Using that \BEA k \sup_{x\in [0;1]} x^{k}(1-x) = k \frac{1}{k+1} (1-\frac{1}{k+1})^{k} = (1-\frac{1}{k+1})^{k+1} =\exp((k+1) \ln((1-\frac{1}{k+1})) \le e^{-1} \label{eq:majpol}
\EEA
leading to 
 \begin{eqnarray*}
\bigg|\frac{1-r_i^{+}r_i^{-}-\rho_i^{k} |A_1|}{|1-r_i^{+}|} \bigg| &\le& 1 +\rho +e^{-1} +4 \rho^{k}
 \end{eqnarray*}

We can also change $3\rho^{k}$ into $\sqrt{5}\rho^{k}$
 We have used that $|(\rho-\cos(\omega))|\le (1-\rho \cos(\omega))$.
 \end{proof}

 \begin{lemma} \label{lem:majcde}
For any $\rho_i \in (0;1)$, for any $\omega_i \in [-\pi/2; \pi/2]$

\BEAS
 \frac{\rho_i^{j}\sin(\omega_i (j+1)) }{\sin(\omega_i)}- \rho_i^{j+1}\frac{\sin(\omega_ij)}{\sin(\omega_i)}
&\leq&  1+e^{-1}
\EEAS
 \end{lemma}
 
 \begin{proof}
 \BEAS
 \frac{\rho_i^{j}\sin(\omega_i (j+1)) }{\sin(\omega_i)}- \rho_i^{j+1}\frac{\sin(\omega_ij)}{\sin(\omega_i)}
&=&  \rho_i^{j} \left( \frac{\sin(\omega_i (j+1)) - \rho_i \sin(\omega_ij)}{\sin(\omega_i)}\right)\\
&=&  \rho_i^{j} \left( \frac{(\cos(\omega_i) - \rho_i) \sin(\omega_ij) }{\sin(\omega_i)} +  \cos(j\omega_i)\right)\\\\
&\le &  \rho_i^{j} \left( (1-\rho_i) j +  1\right)\\
&\le&  1+e^{-1} \text{   using } \eqref{eq:majpol} 
\EEAS
 \end{proof}

  \begin{lemma}\label{lemmajvar} For all  $\rho\in(0,1)$ and $\omega \in[-\pi/2;\pi/2]$ and $r^{\pm}=\rho(\cos(\omega)\pm\sqrt{-1}\sin(\omega))$ we have:
 \begin{equation}
\bigg| \rho_i^{k}B_{1,k} \bigg| \leq  1.75
 \end{equation}
 \end{lemma}

 \begin{proof}
Once again, as the considered quantity is real, we first express it as a combination of sine and cosine functions. We then use some simple trigonometric trics to upper bound the quantity.
 \BEAS
\rho_i^{k}B_{1,k}&=&-\frac{(r_{i}^{-})^{k+1}(1-r_{i}^{+})-(r_{i}^{+})^{k+1}(1-r_{i}^{-})}{r_{i}^{+}-r_{i}^{-}}    \\
&=& -\frac{2\mathfrak{Im}\big[(r_{i}^{-})^{k+1}(1-r_{i}^{+})\big]}{\sqrt{-\Delta_i}}\text{ as it is the difference between a complex and its conjugate,} \\
&=&- \frac{\mathfrak{Im}\big[\rho_i^k e^{-({k+1})i\omega_i}(1-\rho_i\cos(\omega_i)-i\rho_i\sin(\omega_i))\big]}{\sin{\omega_i}\rho_i} \text{ developing the product,}\\
&=& \rho_i^k \frac{\cos(({k+1})\omega_i)\sin(\omega_i)\rho_i+\sin(({k+1})\omega_i)(1-\rho_i\cos(\omega_i))}{\sin{\omega_i}\rho_i}\\
&=& \rho_i^{k}\Big[ \rho_i\cos(({k+1})\omega_i)+(1-\rho_i\cos(\omega_i))\frac{\sin(({k+1})\omega_i)}{\sin{\omega_i}}\Big]\text{ and simplifying}.
\EEAS
Let's turn our interest to the second part of the quantity: \begin{eqnarray*}
\bigg| \rho_i^{k} (1-\rho_i\cos(\omega_i))\frac{\sin(({k+1})\omega_i)}{\sin{\omega_i}}\bigg| &=& \bigg| \rho_i^{k} (1- \rho_i+ \rho_i(1-\cos(\omega_i)))\frac{\sin(({k+1})\omega_i)}{\sin{\omega_i}}\bigg|  \\
&& \text{introducing an artificial } +\rho_i -\rho_i, \\
&\le&  \rho_i^{k} \bigg|  (1- \rho_i)\frac{\sin(({k+1})\omega_i)}{\sin{\omega_i}}\bigg| + \rho_i^{k} \bigg| \rho_i(1-\cos(\omega_i)) \frac{\sin(({k+1})\omega_i)}{\sin{\omega_i}}\bigg|\\ 
&& \text{by triangular inequality,} \\
&\le & \rho_i^{k} \bigg|  (1- \rho_i)(k+1)\bigg| + \rho_i^{k} \bigg| \rho_i\sin^{2}(\frac{\omega}{2}) \frac{1}{2\cos(\frac{\omega}{2})\sin(\frac{\omega}{2})}\bigg|\\ 
&& \text{using } 1-\cos(\omega_i)=2\sin^{2}(\frac{\omega}{2}) \\
&\le & \rho_i^{k}   (1- \rho_i)k  + \rho_i^{k} (1-\rho) + \rho_i^{k} \bigg| \rho_i\sin^{2}(\frac{\omega}{2}) \frac{1}{2\cos(\frac{\omega}{2})\sin(\frac{\omega}{2})}\bigg|\\ 
&\le & (1-(1-\rho_{i}))^{k} (1+ (1- \rho_i))^{k} -\rho_i^{k} + \frac{1}{2(k+1)}+ \rho_i^{k} \bigg| \frac{\rho_i}{2}\tan(\frac{\omega}{2}) \bigg|\\ 
&\le &   (1-(1-\rho_{i})^{2})^{k}   + \frac{1}{4}+ \frac{1}{{2}} \le 1+\frac{1}{4}+\frac{1}{{2}}-\rho_i^{k}\\ 
\end{eqnarray*}
 Thus  
 \BEAS      
\bigg| \rho_i^{k}B_{1,k} \bigg| &=&  \rho_i^{k} +1+\frac{1}{4}+\frac{1}{{2}}-\rho_i^{k} \le  1+\frac{1}{4}+\frac{1}{{2}} = 1.75.
\EEAS
 \end{proof}

\begin{lemma} \label{lem:techtriv}
For any $ s_i, \gamma, \lambda \in \RR_{+}^{3}$ such that $\gamma(s_i +\lambda) \le 1 $, for any $k\in \mathbb N$, we have the two following highly related identities:
 \begin{eqnarray*}
0\le   {2-\sqrt{\gamma (s_i +\lambda)}}-\big(2+(k-1) \sqrt{\gamma (s_i +\lambda)}\big) (1-\sqrt{\gamma (s_i +\lambda)})^k &\le & 2 \\
0\le 1- (1+k\sqrt{\gamma (s_i +\lambda)})(1-\sqrt{\gamma (s_i +\lambda)})^{k} &\le&  1.
\end{eqnarray*}
\end{lemma}

\begin{proof}
Proof relies on the trick, for any $\alpha\in \RR, n\in \mathbb{N}$:  $1+n\alpha \le (1+\alpha)^{n}$. For the first one: 
 \begin{eqnarray*}
&&\sqrt{\gamma (s_i +\lambda)}+\big(2+(k-1) \sqrt{\gamma (s_i +\lambda)}\big) (1-\sqrt{\gamma (s_i +\lambda)})^k = \\ &=&
  \sqrt{\gamma (s_i +\lambda)}+ (1-\sqrt{\gamma (s_i +\lambda)})^k+\big(1+(k-1) \sqrt{\gamma (s_i +\lambda)}\big) (1-\sqrt{\gamma (s_i +\lambda)})^k\\
 &\le &  \sqrt{\gamma (s_i +\lambda)}+ (1-\sqrt{\gamma (s_i +\lambda)})+\big(1+(k-1) \sqrt{\gamma (s_i +\lambda)}\big) (1-\sqrt{\gamma (s_i +\lambda)})^{k-1}\\
  &\le &  1+ (1-\gamma (s_i +\lambda))^{k-1} \le 2.
 \end{eqnarray*}
 For the second one: 
 \begin{eqnarray*}
0\le (1+k\sqrt{\gamma (s_i +\lambda)})(1-\sqrt{\gamma (s_i +\lambda)})^{k} &\le & (1-{\gamma (s_i +\lambda)})^{k}\le 1.
 \end{eqnarray*}
\end{proof}